\documentclass{article}%
\usepackage[colorlinks,bookmarks=true]{hyperref}
\usepackage{amsmath}
\usepackage{graphicx}
\usepackage{amsfonts}
\usepackage{amssymb}%
\providecommand{\U}[1]{\protect\rule{.1in}{.1in}}
\newtheorem{theorem}{Theorem}[section]
\newtheorem{mainthm}{Main Theorem}

\newtheorem{corollary}{Corollary}[section]

\newtheorem{lemma}{Lemma}[section]
\newtheorem{proposition}{Proposition}[section]
\newtheorem{remark}{Remark}[section]

\newenvironment{proof}[1][Proof]{\textbf{#1.} }{\ \rule{1em}{1em}}
\begin{document}

\title{Stability of traveling waves of nonlinear Schr\"{o}dinger equation with
nonzero condition at infinity}
\author{Zhiwu Lin\\School of Mathematics\\Georgia Institute of Technology\\Atlanta, GA 30332, USA
\and Zhengping Wang\\Department of Mathematics, School of Science\\Wuhan University of Technology\\Wuhan 430070,China
\and Chongchun Zeng\\School of Mathematics\\Georgia Institute of Technology\\Atlanta, GA 30332, USA\\and \\Chern Institute of Mathematics \\Nankai University\\Tianjin, 300071, China }
\date{}
\maketitle

\begin{abstract}
We study the stability of traveling waves of nonlinear Schr\"{o}dinger
equation with nonzero condition at infinity obtained via a constrained
variational approach. Two important physical models are Gross-Pitaevskii (GP)
equation and cubic-quintic equation. First, under a non-degeneracy condition
we prove a sharp instability criterion for $3$D traveling waves of (GP), which
had been conjectured in the physical literature. This result is also extended
for general nonlinearity and higher dimensions, including 4D (GP) and 3D
cubic-quintic equations. Second, for cubic-quintic type nonlinearity, we
construct slow traveling waves and prove their nonlinear instability in any
dimension. For dimension two, the non-degeneracy condition is also proved for
these slow traveling waves. For general traveling waves without vortices (i.e.
nonvanishing) and with general nonlinearity in any dimension, we find a sharp
condition for linear instability. Third, we prove that any 2D traveling wave
of (GP) is transversally unstable and find the sharp interval of unstable
transversal wave numbers. Near unstable traveling waves of above all cases, we
construct unstable and stable invariant manifolds.

\end{abstract}

\section{Introduction}

Consider the Gross-Pitaevskii (GP) equation
\begin{equation}
i\frac{\partial u}{\partial t}+\Delta u+(1-|u|^{2})u=0,\ (t,x)\in
\mathbb{R}\times\mathbb{R}^{3}, \label{eqn-gp}%
\end{equation}
where $u$ satisfies the boundary condition $|u|\rightarrow1$ when
$|x|\rightarrow\infty$. Equation (\ref{eqn-gp}), with the considered non-zero
conditions at infinity, arises in lots of physical problems such as
superconductivity, superfluidity in Helium II, and Bose-Einstein condensate
(e.g. \cite{berloff-roberts-review01} \cite{gp-physical03}). On a formal
level, the Gross-Pitaevskii equation is a Hamiltonian PDE. The conserved
Hamiltonian is the energy defined by
\[
E(u)=\frac{1}{2}\int_{\mathbf{R}^{3}}|\nabla u|^{2}dx+\int_{\mathbf{R}^{3}%
}\frac{1}{4}(1-|u|^{2})^{2}dx
\]
and the energy space is defined by
\[
X_{0}=\{u\in H_{loc}^{1}(\mathbf{R}^{3}):E(u)<+\infty\}.
\]
The momentum
\[
\vec{P}\left(  u\right)  =\frac{1}{2}\int_{\mathbf{R}^{3}}\left\langle i\nabla
u,u-1\right\rangle
\]
is also formally conserved, due to the translation invariance of (GP). We
denote
\begin{equation}
P\left(  u\right)  =\frac{1}{2}\int_{\mathbf{R}^{3}}\langle i\partial_{x_{1}%
}u,u-1\rangle\ dx=-\int_{\mathbf{R}^{3}}\left(  u_{1}-1\right)  \partial
_{x_{1}}u_{2}dx. \label{defn-momentum-classical}%
\end{equation}
to be the first component of $\vec{P}$. The global existence of Cauchy problem
for (GP) in the energy space $X_{0}$ was proved in \cite{gerard-cauchy}
\cite{Gerard1-energy}. Some studies on the asymptotic behavior of solution to
\eqref{eqn-gp} can be found in, for example, \cite{GNT07, GNT09}.

Traveling waves are solutions to (GP) of the form $u(t,x)=U_{c}(x-ce_{1}t)$,
where $e_{1}=(1,0,0)$ and $U_{c}$ satisfies the equation
\begin{equation}
-ic\partial_{x_{1}}U_{c}+\Delta U_{c}+(1-|U_{c}|^{2})U_{c}=0.
\label{eqn-TW-gp}%
\end{equation}
Such traveling waves of finite energy play an important role in the dynamics
of the Gross-Pitaevskii equation. In a series of papers including
\cite{jones-et-stability} \cite{jones-et-symmetry}, Jones, Putterman and
Roberts used formal expansions and numerics to construct traveling waves and
studied their properties for both $2$D and $3$D (GP). For $3$D, they found a
branch of traveling waves with the travel speed in the subsonic interval
$\left(  0,\sqrt{2}\right)  $. These traveling waves tend to a pair of vortex
rings when $c\rightarrow0\ $and to solitary waves of the
Kadomtsev-Petviashvili (KP) equation when $c\rightarrow\sqrt{2}$. Starting in
late 1990s \cite{b-s-1998}, B\'{e}thuel and Saut initiated a rigorous
mathematical study of the program of Jones, Putterman and Roberts. Since then,
there have been lots of mathematical study on this subject. We refer to the
survey \cite{B-G-S} and two recent papers (\cite{Maris-annal}
\cite{chiron-marisII-12}) on the existence and properties of traveling waves
of (GP). In particular, the existence of 3D traveling waves in the full
subsonic range $\left(  0,\sqrt{2}\right)  $ was proved in \cite{Maris-annal};
non-existence of supersonic and sonic traveling waves was shown in
\cite{Maris-non-existence}; symmetry, decay and regularity of both 2D and 3D
traveling waves were studied in \cite{Gravejat} \cite{B-G-S}. However, the
stability and dynamics of these traveling waves have not been well studied.
Recently, Chiron and Maris (\cite{chiron-marisII-12}) constructed both 2D and
3D traveling waves of (GP) by minimizing the energy under the constraint of
fixed momentum. They showed the compactness of the minimizing sequence and as
a corollary the orbital stability of these traveling waves were obtained. But,
it is not clear what the range of traveling speeds these stable traveling
waves cover. Moreover, for 3D (GP) it is known that only part of the traveling
waves branch could be constructed as energy minimizers subject to fixed momentum.

In the physical literature (\cite{jones-et-stability}
\cite{berloff-roberts-X-stability}), the following linear stability criterion
for 3D traveling waves was conjectured based on numerics and heuristic
arguments: there is linear stability of the branch of traveling waves $U_{c}$
satisfying $\frac{dP\left(  U_{c}\right)  }{dc}>0$, commonly referred as the
lower branch,
and linear instability on the branch with $\frac{dP\left(  U_{c}\right)  }%
{dc}<0$, commonly referred as the upper branch. More specifically, numerical
evidences (\cite{jones-et-stability} \cite{berloff-roberts-X-stability})
suggested that there exists $c^{\ast}\in\left(  0,\sqrt{2}\right)  $ such that
$\frac{dP\left(  U_{c}\right)  }{dc}>0$ for $c\in\left(  0,c^{\ast}\right)  $
and $\frac{dP\left(  U_{c}\right)  }{dc}<0$ for $c\in\left(  c^{\ast},\sqrt
{2}\right)  $. Here, our definition of $P\left(  u\right)  \ $follows the
notation in \cite{Maris-annal} and differs with that of
(\cite{jones-et-stability} \cite{berloff-roberts-X-stability}) by a negative
sign. In this paper, we rigorously justify this stability criterion, under a
non-degeneracy condition (\ref{assumption-NDG}) or its cylindrical symmetric
version (\ref{assumption-NDG-s}). Roughly, we showed the following main
theorem for \eqref{eqn-GP-generalized}, a more general (than \eqref{eqn-gp})
nonlinear Schr\"odinger equation with non-vanishing condition at infinity.

\begin{mainthm}
Let $0<c_{0}<\sqrt{2}$ and $U_{c_{0}}$ be a traveling wave solution of
(\ref{eqn-GP-generalized}) radial in $(x_{2}, x_{3})$ directions constructed
in \cite{Maris-annal}.

\begin{itemize}
\item Suppose the nonlinearity $F$ in \eqref{eqn-GP-generalized} satisfies
(F1-2) and a non-degeneracy condition \eqref{assumption-NDG}
holds. Then for $c$ in a neighborhood of $c_{0}$, there exists a locally
unique $C^{1}$ family of traveling waves $U_{c}$. If, in addition, $U_{c}$
satisfies $\frac{\partial P(U_{c})}{\partial c}|_{c=c_{0}}>0\ $, then the
traveling wave $U_{c_{0}}$ is orbitally stable in the energy space $X_{0}$.

\item Suppose $F \in C^{5}$ and a cylindrical version of non-degeneracy
condition \eqref{assumption-NDG-s} holds. Then for $c$ in a neighborhood of
$c_{0}$, there exists a $C^{1}$ family of traveling waves $U_{c}$ locally
unique in cylindrically symmetric function spaces. If, in addition, $U_{c}$
satisfies $\frac{\partial P(U_{c})}{\partial c}|_{c=c_{0}}<0\ $, the
linearized equation at $U_{c_{0}}$ has an unstable eigenvalue and locally
$U_{c_{0}}$ has a 1-dim $C^{2}$ unstable manifold and a 1-dim $C^{2}$ stable
manifold, which yields the nonlinear instability.
\end{itemize}
\end{mainthm}

Here assumptions (F1-2) are given in Subsection \ref{SS:GNLS}.
The existence of the local $C^{1}$ family of traveling waves are due to the
Implicit Function Theorem based on the non-degeneracy assumption, see Theorem
\ref{thm-continuation-3d}. We refer to \textbf{Theorems \ref{thm:stability}},
\textbf{\ref{thm:general-stability}}, \textbf{\ref{thm:invariant manifold}}
and \textbf{Corollary \ref{C:critical}} for more precise statements on the
stability/instability, where the exact meaning of the orbital stability is
also given. In fact we do not have to limit ourselves to those traveling waves
constructed in \cite{Maris-annal}. The main properties on $U_{c}$ we really
need is that, as critical points of the energy-momentum functional $E_{c}
\triangleq E+cP$, the Hessian $E_{c}^{\prime\prime}$ of $E_{c}$ at $U_{c}$ has
exactly one negative direction, in addition to the non-degeneracy
\eqref{assumption-NDG} of $E_{c}^{\prime\prime}$.

Condition (\ref{assumption-NDG}) states that the kernel of the
the Hessian $E_{c}^{\prime\prime}$ of the energy-momentum functional $E_{c}$
is spanned by the translation modes $\left\{  \partial_{x_{i}}U_{c}\right\}  $
only. Equivalently, the linearization of the (elliptic) traveling wave
equation
has only solutions of translation modes. Such condition is commonly assumed in
the stability analysis of dynamical systems (e.g. \cite{gss87} \cite{gss90}).
It is a nontrivial task to confirm the non-degeneracy condition for a given
traveling wave associated to a specific nonlinearity, which involves mainly
the analysis of the linearized elliptic equation of traveling waves. In
Appendix 2, we verify such kind of condition in certain cases.

\begin{remark}
Assume $U_{c}$ is a family of traveling waves $C^{1}$ in $c$.
If the stability sign condition $\frac{\partial P(U_{c})}{\partial
c}|_{c=c_{0}}\ge0\ $ is satisfied, actually we can still obtain the spectral
stability of $U_{c_{0}}$ even if the non-degeneracy condition
(\ref{assumption-NDG}) is not satisfied. Here the spectral stability means
that the spectrum of the linearized equation at $U_{c_{0}}$ is contained in
the imaginary axis on the complex plane. This is a consequence of the results
in a more general setting in a forthcoming paper \cite{LZ15}. However, the
linear stability is not guaranteed as linear solutions may grow like $O(t)$.

\end{remark}

We give a brief description of key ideas in the proof. The troubles from the
non-zero condition at infinity can be seen from the linearized operator, which
is of the form $JL_{c}$, where $J=$ $\left(
\begin{array}
[c]{cc}%
0 & 1\\
-1 & 0
\end{array}
\right)  \ $and $L_{c}$ (defined by (\ref{operator-Lc})) is the second
variation operator of the Hamiltonian $E+cP$. When $\left\vert x\right\vert
\rightarrow\infty$, the operator $L_{c}$ has the asymptotic form
\[
\left(
\begin{array}
[c]{cc}%
-\Delta+2 & -c\partial_{x_{1}}\\
c\partial_{x_{1}} & -\Delta
\end{array}
\right)  ,
\]
which implies that the essential spectrum of $L_{c}\ $is $[0,+\infty)$ for any
$c\in\left(  0,\sqrt{2}\right)  $. Therefore, there is no spectral gap for
$L_{c}$ between the discrete spectrum (negative and zero eigenvalues) and the
rest of the spectrum. So we cannot use the standard stability theory for
Hamiltonian PDEs as in \cite{gss87} \cite{gss90}, which requires such a
spectral gap condition. To overcome this issue, we observe that the quadratic
form $\left\langle L_{c}\cdot,\cdot\right\rangle $ has the right spectral
structure in the space $X_{1}=H^{1}\left(  \mathbf{R}^{3}\right)  \times
\dot{H}^{1}\left(  \mathbf{R}^{3}\right)  $. More precisely, the quadratic
form of $L_{c}$ is uniformly positive definite modulo a finite dimensional
negative and zero modes. However, another issue arises since the operator
$J^{-1}=-J$ does not map $X_{1}$ to its dual $\left(  X_{1}\right)  ^{\ast
}=H^{-1}\left(  \mathbf{R}^{3}\right)  \times\dot{H}^{-1}\left(
\mathbf{R}^{3}\right)  $. The boundedness of $J^{-1}:X_{1}\rightarrow\left(
X_{1}\right)  ^{\ast}$ is required in \cite{gss87} \cite{gss90} and is true
for Schr\"{o}dinger equation with vanishing condition where $X_{1}=H^{1}\times
H^{1}$. We use a new argument to avoid using the boundedness of $J^{-1}$ and
prove the linear instability criterion $\frac{dP\left(  U_{c}\right)  }{dc}<0$
(Proposition \ref{prop-linear-insta}) under the non-degeneracy condition
(\ref{assumption-NDG-s}).

To study the nonlinear dynamics, we use a coordinate system of the (non-flat)
energy space $X_{0}$ over the Hilbert space $X_{1}$. More precisely, there
exists a bi-continuous mapping $\psi:X_{1}\rightarrow X_{0}$ as defined in
(\ref{coordinate-mapping}), which was first introduced in
\cite{Gerard1-energy} to understand the structure of the energy space $X_{0}$.
The nonlinear stability on the lower branch with $\frac{dP\left(
U_{c}\right)  }{dc}>0\ $is proved by the Taylor expansions of Hamiltonian
functional $\left(  E+c\tilde{P}\right)  \left(  \psi\left(  w\right)
\right)  $ for $w\in X_{1}$ near $w_{c}$, where $U_{c}=\psi\left(
w_{c}\right)  $ and $\tilde{P}\left(  u\right)  $ is the extended momentum
(defined in (\ref{defn-momentum-new})) in the energy space $X_{0}$. The proof
of stability (Theorem \ref{thm:stability}) implies that the stable traveling
waves are local energy minimizers with a fixed momentum.

To study the nonlinear dynamics near the linearly unstable traveling waves on
the upper branch, we rewrite the (GP) equation in terms of the coordinate
function $w\in X_{1},\ $where $u=\psi\left(  w\right)  $ satisfies the (GP)
equation (\ref{eqn-gp}). We construct stable (unstable) manifolds near
unstable traveling waves by this new equation for $w\in X_{3}=H^{3}\times
\dot{H}^{3}$, on which the nonlinear term of the $w-$equation is shown to be
semilinear in Appendix 1. The linearized operator for $w$ is similar to the
operator $JL_{c}$, that is, of the form $K^{-1}JL_{c}K$, where $K$ is an
isomorphism of $X_{1}$ defined in (\ref{definition-K}). Thus the study of the
linearized $w$ equation is reduced to the study of the semigroup $e^{tJL_{c}}%
$. To show the existence of unstable (stable) manifolds, first we establish an
exponential dichotomy estimate for $e^{tJL_{c}}$ in $X_{3}$. That is, to
decompose $X_{3}$ into the direct sum of two invariant subspaces, on one the
linearized solutions have an exponential growth and on the other one have
strictly slower growth. It is highly nontrivial to get such exponential
dichotomy for $e^{tJL_{c}}\ $from the spectra of $JL_{c}$ due to the issue of
spectral mapping (see Remark \ref{rmk-dichotomy-proof}). In this paper, we
develop a new approach to prove the exponential dichotomy of $e^{tJL_{c}}$,
which might be useful for very general Hamiltonian PDEs. The idea is very
simple and natural. We observe that the quadratic form of $\left\langle
L_{c}u\left(  t\right)  ,v\left(  t\right)  \right\rangle $ is invariant for
any two linearized solutions $u\left(  t\right)  $ and $v\left(  t\right)  $.
It implies that the orthogonal complement (in the inner product $\left\langle
L_{c}\cdot,\cdot\right\rangle $) to the unstable and stable modes defines a
subspace invariant under the linearized flow $e^{tJL_{c}}$. The quadratic form
$\left\langle L_{c}\cdot,\cdot\right\rangle $ restricted to above defined
space is shown to be positive definite modulo the translation modes. By using
this positivity estimate and the invariance of $\left\langle L_{c}\cdot
,\cdot\right\rangle $ under $e^{tJL_{c}}$, the solutions on this subspace are
shown to have at most polynomial growth. Therefore it serves as the invariant
center subspace of the linearized flow and the exponential trichotomy of the
linearized flow between the stable, unstable, and the center subspaces is
established. Consequently, the existence of unstable (stable) manifolds
follows from the standard invariant manifold theory for semilinear equations
(e.g. \cite{bates-jones88, CL88}). In a future work in preparation, we will
construct center manifolds near the orbital neighborhood of the unstable
traveling waves in the energy space $X_{0}$. The positivity of $L_{c}$ on the
center space (modulo the translation modes) then implies the orbital stability
and local uniqueness of the center manifold.

The above study of stability of traveling waves can be generalized to
nonlinear Schr\"{o}dinger equation with general nonlinear terms or in higher
dimensions. Consider
\begin{equation}
i\frac{\partial u}{\partial t}+\Delta u+F(|u|^{2})u=0,
\label{eqn-GP-generalized}%
\end{equation}
where $(t,x)\in\mathbf{R}\times\mathbf{R}^{n}$ $\left(  n\geq3\right)  \ $and
$u$ satisfies the boundary condition $|u|\rightarrow1$ as $|x|\rightarrow
\infty$. Assume that the nonlinear term $F\left(  u\right)  $ satisfies the
assumptions (F1)-(F2) or (F1)-(F2') in Section 2.5 for $n=3$ and in Section 6
for $n\geq4$. These include the 4D (GP) and 3D cubic-quintic equations, which
have the critical nonlinearity. Then the sharp linear instability criterion
$\frac{d}{dc}P\left(  U_{c}\right)  <0$ can be proved in the same way (see
\textbf{Theorems \ref{thm:general-stability}},
\textbf{\ref{thm:invariant manifold}} and \textbf{Corollary \ref{C:critical}%
}), by studying the quadratic form $\left\langle L_{c}\cdot,\cdot\right\rangle
$ in the same space $X_{1}=H^{1}\left(  \mathbf{R}^{n}\right)  \times\dot
{H}^{1}\left(  \mathbf{R}^{n}\right)  $ for $n\geq4$. The unstable (stable)
manifolds can then be constructed near unstable traveling waves by using the
equation (\ref{eqn-GP-generalized}). To prove orbital stability when $\frac
{d}{dc}P\left(  U_{c}\right)  >0$\ for dimensions $n\geq4$, a coordinate
mapping $\psi\ $relating $X_{1}$ and the energy space $X_{0}\ $is required.
For $n=4,~$such a mapping is simply given by $\psi\left(  w\right)
=1+w,\ w\in X_{1}$ and the global existence of $4$D (GP) was recently shown in
\cite{killip-et}. We refer to Section 6 for more details on the extensions.

The above approach does not work for dimensions $n=1,2$. First, the quadratic
form $\left\langle L_{c}\cdot,\cdot\right\rangle $ is not well-defined in the
space $X_{1}=H^{1}\left(  \mathbf{R}^{n}\right)  \times\dot{H}^{1}\left(
\mathbf{R}^{n}\right)  $ for $n=1,2$. Second, the energy space $X_{0}$ cannot
be written as a metric space homeomorphism to $X_{1}$, due to the oscillations
of functions in $X_{0}$ at infinity (see \cite{Gerard1-energy}). However, when
the traveling wave $U_{c}\ $has no vortices, that is, $U_{c}\neq0$, we can
study the linear instability of $U_{c}$ by the following hydrodynamic
formulation. By the Madelung transformation $u=\sqrt{\rho}e^{i\theta}$, the
equation (\ref{eqn-GP-generalized}) becomes
\begin{equation}%
\begin{cases}
\theta_{t}+|\nabla\theta|^{2}-\frac{1}{2}\frac{1}{\rho}\Delta\rho+\frac{1}%
{4}\frac{1}{\rho^{2}}|\nabla\rho|^{2}-F(\rho)=0\\
\rho_{t}+2\nabla.(\rho\nabla\theta)=0
\end{cases}
. \label{hrdrodynamic nls}%
\end{equation}
Define the velocity $\vec{v}=\nabla\theta$. Then the first equation of
\eqref{hrdrodynamic nls}
is the Bernoulli equation for the vector potential $\theta$ and the second
equation is the continuity equation for the density $\rho$. Define the energy
functional
\begin{equation}
E(\rho,\theta)=\frac{1}{2}\int_{\mathbf{R}^{n}}\left(  \frac{\left\vert
\nabla\rho\right\vert ^{2}}{4\rho}+\rho|\nabla\theta|^{2}+V(\rho)\right)
dx\ . \label{energy-hrdrodynamic}%
\end{equation}
The equation
\eqref {hrdrodynamic nls} is also formally Hamiltonian as%
\[
\partial_{t}\left(
\begin{array}
[c]{c}%
\rho\\
\theta
\end{array}
\right)  =JE^{\prime}(\rho,\theta).
\]
Linearizing above equation at the traveling wave $\left(  \rho_{c},\theta
_{c}\right)  $, we get
\begin{equation}
\partial_{t}\left(
\begin{array}
[c]{c}%
\rho\\
\theta
\end{array}
\right)  =JM_{c}\left(
\begin{array}
[c]{c}%
\rho\\
\theta
\end{array}
\right)  , \label{eqn-linearized-madelung}%
\end{equation}
where $J$ is as before and $M_{c}$ is defined in (\ref{defn-M}). We show that
for any dimension $n\geq1$, the quadratic form $\left\langle M_{c}\cdot
,\cdot\right\rangle \ $\ has the right spectral structure for $(\rho
,\theta)\in H^{1}\left(  \mathbf{R}^{n}\right)  \times\dot{H}^{1}\left(
\mathbf{R}^{n}\right)  $, that is, it is positive definite modulo a
one-dimensional negative mode and translation modes. Then by the same proof as
in the $3$D (GP) case, the linear instability criterion $\frac{d}{dc}P\left(
U_{c}\right)  <0\ $is obtained under the non-degeneracy assumption (see
\textbf{Proposition \ref{P-Linear-Insta-nV}}). As an example, we consider the
cubic-quintic equation with $F(s)=-\alpha_{1}+\alpha_{3}s-\alpha_{5}s^{2}$,
where $\alpha_{1},\alpha_{3},\alpha_{5}$ are positive constants satisfying%
\begin{equation}
\frac{3}{16}<\!\!\alpha_{1}\alpha_{5}/\alpha_{3}^{2}\!\!<\frac{1}{4}.
\label{condition-3-5}%
\end{equation}
This equation has many interpretations in physics. For example, in the context
of a Boson gas, it describes two-body attractive and three-body repulsive
interactions (\cite{ba89}, \cite{ba93}). Different from the (GP) equation, the
cubic-quintic type equations have unstable stationary solutions for any
dimension $n\geq1$ (\cite{ba89} \cite{ba93} \cite{de95}). First, by using the
hydrodynamic formulation we show the existence of traveling waves for
cubic-quintic type equations with small traveling speeds (\textbf{Theorem
\ref{thm-existence-slow}}) in any dimension $n\geq2$. This gives a simplified
proof of the previous results on the existence of slow traveling waves in the
work of Maris \cite{maris-4d-slow} for $n\geq4$ and in an unpublished
manuscript of Lin \cite{lin-note-99} for $n=2,3$. Moreover, our proof implies
the local uniqueness and differentiability of the traveling wave branch. For
$n=2,$ we are also able to show that the non-degeneracy condition
(\ref{assumption-NDG-s}) is satisfied for these slow traveling waves (see
Appendix 2 and Proposition \ref{prop-quadratic-madelung}). Then, we show that
the slow traveling waves are linearly unstable (\textbf{Theorem
\ref{thm-instability-madelung}}). This follows from the computation of the
sign of $\frac{dP\left(  U_{c}\right)  }{dc}|_{c=0}$ for stationary solutions.
To construct unstable (stable) manifolds, it is not convenient to use the
hydrodynamic formulation (\ref{hrdrodynamic nls}) which has the loss of
derivative in the nonlinear terms. Our strategy is to construct unstable
(stable) manifolds by the original equation (\ref{eqn-GP-generalized}), based
on the linear exponential dichotomy in $\left(  H^{k}\left(  \mathbf{R}%
^{n}\right)  \right)  ^{2}\ $which is first obtained in the $(\rho,\theta)$
coordinates. This is possible due to the observation that the unstable
(stable) eigenfunctions do have the $L^{2}$ estimate for $\theta$.

Lastly, we show that any 2D traveling wave of (GP) is transversely unstable.
In \textbf{Theorems \ref{thm-linear-transveral-instability}} and
\textbf{\ref{thm-transversal-mfld-2d}}, we find the sharp range of transverse
wave numbers for linear instability, and construct unstable and stable
manifolds under 3D perturbations. For the proof, we observe that the
linearized problem with transversal wave number $k$ is reduced to the study of
the spectrum of the operator $J\left(  L_{c}+k^{2}\right)  $, where $L_{c}$ is
defined by (\ref{operator-Lc-g}) for $2$D traveling waves. For $k>0$, the
spectrum of $L_{c}+k^{2}$ has the gap structure in the usual space $\left(
H^{m}\left(  \mathbf{R}^{2}\right)  \right)  ^{2}$ and thus the proof of
linear instability follows by that of Proposition \ref{prop-linear-insta} in a
much simpler version. In the physical literature (\cite{kuznetsov-rasmussen95}
\cite{berloff-roberts-crow}), the transversal instability of 2D traveling
waves of (GP) was studied by the asymptotic expansions and numerics, in the
long wavelength (or small wave number) limit.

This paper is organized as follows. In Section 2, we study the spectral
structures of the second variation of energy-momentum functional and then
prove the orbital stability on the lower branch of $3$D traveling waves of
(GP). In Section 3, we prove linear instability of 3D traveling waves on the
upper branch and then construct unstable (stable) manifolds. Section 4 is to
show the transversal instability of 2D traveling waves of (GP). In Section 5,
we construct slow traveling waves of cubic-quintic type equations and then
prove their instability. Section 6 extends the main results to other
dimensions and more general nonlinear terms. In the appendix, we give the
proof of several technical Lemmas.

We list some notations and function spaces used in the paper. For any integer
$k\geq1$, $n\geq1$, denote the space
\[
\dot{H}^{k}\left(  \mathbf{R}^{n}\right)  =\left\{  u\ |~\nabla u,\cdots
,\nabla^{k}u\in L^{2}\left(  \mathbf{R}^{n}\right)  \right\}  ,
\]
with the norm $\left\Vert u\right\Vert _{\dot{H}^{k}}=\sum_{j=1}^{k}\left\Vert
\nabla^{j}u\right\Vert _{L^{2}}$. For $n\geq3,$ by Sobolev embedding we can
impose the condition $u\in L^{\frac{2n}{n-2}}\left(  \mathbf{R}^{n}\right)  $
when $u\in\dot{H}^{1}\left(  \mathbf{R}^{n}\right)  $. Let $H_{\mathbf{R}}%
^{k}(\mathbf{R}^{n})$ $\left(  \dot{H}_{\mathbf{R}}^{k}\left(  \mathbf{R}%
^{n}\right)  \right)  \ $be all the real valued functions in $H^{k}%
(\mathbf{R}^{n})\ \left(  \dot{H}^{k}\left(  \mathbf{R}^{n}\right)  \right)
$. Let $X_{k}\left(  \mathbf{R}^{n}\right)  =H_{\mathbf{R}}^{k}(\mathbf{R}%
^{n})\times\dot{H}_{\mathbf{R}}^{k}(\mathbf{R}^{n})$ be equipped with the
norm
\[
\Vert w\Vert_{X_{k}}=\Vert w_{1}\Vert_{H^{k}}+\Vert w_{2}\Vert_{\dot{H}^{k}%
},\ w=\left(  w_{1},w_{2}\right)  \in X_{k}.
\]
In case of no confusion, we write $X_{k}\left(  \mathbf{R}^{n}\right)
,\ H_{\mathbf{R}}^{k}(\mathbf{R}^{n})$ $\left(  \dot{H}_{\mathbf{R}}%
^{k}\left(  \mathbf{R}^{n}\right)  \right)  \ $simply as $X_{k}\mathbf{,\ }%
H^{k}\ \left(  \dot{H}^{k}\right)  $. For $n\geq2,\ $denote $L_{r_{\bot}}%
^{2},\dot{H}_{r_{\bot}}^{-1}$ and $H_{r_{\bot}}^{k}\left(  \dot{H}_{r_{\bot}%
}^{k}\right)  $ to be the cylindrically symmetric subspaces of $L^{2},\dot
{H}^{-1}\left(  \text{the dual of }\dot{H}^{1}\right)  $ and $H^{k}\left(
\dot{H}^{k}\right)  $. A function $u$ is cylindrically symmetric if
$u=u\left(  x_{1},r_{\bot}\right)  \ $with $x_{\bot}=\left(  x_{2}%
,\cdots,x_{n}\right)  ,\ r_{\bot}=|x_{\bot}|$. Denote $X_{k}^{s}$ to be the
cylindrically symmetric subspaces of $X_{k}$, that is, $X_{k}^{s}=$
$H_{r_{\bot}}^{k}\times\dot{H}_{r_{\bot}}^{k}$.

\section{Orbital stability of lower branch traveling waves}

In this section, we prove nonlinear orbital stability for $3$D traveling waves
obtained via a constrained variational approach on the lower branch with
$\frac{d}{dc}P\left(  U_{c}\right)  >0$. The proof is to expand the
energy-momentum functional near the traveling wave and show that the second
variation is positive definite and dominant. A corollary of this proof is that
the stable traveling waves are local energy minimizers with fixed momentum. We
give the detailed proof for (GP) and then discuss briefly the extensions for
general nonlinearity.

The energy functional of (GP)
\[
E\left(  u\right)  =\frac{1}{2}\int_{\mathbf{R}^{3}}\left[  \left\vert \nabla
u\right\vert ^{2}+\frac{1}{2}\left(  \left\vert u\right\vert ^{2}-1\right)
^{2}\right]  dx
\]
is defined on the energy space
\begin{equation}
X_{0}=\left\{  u\in H_{\text{loc}}^{1}\left(  \mathbf{R}^{3};\mathbf{C}%
\right)  \ |\ \nabla u\in L^{2}\left(  \mathbf{R}^{3}\right)  ,\left\vert
u\right\vert ^{2}-1\in L^{2}\left(  \mathbf{R}^{3}\right)  \right\}  .
\label{defn-energy-space}%
\end{equation}
By \cite{Gerard1-energy}, for any $u\in X_{0}$, we can write $u=c\left(
1+v\right)  $ where $c\in\mathbb{S}^{1}$ and $v\in\dot{H}^{1}\left(
\mathbf{R}^{3}\right)  $. Given $u=c\left(  1+v\right)  $ and $\tilde
{u}=\tilde{c}\left(  1+\tilde{v}\right)  $, we define the natural distance in
$X_{0}\ $by
\begin{equation}
d_{1}\left(  u,\tilde{u}\right)  =\left\vert c-\tilde{c}\right\vert
+\left\Vert \nabla v-\nabla\tilde{v}\right\Vert _{L^{2}}+\left\Vert \left\vert
v\right\vert ^{2}+2\operatorname{Re}v-\left\vert \tilde{v}\right\vert
^{2}-2\operatorname{Re}\tilde{v}\right\Vert _{L^{2}} . \label{distance-energy}%
\end{equation}

The global well-posedness of (GP) equation on $X_{0}$ was proved in
\cite{gerard-cauchy}. Moreover, if $u\left(  t\right)  $ is the solution of
(GP) with $u\left(  0\right)  =c\left(  1+v_{0}\right)  $ where $c\in
\mathbb{S}^{1}$ and $v_{0}\in\dot{H}^{1}\left(  \mathbf{R}^{3}\right)  $, then
$u\left(  t\right)  =c\left(  1+v\left(  t\right)  \right)  $ with $v\left(
t\right)  \in\dot{H}^{1}\left(  \mathbf{R}^{3}\right)  $. Thus, for stability
considerations we only need to consider $c=1$ which will be assumed for the
rest of this paper.

\subsection{Momentum}

Besides the energy, another invariant of (GP) is the momentum which is due to
the translation invariance in $x_{1}\ $of the equation. For
\[
u=u_{1}+iu_{2}\in1+H^{1}\left(  \mathbf{R}^{3}\right)  \subset X_{0},
\]
the momentum is defined by (\ref{defn-momentum-classical}). However, that form
of momentum is not defined for an arbitrary function $u$ in the energy space
$X_{0}$. So, first we need to extend the definition of $P$ to all functions in
$X_{0}$. For this and the proof of main Theorem, we use the following manifold
structure of $X_{0}$ given in \cite{Gerard1-energy}. Let $\chi\in
C_{0}^{\infty}(\mathbf{R}^{3},[0,1])$ be a real valued and radial function
such that $\chi(\xi)=1$ near $\xi=0$, and consider the the Fourier multiplier
$\chi(D)$ defined on $\mathcal{S}^{\prime}(\mathbf{R}^{3})$ by
\[
(\widehat{\chi(D)u})(\xi)=\chi(\xi)\hat{u}(\xi).
\]
Define
\begin{equation}
\psi(w)=1+w_{1}-\chi(D)(\frac{w_{2}^{2}}{2})+iw_{2},\ \mathrm{for}\ w=\left(
w_{1},w_{2}\right)  \in X_{1}. \label{coordinate-mapping}%
\end{equation}
By Proposition 1.3 in \cite{Gerard1-energy}, the mapping $w\rightarrow\psi(w)$
is locally bi-Lipschitz between $X_{1}$ and $\left(  X_{0},d_{1}\right)  $. So
the space $X_{0}$ can be considered as a manifold over the coordinate space
$X_{1}$. For any $u=\psi(w)\in X_{0}$ with $w\in X_{1}$, we define the
momentum by
\begin{equation}
\tilde{P}\left(  u\right)  =-\int_{\mathbf{R}^{3}}[w_{1}+(1-\chi
(D))(\frac{w_{2}^{2}}{2})]\partial_{x_{1}}w_{2}dx. \label{defn-momentum-new}%
\end{equation}
By Proposition 1.3 of \cite{Gerard1-energy}, we have
\begin{equation}
\Vert\chi(D)f\Vert_{L^{p}\cap L^{\infty}}\leq C\Vert f\Vert_{L^{p}}%
,\ \forall\ f\in L^{p},1\leq p\leq\infty, \label{ineq-1}%
\end{equation}%
\begin{equation}
\Vert(1-\chi(D))(fg)\Vert_{L^{2}}\leq C\Vert\nabla f\Vert_{L^{2}}\Vert\nabla
g\Vert_{L^{2}},\ \forall\ f,g\in\dot{H}^{1}(\mathbf{R}^{3}). \label{ineq-2}%
\end{equation}
So the right hand side of (\ref{defn-momentum-new}) is well-defined. First, we
show that when $u\in1+H^{1}\left(  \mathbf{R}^{3}\right)  ,$ $\tilde{P}\left(
u\right)  =P\left(  u\right)  $, that is, $\tilde{P}$ is an extension of $P$.
Indeed, when $u\in1+H^{1}\left(  \mathbf{R}^{3}\right)  $, or equivalently,
$u=\psi(w)$ with $w\in H^{1}\left(  \mathbf{R}^{3}\right)  $, we have
\begin{align*}
&  P\left(  u\right)  =\frac{1}{2}\int_{\mathbf{R}^{3}}\langle i\partial
_{x_{1}}\psi(w),\psi(w)-1\rangle dx\\
&  =-\int_{\mathbf{R}^{3}}(w_{1}-\chi(D)(\frac{w_{2}^{2}}{2}))\partial_{x_{1}%
}w_{2}dx\\
&  =-\int_{\mathbf{R}^{3}}[w_{1}+(1-\chi(D))(\frac{w_{2}^{2}}{2}%
)]\partial_{x_{1}}w_{2}dx+\frac{1}{2}\int_{\mathbf{R}^{3}}w_{2}^{2}%
\partial_{x_{1}}w_{2}dx.\\
&  =\tilde{P}\left(  u\right)  +\frac{1}{6}\int_{\mathbf{R}^{3}}%
\partial_{x_{1}}w_{2}^{3}dx.
\end{align*}
Since $w_{2}^{3}\in L^{2}$ and $\partial_{x_{1}}w_{2}^{3}=3w_{2}^{2}%
\partial_{x_{1}}w_{2}\in L^{1}$, thus $\tilde{P}\left(  u\right)  =P\left(
u\right)  $ by the following lemma.

\begin{lemma}
Let $\mathcal{X}=\{\partial_{x_{1}}\phi\ |\ \phi\in L^{2}(\mathbf{R}^{3})\}$.
If $v\in L^{1}(\mathbf{R}^{3})\cap\mathcal{X}$, then $\int_{\mathbf{R}^{3}%
}v(x)dx=0$.
\end{lemma}

\begin{proof}
The proof is similar to that of Lemma 2.3 in \cite{Maris-annal}, so we skip it.
\end{proof}

We collect the main properties of $\tilde{P}\left(  u\right)  $.

\begin{lemma}
\label{lemma-momentum}(i) The functional $\bar{P}\left(  w\right)  :=\tilde
{P}\circ\psi(w)$ is $C^{\infty}$ for $w\in X_{1}.$

(ii) $\tilde{P}\left(  u\right)  $ is the unique continuous extension of
$P\left(  u\right)  $ from $1+H^{1}\left(  \mathbf{R}^{3}\right)  $
to$\ \left(  X_{0},d_{1}\right)  .$

(iii) When $u\left(  t\right)  $ is the solution of (GP) with $u\left(
0\right)  \in X_{0}$, $\tilde{P}\left(  u\left(  t\right)  \right)  =$
$\tilde{P}\left(  u\left(  0\right)  \right)  .$
\end{lemma}

\begin{proof}
(i) Since
\begin{align*}
\tilde{P}\circ\psi(w)  &  =-\int_{\mathbf{R}^{3}}w_{1}\partial_{x_{1}}%
w_{2}dx-\frac{1}{2}\int_{\mathbf{R}^{3}}(1-\chi(D))w_{2}^{2}\partial_{x_{1}%
}w_{2}dx\\
&  =B_{1}\left(  w_{1},w_{2}\right)  +B_{2}\left(  w_{2},w_{2},w_{2}\right)  ,
\end{align*}
where
\[
B_{1}\left(  w_{1},w_{2}\right)  =-\int_{\mathbf{R}^{3}}w_{1}\partial_{x_{1}%
}w_{2}dx:H^{1}\times\dot{H}^{1}\rightarrow R
\]
and
\[
B_{2}\left(  w_{1},w_{2},w_{3}\right)  =-\frac{1}{2}\int_{\mathbf{R}^{3}%
}(1-\chi(D))\left(  w_{1}w_{2}\right)  \partial_{x_{1}}w_{3}dx:\left(  \dot
{H}^{1}\right)  ^{3}\rightarrow R
\]
are multi-linear and bounded, so it follows that $\tilde{P}\circ\psi(w)$ is
$C^{\infty}$ on $X_{1}$.

(ii) follows from (i), the bicontinuity of $\psi:\left(  X_{1},\left\Vert
{}\right\Vert _{X_{1}}\right)  \rightarrow\left(  X_{0},d_{1}\right)  $, and
the density of $1+H^{1}\left(  \mathbf{R}^{3}\right)  $ in $\left(  X_{0},
d_{1}\right)  $.

(iii): When $u\left(  0\right)  \in1+H^{1}\left(  \mathbf{R}^{3}\right)  $, we
have $u\left(  t\right)  \in1+H^{1}\left(  \mathbf{R}^{3}\right)  $. The
global existence in this case was first proved in \cite{b-s-1998}. It is
straightforward to show that $\tilde{P}\left(  u\left(  t\right)  \right)
=P\left(  u\left(  t\right)  \right)  $ is invariant in time by using the
translation invariance of (GP). For general $u\left(  0\right)  \in X_{0}$, we
can choose a sequence $\left\{  u_{n}\left(  0\right)  \right\}
\subset1+H^{1}\left(  R^{3}\right)  $ such that $\left\Vert u_{n}\left(
0\right)  -u\left(  0\right)  \right\Vert _{d_{1}}\rightarrow0$ when
$n\rightarrow\infty$. Then for any $t\in\mathbf{R},\ P\left(  u_{n}\left(
t\right)  \right)  =P\left(  u_{n}\left(  0\right)  \right)  $, letting
$n\rightarrow\infty$, we get $\tilde{P}\left(  u\left(  t\right)  \right)
=\tilde{P}\left(  u\left(  0\right)  \right)  $ due to the continuous
dependence of solutions to (GP) on the initial with respect to the distance
$d_{1}$ (see \cite{gerard-cauchy}).
\end{proof}

\begin{remark}
In \cite{Maris-annal}, the$\ $momentum was extended from $1+H^{1}\left(
\mathbf{R}^{3}\right)  $ to the energy space $X_{0}$ in a different way and it
was shown that such extended momentum is continuous on $\left(  X_{0}%
,d_{1}\right)  $. So by Lemma \ref{lemma-momentum} (ii), the extended momentum
in \cite{Maris-annal} gives the same functional as $\tilde{P}\left(  u\right)
$, but the form of $\tilde{P}\left(  u\right)  $ given in
(\ref{defn-momentum-new}) is more explicit.
\end{remark}

\subsection{The energy-momentum functional}

First, we show that the functional $E\circ\psi:X_{1}\rightarrow\mathbf{R}$ is smooth.

\begin{lemma}
\label{lemma-smooth-energy-gp}The functional
\begin{equation}
\bar{E}\left(  w\right)  :=E\circ\psi(w)=\frac{1}{2}\int_{\mathbf{R}^{3}%
}\left[  |\nabla\psi(w)|^{2}+\frac{1}{2}(|\psi(w)|^{2}-1)^{2}\right]  \ dx
\label{energy-coordinate}%
\end{equation}
is $C^{\infty}$ on $X_{1}.$
\end{lemma}

\begin{proof}
For $w=\left(  w_{1},w_{2}\right)  \in X_{1}$, that is, $w_{1}\in H^{1}%
,w_{2}\in\dot{H}^{1},$
\[
\nabla\psi(w)=\nabla w_{1}-\chi\left(  D\right)  \left(  w_{2}\nabla
w_{2}\right)  +i\nabla w_{2},
\]%
\[
|\psi(w)|^{2}-1=(w_{1}-\chi(D)(\frac{w_{2}^{2}}{2}))^{2}+2w_{1}+(1-\chi
(D))(w_{2}^{2}),
\]
and by (\ref{ineq-1})-(\ref{ineq-2}),
\[
\nabla w_{1},\ \chi\left(  D\right)  \left(  w_{2}\nabla w_{2}\right)  ,\nabla
w_{2},(1-\chi(D))(w_{2}^{2})\in L^{2}%
\]%
\[
w_{1},\chi(D)(\frac{w_{2}^{2}}{2})\in L^{4}.
\]
We can write the right hand side of (\ref{energy-coordinate}) as a sum of
multilinear forms, as in the proof of Lemma \ref{lemma-momentum}. The
$C^{\infty}$ property of $E\circ\psi(w)$ thus follows.
\end{proof}

Define the energy-momentum functional $E_{c}\left(  u\right)  =E\left(
u\right)  +c\tilde{P}\left(  u\right)  $ on $X_{0}$ and
\begin{equation}
\bar{E}_{c}\left(  w\right)  =E_{c}\circ\psi\left(  w\right)  =\bar{E}\left(
w\right)  +c\bar{P}\left(  w\right)  ,\ w\in X_{1}
\label{definition-energy-momentum}%
\end{equation}
which is a smooth functional on space $X_{1}$. Let $U_{c}=u_{c}+iv_{c}$ be a
finite energy traveling wave solution of (GP) equation, that is, $\left(
u_{c},v_{c}\right)  $ satisfies
\begin{equation}
\left\{
\begin{array}
[c]{ll}%
\Delta u_{c}+c\partial_{x_{1}}v_{c}=-\left(  1-|U_{c}|^{2}\right)  u_{c}, & \\
\Delta v_{c}-c\partial_{x_{1}}u_{c}=-\left(  1-|U_{c}|^{2}\right)  v_{c}. &
\end{array}
\right.  \label{eqn-tW-GP}%
\end{equation}

\begin{lemma}
\label{le-2} Let $w_{c}=\left(  w_{1c},w_{2c}\right)  \in X_{1}$ be such that
$\psi(w_{c})=U_{c}=u_{c}+iv_{c}$. Then $U_{c}$ solves the traveling wave
equation if and only if $\bar{E}_{c}^{\prime}(w_{c})=0$.
\end{lemma}

\begin{proof}
Since $\bar{E}_{c}^{\prime}(w_{c})\in X_{1}^{\ast}$ and the Schwartz class is
dense in $X_{1}$, we have $\bar{E}_{c}^{\prime}(w_{c})=0$ if and only if
$\langle\bar{E}_{c}{}^{\prime}(w_{c}),\phi\rangle=0$ for all $\phi$ in
Schwartz class. One may compute by integration by parts that, $\bar{E}%
_{c}^{\prime}(w_{c})$ satisfies, for any $\phi=\left(  \phi_{1},\phi
_{2}\right)  $ in Schwartz class,
\begin{align}
\langle\bar{E}_{c}{}^{\prime}(w_{c}),\phi\rangle &  =\int_{\mathbf{R}^{3}%
}[-\Delta u_{c}-\left(  1-|U_{c}|^{2}\right)  u_{c}-c\partial_{x_{1}}%
v_{c}]\left(  \phi_{1}-\chi(D)\left(  v_{c}\phi_{2}\right)  \right)
dx\nonumber\\
&  +\int_{\mathbf{R}^{3}}[-\Delta v_{c}-\left(  1-|U_{c}|^{2}\right)
v_{c}+c\partial_{x_{1}}u_{c}]\phi_{2}dx. \label{eqn-E-c-1st-variation}%
\end{align}
Therefore, it is clear that $\bar{E}_{c}^{\prime}(w_{c})=0$ if and only if
(\ref{eqn-tW-GP}) holds.

\end{proof}

We now compute the second variation of $\bar{E}_{c}\left(  \psi\right)  $. By
straightforward computations using the criticality of $w_{c}$, we have, for
any $\phi$ in Schwartz class,
\begin{align}
&  \langle\bar{E}_{c}{}^{\prime\prime}(w_{c})\phi,\phi\rangle:=q_{c}%
(\phi)\label{2nd -variation}\\
&  =\int_{\mathbf{R}^{3}}[\left\vert \nabla\left(  \phi_{1}-\chi(D)\left(
v_{c}\phi_{2}\right)  \right)  \right\vert ^{2}+\left\vert \nabla\phi
_{2}\right\vert ^{2}\nonumber\\
&  \ \ \ \ \ \ \ \ +\left(  3u_{c}^{2}+v_{c}^{2}-1\right)  \left\vert \phi
_{1}-\chi(D)\left(  v_{c}\phi_{2}\right)  \right\vert ^{2}+\left(  u_{c}%
^{2}+3v_{c}^{2}-1\right)  \left\vert \phi_{2}\right\vert ^{2}\nonumber\\
&  \ \ \ \ \ \ \ \ +4u_{c}v_{c}\left(  \phi_{1}-\chi(D)\left(  v_{c}\phi
_{2}\right)  \right)  \phi_{2}-2c\left(  \phi_{1}-\chi(D)\left(  v_{c}\phi
_{2}\right)  \right)  \partial_{x_{1}}\phi_{2}]\ dx.\nonumber
\end{align}
Since the functional $\bar{E}_{c}\left(  w\right)  $ is smooth on $X_{1}$, its
second variation at $w_{c}$ which is given by the quadratic form $q_{c}$ of
(\ref{2nd -variation}) is well-defined and bounded on $X_{1}$.

Define the operator
\begin{equation}
L_{c}:=\left(
\begin{array}
[c]{cc}%
-\Delta-1+3u_{c}^{2}+v_{c}^{2} & -c\partial_{x_{1}}+2u_{c}v_{c}\\
c\partial_{x_{1}}+2u_{c}v_{c} & -\Delta-1+u_{c}^{2}+3v_{c}^{2}%
\end{array}
\right)  , \label{operator-Lc}%
\end{equation}
then formally we can write
\begin{equation}
q_{c}(\phi)=\left\langle L_{c}\left(
\begin{array}
[c]{c}%
\phi_{1}-\chi(D)\left(  v_{c}\phi_{2}\right) \\
\phi_{2}%
\end{array}
\right)  ,\left(
\begin{array}
[c]{c}%
\phi_{1}-\chi(D)\left(  v_{c}\phi_{2}\right) \\
\phi_{2}%
\end{array}
\right)  \right\rangle . \label{2nd-variation-matrix}%
\end{equation}
Here we use $\left\langle \cdot,\cdot\right\rangle $ for the dual product of
$X_{1}=H^{1}\times\dot{H}^{1}$ and its dual $X_{1}^{\ast}=H^{-1}\times\dot
{H}^{-1}$, and $\left(  \cdot,\cdot\right)  $ is used for the inner product in
$X_{1}$. By \cite{B-G-S}, the traveling wave solutions $\left(  u_{c}%
,v_{c}\right)  $ of (GP) equation satisfy: $u_{c}-1,v_{c}\in H^{k}$ for any
$k\geq0$, and $u_{c}-1=O\left(  \frac{1}{\left\vert x\right\vert ^{3}}\right)
$, $v_{c}=O\left(  \frac{1}{\left\vert x\right\vert ^{2}}\right)  $ for
$|x|\rightarrow\infty$. Since $\phi_{2}\in\dot{H}^{1}$ implies that
$\chi(D)\left(  v_{c}\phi_{2}\right)  \in L^{2}$, the mapping $K:X_{1}%
\rightarrow X_{1}\ $defined by
\begin{equation}
\ K\left(
\begin{array}
[c]{c}%
\phi_{1}\\
\phi_{2}%
\end{array}
\right)  =\left(
\begin{array}
[c]{c}%
\phi_{1}-\chi(D)\left(  v_{c}\phi_{2}\right) \\
\phi_{2}%
\end{array}
\right)  \label{definition-K}%
\end{equation}
is an isomorphism on $X_{1}$. To study the quadratic form $q_{c}(\phi)$ on
$X_{1}$, it is equivalent to study the quadratic form
\[
\tilde{q}_{c}(\phi)=\langle L_{c}\left(  \phi_{1},\phi_{2}\right)
^{T},\left(  \phi_{1},\phi_{2}\right)  ^{T}\rangle
\]
on $X_{1}$. To simplify notations, we write $\langle L_{c}\left(  \phi
_{1},\phi_{2}\right)  ^{T},\left(  \phi_{1},\phi_{2}\right)  ^{T}\rangle$ as
$\langle L_{c}\phi,\phi\rangle$. The quadratic form $\langle L_{c}\phi
,\phi\rangle$ is well defined and bounded on $X_{1}$, by the boundedness of
$q_{c}(\phi)$ and the isomorphism of $K$ on $X_{1}$. This can also be seen
directly by using the Hardy inequality
\begin{equation}
\left\Vert \frac{u}{\left\vert x\right\vert }\right\Vert _{L^{2}\left(
\mathbf{R}^{N}\right)  }\leq\frac{2}{N-2}\left\Vert \nabla u\right\Vert
_{L^{2}\left(  \mathbf{R}^{N}\right)  },\text{ for any }N\geq3. \label{hardy}%
\end{equation}
Since $|u_{c}^{2}+3v_{c}^{2}-1|,\ \left\vert v_{c}\right\vert \leq\frac
{C}{|x|^{2}}$, we have
\[
\left\vert \int_{\mathbf{R}^{3}}(u_{c}^{2}+3v_{c}^{2}-1)\phi_{2}%
^{2}dx\right\vert \leq C\int_{\mathbf{R}^{3}}\frac{\phi_{2}^{2}}{|x|^{2}%
}dx\leq C\int_{\mathbf{R}^{3}}|\nabla\phi_{2}|^{2}dx
\]
and
\begin{align*}
\left\vert \int_{\mathbf{R}^{3}}u_{c}v_{c}\ \phi_{1}\phi_{2}\ dx\right\vert
&  \leq C\left\Vert \phi_{1}\right\Vert _{L^{2}}\left\Vert \frac{\phi_{2}%
}{|x|}\right\Vert _{L^{2}}\\
&  \leq C\left(  \left\Vert \phi_{1}\right\Vert _{L^{2}}^{2}+\left\Vert
\nabla\phi_{2}\right\Vert _{L^{2}}^{2}\right)  .
\end{align*}

\begin{remark}
The quadratic form $q_{c}(\phi)=\langle\bar{E}_{c}{}^{\prime\prime}(w_{c}%
)\phi,\phi\rangle$ given in (\ref{2nd -variation}) and
(\ref{2nd-variation-matrix}) can be seen in the following way. Suppose $w\in
H^{1}$, then $u=\psi\left(  w\right)  \in1+H^{1}$ and
\[
\bar{E}_{c}\left(  w\right)  :=E_{c}\circ\psi(w)=E\left(  u\right)  +cP\left(
u\right)  \text{. }%
\]
If the first order variation of $w$ at $w_{c}~$is $\delta w=\phi$, then
$\delta u=K\phi$ and $\delta^{2}u=-\chi(D)(\phi_{2}^{2})$. So
\begin{align*}
\left\langle \bar{E}_{c}^{\prime\prime}\left(  w_{c}\right)  \phi
,\phi\right\rangle  &  =\left\langle E_{c}^{\prime\prime}\left(  U_{c}\right)
\delta u,\delta u\right\rangle +\left\langle E_{c}^{\prime}\left(
U_{c}\right)  ,\delta^{2}u\right\rangle \\
&  =\left\langle L_{c}\left(  K\phi\right)  ,K\phi\right\rangle ,
\end{align*}
since $E_{c}^{\prime\prime}\left(  U_{c}\right)  =L_{c}$ and $E_{c}^{\prime
}\left(  U_{c}\right)  =0$ by the equation (\ref{eqn-tW-GP}).
\end{remark}

\subsection{Spectral properties of second order variation}

Differentiating the traveling wave equation (\ref{eqn-tW-GP}) in $x_{i}$, we
get $L_{c}\partial_{x_{i}}U_{c}=0$. We assume that these are all the kernels
of $L_{c}$, i.e.,
\begin{equation}
\ker L_{c}=Z=span\left\{  \partial_{x_{i}}U_{c},i=1,2,3\right\}  .
\label{assumption-NDG}%
\end{equation}

\begin{remark}
The non-degeneracy condition (\ref{assumption-NDG}) for $c=c_{0}\ $implies
that the traveling wave $U_{c_{0}}$ is locally unique. More precisely, there
exists a unique $C^{1}$ curve of traveling waves passing through $\left(
c_{0},U_{c_{0}}\right)  $. See Theorem \ref{thm-continuation-3d} for the proof.
\end{remark}

In \cite{Maris-annal}, traveling wave solutions to (GP) were found by
minimizing $\bar E_{c}$ subject to a Pohozaev type constraint. Our main result
of this section is to give a spectral decomposition of the quadratic form
$\tilde{q}_{c}(\phi)$ which is the quadratic part of $E_{c}$ at $U_{c}$.

\begin{proposition}
\label{prop-quadratic}For $0<c<\sqrt{2}$, let $U_{c}\ $be a traveling wave
solution of (GP) constructed in \cite{Maris-annal} and $L_{c}$ be the operator
defined by (\ref{operator-Lc}). Assume the non-degeneracy condition
(\ref{assumption-NDG}). The space $X_{1}$ is decomposed as a direct sum%
\[
X_{1}=N \oplus Z \oplus P,
\]
where $Z$ is defined in (\ref{assumption-NDG}), $N$ is one-dimensional and
such that $\tilde{q}_{c}(u)=\left\langle L_{c}u,u\right\rangle <0$ for $0\neq
u\in N$, and $P$ is a closed subspace such that
\[
\tilde{q}_{c}(u)\geq\delta\left\Vert u\right\Vert _{X_{1}}^{2},\ \ \forall
u\in P,
\]
for some constant $\delta>0.$
\end{proposition}

\begin{proof}
Define the isomorphism $G:L^{2}\rightarrow X_{1}$ by
\begin{equation}
G\varphi=(-\Delta+1)^{-\frac{1}{2}}\varphi_{1}+i(-\Delta)^{-\frac{1}{2}%
}\varphi_{2}, \label{definition-G}%
\end{equation}
for $\varphi=\varphi_{1}+i\varphi_{2}\in L_{2}$. Let $\tilde{L}_{c}:=\tilde
{G}\circ L_{c}\circ\tilde{G}$ with
\begin{equation}
\tilde{G}=\left(
\begin{array}
[c]{cc}%
(-\Delta+1)^{-\frac{1}{2}} & 0\\
0 & (-\Delta)^{-\frac{1}{2}}%
\end{array}
\right)  , \label{defn-G-tilde}%
\end{equation}
and define the quadratic form on $L_{2}$ by
\begin{equation}
p_{c}(\varphi)=\tilde{q}_{c}\left(  G\varphi\right)  =\langle\tilde{L}%
_{c}\left(  \varphi_{1},\varphi_{2}\right)  ^{T},\left(  \varphi_{1}%
,\varphi_{2}\right)  ^{T}\rangle:=\langle\tilde{L}_{c}\varphi,\varphi\rangle.
\label{quadratic-form-relation}%
\end{equation}
Then
\[
\frac{\tilde{q}_{c}(\varphi)}{\left\Vert \varphi\right\Vert _{X_{1}}^{2}%
}=\frac{p_{c}(G^{-1}\varphi)}{\Vert G^{-1}\varphi\Vert_{L^{2}}^{2}%
},\ \text{for any }\varphi\in X_{1},
\]
and it is equivalent to prove the conclusions of Proposition for the quadratic
form $p_{c}(\varphi)$ on $L^{2}$. Since $u_{c}\rightarrow1$, $v_{c}%
\rightarrow0$ as $|x|\rightarrow\infty$, let
\begin{equation}
L_{c,\infty}:=\left(
\begin{array}
[c]{cc}%
-\Delta+2 & -c\partial_{x_{1}}\\
c\partial_{x_{1}} & -\Delta
\end{array}
\right)  \label{defn-L-infty}%
\end{equation}
and $q_{c,\infty}(\phi)=\langle L_{c,\infty}\phi,\phi\rangle$.
Correspondingly, let%
\[
\tilde{L}_{c,\infty}:=\tilde{G}\circ L_{c,\infty}\circ\tilde{G}%
\]
and $p_{c,\infty}(\varphi)=\langle\tilde{L}_{c,\infty}\varphi,\varphi\rangle$.
The properties of the quadratic form $p_{c}(\varphi)$ on $L^{2}$ follow from
the spectral properties of the operator $\tilde{L}_{c}$. We claim that:

\begin{enumerate}
\item[(i)] $\tilde{L}_{c}:L^{2}\rightarrow L^{2}$ is self-adjoint and bounded.

\item[(ii)] $\tilde{L}_{c}$ has one-dimensional negative eigenspace,
\[
\ker\tilde{L}_{c}=\left\{  G^{-1}\partial_{x_{i}}U_{c},i=1,2,3\right\}  ,
\]
and the rest of the spectrum are uniformly positive.
\end{enumerate}

The proof of these claims will be split into a few lemmas to be proved later
and we outline the rest of the proof of the Proposition based on these lemmas.

To prove claim (i), first we note that the constant coefficient operator
$\tilde{L}_{c,\infty}:L^{2}\rightarrow L^{2}$ is self-adjoint and bounded. We
shall show that the operator $\tilde{L}_{c}-\tilde{L}_{c,\infty}$ is compact
on $L^{2}$. Indeed,
\begin{align}
&  \ \ \ \ \ \ \tilde{L}_{c}-\tilde{L}_{c,\infty}\label{operator-difference}\\
&  =\left(
\begin{array}
[c]{cc}%
(-\Delta+1)^{-\frac{1}{2}}\left(  3u_{c}^{2}-3+v_{c}^{2}\right)
(-\Delta+1)^{-\frac{1}{2}} & 2(-\Delta+1)^{-\frac{1}{2}}u_{c}v_{c}%
(-\Delta)^{-\frac{1}{2}}\\
2(-\Delta)^{-\frac{1}{2}}u_{c}v_{c}(-\Delta+1)^{-\frac{1}{2}} & (-\Delta
)^{-\frac{1}{2}}\left(  u_{c}^{2}-1+3v_{c}^{2}\right)  (-\Delta)^{-\frac{1}%
{2}}%
\end{array}
\right) \nonumber\\
&  =\left(
\begin{array}
[c]{cc}%
L_{11} & L_{12}\\
L_{21} & L_{22}%
\end{array}
\right)  .\nonumber
\end{align}
By Lemma \ref{lemma-compactness}, the operators $L_{ij}$ $\left(
i,j=1,2\right)  $ are all compact on $L^{2}\left(  \mathbf{R}^{3}\right)  $.
Moreover, $L_{11},L_{22}$ are symmetric and $L_{21}=L_{12}^{\ast}$, thus
$\tilde{L}_{c}-\tilde{L}_{c,\infty}$ is bounded and self-adjoint and (i) is proved.

To prove claim (ii), first we note that by Lemma \ref{le-6}, there exists
$\delta_{0}>0$, such that
\[
\left\langle \tilde{L}_{c,\infty}\varphi,\varphi\right\rangle \geq\delta
_{0}\left\Vert \varphi\right\Vert _{L^{2}}^{2}.
\]
Thus $\sigma_{\text{ess}}\left(  \tilde{L}_{c,\infty}\right)  \subset
\lbrack\delta_{0},+\infty)$. By Weyl's Theorem and the compactness of
$\tilde{L}_{c}-\tilde{L}_{c,\infty}$, we have $\sigma_{\text{ess}}\left(
\tilde{L}_{c}\right)  =\sigma_{\text{ess}}\left(  \tilde{L}_{c,\infty}\right)
\subset\lbrack\delta_{0},+\infty)$. This shows that the negative eigenspace of
$\tilde{L}_{c}$ is finite-dimensional. By assumption (\ref{assumption-NDG}),
$\ker\tilde{L}_{c}=\left\{  G\partial_{x_{i}}U_{c},i=1,2,3\right\}  $. By
Lemmas \ref{le-5} and \ref{le-4}, the negative eigenspace of $\tilde{L}_{c}$
is one-dimensional. This proves claim (ii) and finishes the proof of the
Proposition.

\end{proof}

From the above Proposition and the relation
\begin{equation}
\langle\bar{E}_{c}{}^{\prime\prime}(w_{c})\phi,\phi\rangle=\left\langle
L_{c}\left(  K\phi\right)  ,K\phi\right\rangle ,
\label{relation-quadratic forms}%
\end{equation}
where $K$ is defined by (\ref{definition-K}), we immediately get the following

\begin{corollary}
\label{Cor-quadratic-form}Under the conditions of Proposition
\ref{prop-quadratic}, the space $X_{1}$ is decomposed as a direct sum%
\[
X_{1}=N^{\prime}\oplus Z^{\prime}\oplus P^{\prime},
\]
where $Z^{\prime}=\left\{  \partial_{x_{i}}w_{c},i=1,2,3\right\}  $,
$N^{\prime}$ is a one-dimensional subspace such that $q_{c}(u)=\langle\bar
{E}_{c}{}^{\prime\prime}(w_{c})u,u\rangle<0$ for $0\neq u\in N^{\prime}$, and
$P^{\prime}$ is a closed subspace such that
\[
q_{c}(u)\geq\delta\left\Vert u\right\Vert _{X_{1}}^{2},\ \forall u\in
P^{\prime}%
\]
for some constant $\delta>0.$
\end{corollary}

\begin{proof}
We define $N^{\prime}=K^{-1}N,\ Z^{\prime}=K^{-1}Z$ and $P^{\prime}=K^{-1}P$,
where $N,Z,P$ are defined in Proposition \ref{prop-quadratic}. Then the
conclusion follows by (\ref{relation-quadratic forms}). In particular,
$Z^{\prime}=K^{-1}Z=\left\{  \partial_{x_{i}}w_{c},i=1,2,3\right\}  $ since
\[
\partial_{x_{i}}U_{c}=\partial_{x_{i}}\psi\left(  w_{c}\right)  =K\partial
_{x_{i}}w_{c}\text{.}%
\]

\end{proof}

Now we prove several lemmas used in the proof of Proposition
\ref{prop-quadratic}. We use $C$ for a generic constant in the estimates.

\begin{lemma}
\label{lemma-compactness}The operators $L_{ij}$ $\left(  i,j=1,2\right)  $
defined in (\ref{operator-difference}) are compact on $L^{2}\left(
\mathbf{R}^{3}\right)  .$
\end{lemma}

\begin{proof}
Since $V_{1}\left(  x\right)  =3u_{c}^{2}-3+v_{c}^{2}\rightarrow0$ when
$\left\vert x\right\vert \rightarrow\infty$, and the operator $(-\Delta
+1)^{-\frac{1}{2}}:L^{2}\rightarrow H^{1}\ $is bounded, thus $\left(
3u_{c}^{2}-3+v_{c}^{2}\right)  (-\Delta+1)^{-\frac{1}{2}}$ is compact on
$L^{2}$ by the local compactness of $H^{1}\rightarrow L^{2}$. So $L_{11}$ is
compact on $L^{2}$.

Take a sequence $\left\{  v_{k}\right\}  \subset L^{2}\left(  \mathbf{R}%
^{3}\right)  $ and $v_{k}\rightarrow v_{\infty}$ weakly in $L^{2}$. To show an
operator $T$ is compact on $L^{2}$, it suffices to prove that $Tv_{k}%
\rightarrow Tv_{\infty}$ strongly in $L^{2}$. By Hardy's inequality in the
Fourier space,
\begin{align*}
\left\Vert L_{21}\left(  v_{k}-v_{\infty}\right)  \right\Vert _{L^{2}}  &
=\left\Vert \frac{1}{\left\vert \xi\right\vert }\left(  2u_{c}v_{c}%
(-\Delta+1)^{-\frac{1}{2}}\left(  v_{k}-v_{\infty}\right)  \right)
^{\symbol{94}}\left(  \xi\right)  \right\Vert _{L^{2}}\\
&  \leq C\left\Vert \nabla_{\xi}\left(  2u_{c}v_{c}(-\Delta+1)^{-\frac{1}{2}%
}\left(  v_{k}-v_{\infty}\right)  \right)  ^{\symbol{94}}\left(  \xi\right)
\right\Vert _{L^{2}}\\
&  \leq C\left\Vert \left\vert x\right\vert v_{c}(-\Delta+1)^{-\frac{1}{2}%
}\left(  v_{k}-v_{\infty}\right)  \right\Vert _{L^{2}}\rightarrow0,
\end{align*}
since the operator $\left\vert x\right\vert v_{c}(-\Delta+1)^{-\frac{1}{2}}$
is compact by using $\left\vert x\right\vert v_{c}=O\left(  \frac
{1}{\left\vert x\right\vert }\right)  $. Then $L_{12}=L_{21}^{\ast}$ is also compact.

To show the compactness of $L_{22}$, first note that $V_{2}\left(  x\right)
=\left(  u_{c}^{2}-1+3v_{c}^{2}\right)  =O\left(  \frac{1}{\left\vert
x\right\vert ^{3}}\right)  $. Let $\chi\in C_{0}^{\infty}(\mathbf{R}%
^{3},[0,1])$ be a radial cut-off function such that $\chi(\xi)=1$ when
$\left\vert \xi\right\vert \leq\frac{1}{2}$ and $\chi(\xi)=0$ when $\left\vert
\xi\right\vert \geq1$. For any $R>0$, let $\chi_{R}=\chi\left(  \frac{x}%
{R}\right)  $. Denote$\ u_{k}=(-\Delta)^{-\frac{1}{2}}\left(  v_{k}-v_{\infty
}\right)  $, then $\left\Vert u_{k}\right\Vert _{\dot{H}^{1}}\leq\left\Vert
v_{k}-v_{\infty}\right\Vert _{L^{2}}\leq C$. Thus
\begin{align*}
&  \ \ \ \ \left\Vert L_{22}\left(  v_{k}-v_{\infty}\right)  \right\Vert
_{L^{2}}\\
&  \leq C\left\Vert \left\vert x\right\vert V_{2}\left(  x\right)
u_{k}\right\Vert _{L^{2}}\\
&  \leq C\left(  \left\Vert \left\vert x\right\vert V_{2}\left(  x\right)
\chi_{R}u_{k}\right\Vert _{L^{2}}+\left\Vert \left\vert x\right\vert
V_{2}\left(  x\right)  \left(  1-\chi_{R}\right)  u_{k}\right\Vert _{L^{2}%
}\right) \\
&  \leq C\left(  \left\Vert \left\vert x\right\vert V_{2}\left(  x\right)
\chi_{R}u_{k}\right\Vert _{L^{2}} +\frac{1}{R}\left\Vert \frac{1}{\left\vert
x\right\vert }\left(  1-\chi_{R}\right)  u_{k}\right\Vert _{L^{2}}\right)  .
\end{align*}
We have
\begin{align*}
\left\Vert \frac{1}{\left\vert x\right\vert }\left(  1-\chi_{R}\right)
u_{k}\right\Vert _{L^{2}}  &  \leq C\left\Vert \nabla\left[  \left(
1-\chi_{R}\right)  u_{k}\right]  \right\Vert _{L^{2}}\\
&  \leq C\left(  \left\Vert \nabla u_{k}\right\Vert _{L^{2}}+\left\Vert
\frac{1}{R}\nabla\chi\left(  \frac{x}{R}\right)  \right\Vert _{L^{3}%
}\left\Vert u_{k}\right\Vert _{L^{6}}\right) \\
&  \leq C\left\Vert \nabla u_{k}\right\Vert _{L^{2}},
\end{align*}
so $\frac{1}{R}\left\Vert \frac{1}{\left\vert x\right\vert }\left(  1-\chi
_{R}\right)  u_{k}\right\Vert _{L^{2}}$ can be made arbitrarily small by
taking $R$ sufficiently large. For fixed $R$, by the compactness of $H_{0}^{1}
(\{|x| < R\}) \rightarrow L^{2}$, we get that
\[
\left\Vert \left\vert x\right\vert V_{2}\left(  x\right)  \chi_{R}%
u_{k}\right\Vert _{L^{2}}
\rightarrow0\text{, when }k\rightarrow\infty\text{. }%
\]
Thus $\left\Vert L_{22}\left(  v_{k}-v_{\infty}\right)  \right\Vert _{L^{2}%
}\rightarrow0$ when $k\rightarrow\infty$ and this finishes the proof of the lemma.
\end{proof}

Define the functional $\bar{P}_{c}:X_{1}\rightarrow\mathbf{R}$ by
\begin{equation}
\bar{P}_{c}(w)=\int_{\mathbf{R}^{3}}|\partial_{x_{1}}\psi(w)|^{2}%
dx+2c\tilde{P}\circ\psi(w)+\int_{\mathbf{R}^{3}}\frac{1}{2}(1-|\psi
(w)|^{2})^{2}dx. \label{fun-5}%
\end{equation}

\begin{lemma}
\label{le-3} Assume that $\varphi\in X_{1}$ satisfies $\langle\bar{P}_{c}%
{}^{\prime}(w_{c}),\varphi\rangle=0$, where $w_{c}=\psi^{-1}\left(
U_{c}\right)  $, then there holds $q_{c}(\varphi)=\langle\bar{E}_{c}{}%
^{\prime\prime}(w_{c})\varphi,\varphi\rangle\geq0$.
\end{lemma}

\begin{proof}
First, we note that $\bar{P}_{c}{}^{\prime}(w_{c})\neq0$. Indeed, suppose
$\bar{P}_{c}{}^{\prime}(w_{c})=0$. Define
\[
A\left(  u\right)  =\int_{\mathbf{R}^{3}}|\partial_{x_{2}}u|^{2}%
+|\partial_{x_{3}}u|^{2}dx.
\]
Then since $\bar{E}_{c}{}^{\prime}(w_{c})=0$ and $\bar{E}_{c}\left(  w\right)
-\frac{1}{2}\bar{P}_{c}\left(  w\right)  =A\circ\psi(w),$ we have $\left(
A\circ\psi\right)  ^{\prime}(w_{c})=0$, that is, $K^{\ast}\left(
\Delta_{x_{2}x_{3}}U_{c}\right)  =0$. Here,
\begin{equation}
K^{\ast}\phi=\left(
\begin{array}
[c]{c}%
\phi_{1}\\
\phi_{2}-v_{c}\left(  \chi(D)\phi_{1}\right)
\end{array}
\right)  \label{defn-K*}%
\end{equation}
is the adjoint of $K$ defined in (\ref{definition-K}). Thus $\Delta
_{x_{2}x_{3}}U_{c}=0$ which implies that $U_{c}\equiv1$, a contradiction.

Thus we can choose $\phi\in X_{1}$ such that $\langle\bar{P}_{c}{}^{\prime
}(w_{c}),\phi\rangle\neq0.\ $Set $G(\sigma,s)=\bar{P}_{c}(w_{c}+\sigma
\phi+s\varphi)$. Then
\begin{align*}
G(0,0)  &  =\bar{P}_{c}(w_{c})\\
&  =\int_{\mathbf{R}^{3}}|\partial_{x_{1}}U_{c}|^{2}dx+2c\int_{\mathbf{R}^{3}%
}\langle i\partial_{x_{1}}(1-U_{c}),1-U_{c}\rangle dx+\int_{\mathbf{R}^{3}%
}\frac{1}{2}(1-|U_{c}|^{2})^{2}dx\\
&  =0,
\end{align*}
by the Pohozaev-type identity (see Proposition 4.1 of
\cite{Maris-non-existence}). Since
\[
\frac{\partial G}{\partial\sigma}(0,0)=\langle\bar{P}_{c}{}^{\prime}%
(w_{c}),\phi\rangle\neq0,
\]
by the implicit function theorem, there exists a $C^{1}$ function $\sigma(s)$
near $s=0$ such that $\sigma(0)=0$ and
\[
G(\sigma(s),s)=\bar{P}_{c}(w_{c}+\sigma(s)\phi+s\varphi)=0.
\]
Then from $\frac{d}{ds}G(\sigma(s),s)|_{s=0}=0$, we get $\langle\bar{P}_{c}%
{}^{\prime}(w_{c}),\sigma^{\prime}(0)\phi+\varphi\rangle=0$. Since
$\langle\bar{P}_{c}{}^{\prime}(w_{c}),\varphi\rangle=0$ and $\langle\bar
{P}_{c}{}^{\prime}(w_{c}),\phi\rangle\neq0$, we get $\sigma^{\prime}(0)=0$.
Let $w(s)=w_{c}+\sigma(s)\phi+s\varphi$ and $g(s)=\bar{E}_{c}(w(s))$. Then we
have $w(0)=w_{c}$, $w^{\prime}(0)=\varphi$ and $\bar{P}_{c}(w(s))=0$. By the
variational characterization of traveling wave solution \cite{Maris-annal}, we
know that $s=0$ is a local minimum point of $g(s)$. So, we get $g^{\prime
\prime}(0)\geq0$. This implies that $\langle\bar{E}_{c}{}^{\prime\prime}%
(w_{c})\varphi,\varphi\rangle\geq0$.
\end{proof}

The above lemma implies the following

\begin{lemma}
\label{le-5} For any $0<c_{0}<\sqrt{2}$, $\tilde{L}_{c_{0}}$ has at most
one-dimensional negative eigenspace.
\end{lemma}

\begin{proof}
We assume by contradiction that $\varphi,\tilde{\varphi}\in L^{2}$ are two
linearly independent eigenfunctions of $\tilde{L}_{c_{0}}$corresponding to
negative eigenvalues. Since $\tilde{L}_{c_{0}}$ is self-adjoint, we can assume
that $\left\langle \tilde{L}_{c_{0}}\varphi,\tilde{\varphi}\right\rangle =0$.
Let $w_{c_{0}}\in X_{1}$ be such that $\psi(w_{c_{0}})=U_{c_{0}}$. From the
definition of $\tilde{L}_{c_{0}}$, for any $\phi,\tilde{\phi}\in L^{2}$, we
have
\begin{equation}
\left\langle \tilde{L}_{c_{0}}\phi,\tilde{\phi}\right\rangle =\langle\bar
{E}_{c_{0}}{}^{\prime\prime}(w_{c_{0}})K^{-1}G\phi,K^{-1}G\tilde{\phi}\rangle,
\label{bilinear-relation}%
\end{equation}
where the mappings $G,K$ are defined in (\ref{definition-G}) and
(\ref{definition-K}). Let $w=K^{-1}G\varphi,\tilde{w}=K^{-1}G\tilde{\varphi}$,
then $\langle\bar{E}_{c_{0}}{}^{\prime\prime}(w_{c_{0}})w,\tilde{w}\rangle=0$
and$\ $
\[
\ \langle\bar{E}_{c_{0}}{}^{\prime\prime}(w_{c_{0}})w,w\rangle,\langle\bar
{E}_{c_{0}}{}^{\prime\prime}(w_{c_{0}})\tilde{w},\tilde{w}\rangle<0.
\]
By Lemma \ref{le-3} we have
\[
\langle(P_{c_{0}}\circ\psi)^{\prime}(w_{c_{0}}),w\rangle\neq0,\langle
(P_{c_{0}}\circ\psi)^{\prime}(w_{c_{0}}),\tilde{w}\rangle\neq0.
\]
Thus there exists $\alpha\neq0$ such that
\[
\langle(P_{c_{0}}\circ\psi)^{\prime}(w_{c_{0}}),\xi_{0}\rangle
=0,\ \mathrm{for}\ \xi_{0}=w+\alpha\tilde{w}.
\]
Again by lemma \ref{le-3}, we get
\[
\left\langle \bar{E}_{c_{0}}{}^{\prime\prime}(w_{c_{0}})\xi_{0},\xi
_{0}\right\rangle \geq0.
\]
This contradicts with
\[
\left\langle \bar{E}_{c_{0}}{}^{\prime\prime}(w_{c_{0}})\xi_{0},\xi
_{0}\right\rangle =\ \langle\bar{E}_{c_{0}}{}^{\prime\prime}(w_{c_{0}%
})w,w\rangle+\alpha^{2}\langle\bar{E}_{c_{0}}{}^{\prime\prime}(w_{c_{0}%
})\tilde{w},\tilde{w}\rangle<0.
\]
So $\tilde{L}_{c_{0}}$ has at most one-dimensional negative eigenspace.
\end{proof}

\begin{lemma}
\label{le-4} For any $0<c_{0}<\sqrt{2}$, $\tilde{L}_{c_{0}}$ has at least one
negative eigenvalue.
\end{lemma}

\begin{proof}
By (\ref{bilinear-relation}), it suffices to find a test function $w_{0}\in
X_{1}$ such that $q_{c}(w_{0})=\langle\bar{E}_{c_{0}}{}^{\prime\prime
}(w_{c_{0}})w_{0},w_{0}\rangle<0$. By (\ref{relation-quadratic forms}), it is
equivalent to find $\phi\in X_{1}$ such that $\left\langle L_{c_{0}}\phi
,\phi\right\rangle <0$. We note that the traveling wave solutions of
(\ref{eqn-tW-GP}) constructed in \cite{Maris-annal} are cylindrical symmetric,
that is, $U_{c_{0}}=U_{c_{0}}\left(  x_{1},r_{\perp}\right)  $ with $r_{\perp
}=\sqrt{x_{2}^{2}+x_{3}^{2}}$. Differentiating (\ref{eqn-tW-GP}) to $r_{\perp
}$, we get
\[
L_{c_{0}}\partial_{r_{\perp}}U_{c_{0}}=-\frac{1}{r_{\perp}^{2}}\partial
_{r_{\perp}}U_{c_{0}}.
\]
In Appendix 3, we show that $\partial_{r_{\bot}}U_{c}\in H^{1}(\mathbf{R}%
^{3})$ and $\frac{1}{r_{\perp}}\partial_{r_{\perp}}U_{c_{0}}\in L^{2}\left(
\mathbf{R}^{3}\right)  $. Thus%

\[
\left\langle L_{c_{0}}\partial_{r_{\perp}}U_{c_{0}},\partial_{r_{\perp}%
}U_{c_{0}}\right\rangle =-\left\Vert \frac{1}{r_{\perp}}\partial_{r_{\perp}%
}U_{c_{0}}\right\Vert _{L^{2}}^{2}<0\text{. }%
\]
This proves the lemma.
\end{proof}

\begin{lemma}
\label{le-6} For any $0<c_{0}<\sqrt{2}$, there exists $\delta_{0}>0$ such
that
\begin{equation}
\frac{p_{c_{0},\infty}(\varphi)}{\Vert\varphi\Vert_{L^{2}}^{2}}\geq\delta
_{0},\ \forall\ \varphi\in L^{2}. \label{ess-1}%
\end{equation}

\end{lemma}

\begin{proof}
By (\ref{quadratic-form-relation}), it suffices to prove that there exists
$\delta_{0}>0$ such that
\begin{equation}
\frac{q_{c_{0},\infty}(w)}{\Vert w_{1}\Vert_{H^{1}}^{2}+\Vert w_{2}\Vert
_{\dot{H}^{1}}^{2}}\geq\delta_{0},\ \forall\ w=w_{1}+iw_{2}\in X_{1}.
\label{ess-2}%
\end{equation}
Since $0<c_{0}<\sqrt{2}$, there exists $0<a_{0}<1$ such that $2-\frac
{c_{0}^{2}}{a_{0}^{2}}>0$. Then for $w=w_{1}+iw_{2}\in X_{1}$, we have
\begin{align*}
q_{c,\infty}(w)  &  =\int_{\mathbf{R}^{3}}\left[  |\nabla w_{1}|^{2}%
+2w_{1}^{2}+|\nabla w_{2}|^{2}-2c_{0}(\partial_{x_{1}}w_{2})w_{1}\right]
\ dx\\
&  =\int_{\mathbf{R}^{3}}(\ |\nabla w_{1}|^{2}+(2-\frac{c_{0}^{2}}{a_{0}^{2}%
})w_{1}^{2}+(1-a_{0}^{2})(\partial_{x_{1}}w_{2})^{2}\\
&  ~~~\ \ \ \ +(\partial_{x_{2}}w_{2})^{2}+(\partial_{x_{3}}w_{2})^{2}%
+(\frac{c_{0}}{a_{0}}w_{1}-a_{0}\partial_{x_{1}}w_{2})^{2})\ dx.\\
&  \geq\min\left\{  2-\frac{c_{0}^{2}}{a_{0}^{2}},1-a_{0}^{2}\right\}  \left(
\Vert w_{1}\Vert_{H^{1}}^{2}+\Vert w_{2}\Vert_{\dot{H}^{1}}^{2}\right)  .
\end{align*}
Thus (\ref{ess-2}) holds for $\delta=\min\left\{  2-\frac{c_{0}^{2}}{a_{0}%
^{2}},1-a_{0}^{2}\right\}  $.
\end{proof}

\subsection{Proof of nonlinear stability}

We can now prove the orbital stability of traveling waves on the lower branch
(i.e. when $\frac{\partial P(U_{c})}{\partial c}|_{c=c_{0}}>0$).

\begin{theorem}
\label{thm:stability}For $0<c_{0}<\sqrt{2}$, let $U_{c_{0}}\ $be a traveling
wave solution of (GP) constructed in \cite{Maris-annal}, satisfying
(\ref{assumption-NDG})$~$and $\frac{\partial P(U_{c})}{\partial c}|_{c=c_{0}%
}>0$. Then the traveling wave $U_{c_{0}}$ is orbitally stable in the following
sense:There exists constants $\varepsilon_{0},M>0$ such that for any
$0<\varepsilon<\varepsilon_{0},\ $if
\begin{equation}
u\left(  0\right)  \in X_{0},\ d_{1}\left(  u\left(  0\right)  ,U_{c_{0}%
}\right)  <M\varepsilon, \label{estimate-stability-initial}%
\end{equation}
then \qquad%
\[
\sup_{0<t<\infty}\inf_{y\in\mathbf{R}^{3}}d_{1}\left(  u\left(  \cdot
,t\right)  ,U_{c_{0}}\left(  \cdot+y\right)  \right)  <\varepsilon.
\]

\end{theorem}

The proof of this theorem basically follows the line in \cite{gss87}. However,
the more precise stability estimate (\ref{estimate-stability-initial}) was not
given there. The proof given below is to modify the proof of Theorem 3.4 in
\cite{gss87} and get (\ref{estimate-stability-initial}). First, we need the following

\begin{lemma}
\label{lemma-positive-quadratic}Let $\frac{\partial P(U_{c})}{\partial
c}|_{c=c_{0}}>0$. If $\phi\in X_{1}$ is such that
\[
\left\langle \bar{P}^{\prime}\left(  w_{c}\right)  ,\phi\right\rangle =\left(
\partial_{x_{i}}w_{c_{0}},\phi\right)  =0,\ i=1,2,3,
\]
then
\[
\langle\bar{E}_{c}{}^{\prime\prime}(w_{c_{0}})\phi,\phi\rangle\geq
\delta\left\Vert \phi\right\Vert _{X_{1}}^{2},
\]
for some $\delta>0$.
\end{lemma}

The proof of this Lemma is essentially the same as in \cite{gss87}, by using
Corollary \ref{Cor-quadratic-form} on the spectral properties of the quadratic
form $\langle\bar{E}_{c}{}^{\prime\prime}(w_{c_{0}})\cdot,\cdot\rangle$.

\textbf{Proof of Theorem \ref{thm:stability}}: Let $u\left(  t\right)
=\psi\left(  w\left(  t\right)  \right)  $. Since the mapping
\[
\psi: \left(  X_{1},\left\Vert \cdot\right\Vert _{X_{1}}\right)
\rightarrow\left(  X_{0},d_{1}\right)
\]
is locally bi-Lipschitz with the local Lipschitz constant invariant under
translation, it suffices to show the following statement: if $w\left(
0\right)  \in X_{1},\ \left\Vert w\left(  0\right)  -w_{c_{0}}\right\Vert
_{X_{1}}<M\varepsilon$, then
\[
\sup_{0<t<\infty}\inf_{y\in\mathbf{R}^{3}}\left\Vert w\left(  t,\cdot\right)
-w_{c_{0}}\left(  \cdot+y\right)  \right\Vert _{X_{1}}<\varepsilon.
\]

Let $y\left(  w\left(  t\right)  \right)  \in\mathbf{R}^{3}$ be such that the
infimum
\[
\inf_{y\in\mathbf{R}^{3}}\left\Vert w\left(  t,\cdot\right)  -w_{c_{0}}\left(
\cdot+y\right)  \right\Vert _{X_{1}}=\inf_{y\in\mathbf{R}^{3}}\left\Vert
w\left(  t,\cdot-y\right)  -w_{c_{0}}\left(  \cdot\right)  \right\Vert
_{X_{1}}%
\]
is obtained. Below, we use $\left\Vert \cdot\right\Vert $ for $\left\Vert
\cdot\right\Vert _{X_{1}}$ for simplicity and denote $T\left(  y\right)
w\left(  t\right)  =w\left(  t,\cdot+y\right)  $. Then by definition
\[
\left(  T\left(  y\left(  w\left(  t\right)  \right)  \right)  w\left(
t\right)  -w_{c_{0}},\partial_{x_{i}}w_{c_{0}}\right)  =0,i=1,2,3.
\]
Denote $u\left(  t\right)  =T\left(  y\left(  w\left(  t\right)  \right)
\right)  w\left(  t\right)  -w_{c_{0}},$ and $d\left(  t\right)  =\left\Vert
u\left(  t\right)  \right\Vert ^{2}$. Since
\begin{align*}
\left\vert \bar{P}\left(  T\left(  y\left(  w\right)  \right)  w\left(
t\right)  \right)  -\bar{P}\left(  w_{c_{0}}\right)  \right\vert  &
=\left\vert \bar{P}\left(  w\left(  0\right)  \right)  -\bar{P}\left(
w_{c_{0}}\right)  \right\vert \\
&  \leq C\left\Vert w\left(  0\right)  -w_{c_{0}}\right\Vert =C\left\Vert
d\left(  0\right)  \right\Vert ^{\frac{1}{2}}%
\end{align*}
and
\[
\bar{P}\left(  T\left(  y\left(  w\right)  \right)  w\left(  t\right)
\right)  -\bar{P}\left(  w_{c_{0}}\right)  =\left\langle \bar{P}^{\prime
}\left(  w_{c_{0}}\right)  ,u\left(  t\right)  \right\rangle +O\left(
\left\Vert u\left(  t\right)  \right\Vert ^{2}\right)  ,
\]
so
\begin{equation}
\left\vert \left\langle \bar{P}^{\prime}\left(  w_{c_{0}}\right)  ,u\left(
t\right)  \right\rangle \right\vert \leq C\left(  d\left(  t\right)
+\left\Vert d\left(  0\right)  \right\Vert ^{\frac{1}{2}}\right)  .
\label{estimate-abs-a}%
\end{equation}
Let $I:X_{1}\rightarrow\left(  X_{1}\right)  ^{\ast}$ be the isomorphism
defined by $\left\langle Iu,v\right\rangle =\left(  u,v\right)  $ for any
$u,v\in X_{1}$. Define $q=I^{-1}\bar{P}^{\prime}\left(  w_{c_{0}}\right)  $
and decompose $u\left(  t\right)  =v+aq$, where $a=\left(  u,q\right)
/\left(  q,q\right)  $ and $\left(  v,q\right)  =\left\langle \bar{P}^{\prime
}\left(  w_{c_{0}}\right)  ,v\right\rangle =0$. Then (\ref{estimate-abs-a})
implies that
\[
\left\vert a\right\vert =\frac{\left\vert \left\langle \bar{P}^{\prime}\left(
w_{c_{0}}\right)  ,u\left(  t\right)  \right\rangle \right\vert }{\left(
q,q\right)  }\leq C\left(  d\left(  t\right)  +\left\Vert d\left(  0\right)
\right\Vert ^{\frac{1}{2}}\right)  .
\]
Moreover,%
\[
\left(  v,\partial_{x_{i}}w_{c_{0}}\right)  =\left(  u\left(  t\right)
,\partial_{x_{i}}w_{c_{0}}\right)  -a\left\langle \bar{P}^{\prime}\left(
w_{c_{0}}\right)  ,\partial_{x_{i}}w_{c_{0}}\right\rangle =0.
\]
So by Lemma \ref{lemma-positive-quadratic}, we get $\langle\bar{E}_{c}%
{}^{\prime\prime}(w_{c_{0}})v,v\rangle\geq\delta\left\Vert v\right\Vert ^{2}$.
We start with
\begin{equation}
\bar{E}_{c}{}\left(  T\left(  y\left(  w\right)  \right)  w\left(  t\right)
\right)  -\bar{E}_{c}\left(  w_{c_{0}}\right)  =\bar{E}_{c}{}\left(  w\left(
0\right)  \right)  -\bar{E}_{c}\left(  w_{c_{0}}\right)  .
\label{energy-conservation}%
\end{equation}
The Taylor expansion of the left hand side of (\ref{energy-conservation})
yields
\begin{align*}
&  \frac{1}{2}\left\langle \bar{E}_{c}{}^{\prime\prime}(w_{c_{0}})u\left(
t\right)  ,u\left(  t\right)  \right\rangle +O\left(  \left\Vert u\left(
t\right)  \right\Vert ^{3}\right) \\
&  =\frac{1}{2}\left\langle \bar{E}_{c}{}^{\prime\prime}(w_{c_{0}})v\left(
t\right)  ,v\left(  t\right)  \right\rangle +O\left(  \left\vert a\right\vert
^{2}+a\left\Vert v\right\Vert +\left\Vert u\right\Vert ^{3}\right) \\
&  \geq\frac{1}{2}\delta\left\Vert v\right\Vert ^{2}-C\left(  \left\vert
a\right\vert ^{2}+a\left\Vert v\right\Vert +\left\Vert u\right\Vert
^{3}\right) \\
&  \geq\frac{1}{2}\delta\left\Vert u\right\Vert ^{2}-C^{\prime}\left(
\left\vert a\right\vert ^{2}+a\left\Vert u\right\Vert +\left\Vert u\right\Vert
^{3}\right) \\
&  =\frac{1}{2}\delta d-C^{\prime}\left(  \left(  d+\sqrt{d\left(  0\right)
}\right)  ^{2}+\left(  d+\sqrt{d\left(  0\right)  }\right)  \sqrt{d}%
+d^{\frac{3}{2}}\right) \\
&  \geq\frac{1}{4}\delta d-C^{\prime\prime}\left(  d^{2}+d^{\frac{3}{2}%
}+d\left(  0\right)  \right)  ,
\end{align*}
here, in the second inequality above we use
\[
\left\Vert u\right\Vert -\left\vert a\right\vert \left\Vert q\right\Vert
\leq\left\Vert v\right\Vert \leq\left\Vert u\right\Vert +\left\vert
a\right\vert \left\Vert q\right\Vert
\]
and in the last inequality we use
\[
\sqrt{d\left(  0\right)  }\sqrt{d}\leq\frac{1}{2}\left(  \eta d+\frac{1}{\eta
}d\left(  0\right)  \right)  ,\text{ }\eta=\frac{1}{2}\delta\text{. }%
\]
The right hand side of (\ref{energy-conservation}) is controlled by $Cd\left(
0\right)  $. Combining above, we get
\begin{equation}
d\left(  t\right)  -C_{1}F\left(  d\left(  t\right)  \right)  \leq
C_{2}d\left(  0\right)  , \label{stability-algebraic}%
\end{equation}
for some $C_{1},C_{2}>0$ and $F\left(  d\right)  =d^{2}+d^{\frac{3}{2}}$. The
stability and the estimate (\ref{estimate-stability-initial}) follows easily
from (\ref{stability-algebraic}) by taking $M=\frac{2}{C_{2}}$. $\blacksquare$

\begin{remark}
In \cite{chiron-marisII-12}, Chiron and Maris constructed $3$D traveling waves
of (\ref{eqn-GP-generalized}) with a nonnegative potential function $V\left(
s\right)  $, by minimizing the energy functional under the constraint of
constant momentum. They proved the compactness of the minimizing sequence and
as a corollary the orbital stability of these traveling waves is obtained.
There are two differences of their result and Theorem \ref{thm:stability}.
First, in \cite{chiron-marisII-12}, the orbital stability is for the set of
all minimizers which are not known to be unique. Moreover, the more precise
stability estimate (\ref{estimate-stability-initial}) cannot be obtained by
such compactness approach. Second, the stability criterion $\frac{\partial
P(U_{c})}{\partial c}|_{c=c_{0}}>0$ obtained in Theorem \ref{thm:stability}
(under the non-degeneracy assumption) confirmed the conjecture in the physical
literature (\cite{jones-et-stability} \cite{berloff-roberts-X-stability}). No
such stability criterion was obtained in \cite{chiron-marisII-12}. In our
proof, the variational characterization (such as in \cite{Maris-annal}) is
only used in Lemma \ref{le-3} to show that the second variation of
energy-momentum functional has at most one negative direction. We do not need
the compactness of the minimizing sequence and the traveling waves constructed
by other variational arguments (e.g. \cite{betheul-et-minimizer-cmp}) could
also fit into our approach.
\end{remark}

\subsection{The case of general nonlinearity}

\label{SS:GNLS}

In this section, we extend Theorem \ref{thm:stability} on nonlinear stability
to general nonlinearity $F$ satisfying the following conditions:\newline(F1)
$F\in C^{1}(\mathbf{R}^{+}) \cap C^{0} ([0, \infty)$, $C^{2}$ in a
neighborhood of $1$, $F(1)=0$ and $F^{\prime}(1)=-1$. \newline(F2) There
exists $C>0$ and $0< p_{1} \le1 \le p_{0} <2$
such that $|F^{\prime}(s)|\leq C(1 + s^{p_{1}-1} + s^{p_{0}-1})$
for all $s\geq0$.

\begin{remark}
The exponent $p_{0}$ in condition (F2) restricts the growth of $F^{\prime}$ at
infinity and $p_{1}$ is the order of singularity allowed for $F^{\prime}$ at
$s=0$, which means $F$ is only assumed to be H\"{o}lder near $s=0$. Condition
(F2) implies that $|F(s)|\leq C(1+s^{p_{0}})\ $for all $s\geq0$. The
nonlinearity of Gross-Pitaevskii equation is $F\left(  s\right)  =1-s$ which
certainly satisfies (F1)(F2).
\end{remark}

The energy function is now given by
\[
E\left(  u\right)  =\frac{1}{2}\int_{\mathbf{R}^{3}}\left[  \left\vert \nabla
u\right\vert ^{2}+V\left(  \left\vert u\right\vert ^{2}\right)  \right]  dx,
\]
\textrm{where}$\ V(s)=\int_{s}^{1}F(\tau)d\tau$. By the proof of Lemma 4.1 in
\cite{Maris-annal}, $E\left(  u\right)  <\infty$ if and only if $u\in X_{0}$
(defined in (\ref{defn-energy-space})). So we can use the same coordinate
mapping $u=\psi(w)\ \left(  w\in X_{1}\right)  \ $for the energy space. For
$w\in X_{1},\ $define
\begin{equation}
\bar{E}\left(  w\right)  :=E\circ\psi(w)=\frac{1}{2}\int_{\mathbf{R}^{3}%
}\left[  |\nabla\psi(w)|^{2}+V(|\psi(w)|^{2})\right]  dx.
\label{energy-fun-general}%
\end{equation}
In order to prove the smoothness of $\bar E$, we need the following standard
properties of Nemitski operators.

\begin{lemma}
\label{L:Nemitski} Suppose $g \in C(\mathbf{R}^{m}, \mathbf{R})$ and $|g(s)|
\le|s|^{q_{0}}$ for some $q_{0}>0$ and all $s \in\mathbf{R}^{m}$, then the
mapping $G(\phi) \triangleq g \circ\phi$ is continuous from $L^{q_{1}}
(\mathbf{R}^{n}, \mathbf{R}^{m})$ to $L^{\frac{q_{1}}{q_{0}}} (\mathbf{R}^{n},
\mathbf{R})$ where $q_{1}\in[\min\{1, \frac1{q_{0}}\}, \infty]$.
\end{lemma}

The proof is simply a modification of the one of Theorem 2.2 of
\cite{ambrosetti-primer-NA} based on Theorem 4.9 in \cite{brezis-book}, the
latter of which is valid on $\mathbf{R}^{n}$ in particular.

\begin{lemma}
\label{le-1} Assume (F1)(F2). Then the functional $\bar{E}\left(  w\right)
:X_{1}\rightarrow\mathbf{R}$ is $C^{2}$.
\end{lemma}

\begin{proof}
For $w=w_{1}+iw_{2}\in X_{1}$, we set
\[
J_{1}(w)=\int_{\mathbf{R}^{3}}|\nabla\psi(w)|^{2}dx=\int_{\mathbf{R}^{3}%
}|\nabla w_{1}-\nabla\chi(D)(\frac{w_{2}^{2}}{2})|^{2}+|\nabla w_{2}|^{2}dx,
\]%
\[
J_{2}(w)=\int_{\mathbf{R}^{3}}V(|\psi\left(  w\right)  |^{2})dx.
\]
Then $\bar{E}\left(  w\right)  =\frac{1}{2}\left(  J_{1}(w)+J_{2}(w)\right)
$. Since $J_{1}\left(  w\right)  \in C^{\infty}(X_{1},\mathbf{R})\ $as shown
in the proof of Lemma \ref{lemma-smooth-energy-gp}, it suffices to show that
$J_{2}\in C^{2}(X_{1},\mathbf{R})$. In the sequel, let $C(\Vert w\Vert_{X_{1}%
})$ be a positive constant depending on $\Vert w\Vert_{X_{1}}$ increasingly.

Following the notation in Appendix 1, we denote
\begin{align*}
\Psi_{2}\left(  w\right)   &  =\left\vert \psi\left(  w\right)  \right\vert
^{2}-1\\
&  =\left(  w_{1}-\chi(D)(\frac{w_{2}^{2}}{2})\right)  ^{2}+\left(
1-\chi(D)\right)  w_{2}^{2}+2w_{1}.
\end{align*}
Then by (\ref{ineq-1})-(\ref{ineq-2}), it is easy to show that $\Psi_{2}\in
C^{\infty}\left(  X_{1},L^{2}\cap L^{3}\right)  $. By (F1),
\[
F(s)=F(1)+F^{\prime}(1)(s-1)+(s-1)\epsilon(s-1),
\]
where $\epsilon(t)\rightarrow0$ as $t\rightarrow0$. Thus there exists
$\beta\in(0,1)$ such that
\begin{equation}
|F(s)|=|F(s)-F(1)|\leq2|s-1|,\ \mathrm{for\ all}\ s\in(1-\beta,1+\beta).
\label{general-1}%
\end{equation}
We choose three cut-off functions $\xi_{1},\xi_{2},\xi_{3}$ with supports in
\[
[0,1-\beta/2) ,\ \left(  1-\beta,1+\beta\right)  \ \text{and\ }\left(
1+\beta/2,\infty\right)
\]
respectively,\ and $0\leq\xi_{i}\leq1$, $\sum_{i=1}^{3}\xi_{i}=1$. Denote
$F_{i}\left(  s\right)  =F\left(  s\right)  \xi_{i}\left(  s\right)  $, and
$V_{i}\left(  s\right)  =\int_{s}^{1}F_{i}(\tau)d\tau$. Then by
(\ref{general-1}),$\ \left\vert F_{2}\left(  s\right)  \right\vert
\leq2\left\vert s-1\right\vert $ and by (F2)
\[
\left\vert F_{1}\left(  s\right)  \right\vert \le C, \, \left\vert
F_{3}\left(  s\right)  \right\vert \leq C\left(  1+s^{p_{0}}\right)
\implies|F_{1}(s)|, |F_{3}(s)| \leq C^{\prime}\left\vert s-1\right\vert
^{p_{0}},\text{ }%
\]
since $\left\vert s-1\right\vert \geq\beta/2$ on the supports of $F_{1},F_{3}%
$. By Lemma \ref{L:Nemitski}
we have
\[
F_{1}\left(  |\psi(w)|^{2}\right)  ,F_{3}\left(  |\psi(w)|^{2}\right)  \in
C\left(  X_{1},L^{\frac{3}{2}}\right)
\]
and $F_{2}\left(  |\psi(w)|^{2}\right)  \in C\left(  X_{1},L^{2}\right)  $.
Thus the Gateau derivative of $J_{2}(w)$ at $\phi\in X_{1}$%
\[
\langle J_{2}^{\prime}(w),\phi\rangle=-\sum_{i=1}^{3}\int_{\mathbf{R}^{3}%
}F_{i}(|\psi\left(  w\right)  |^{2})\left(  \Psi_{2}^{\prime}\left(  w\right)
\phi\right)  \ dx
\]
is continuous in $w\in X_{1}\ $and thus $J_{2}\in C^{1}\left(  X_{1}%
,\mathbf{R}\right)  $.

Now we consider the Gateau derivative of $J_{2}^{\prime}(w)$. For any
$\phi=\phi_{1}+i\phi_{2},h=h_{1}+ih_{2}\in X_{1}$, we have
\begin{align*}
J_{2}^{\prime\prime}(w)\left(  \phi,h\right)   &  =-\sum_{i=1}^{3}%
\int_{\mathbf{R}^{3}}F_{i}(|\psi\left(  w\right)  |^{2})\left(  \Psi
_{2}^{\prime\prime}\left(  w\right)  \left(  \phi,h\right)  \right)  \ dx\\
&  \ \ \ \ \ -\sum_{i=1}^{3}\int_{\mathbf{R}^{3}}F_{i}^{\prime}(|\psi\left(
w\right)  |^{2})\left(  \Psi_{2}^{\prime}\left(  w\right)  \phi\right)
\left(  \Psi_{2}^{\prime}\left(  w\right)  h\right)  \ dx\\
&  =I+II.
\end{align*}
It is not difficult to verify that above is indeed the Gateau derivative of
$J_{2}^{\prime}(w)$ and we skip the details. Now we show the continuity of
$J_{2}^{\prime\prime}(w)\left(  \phi,h\right)  $ in $w$, which implies that it
is the Fr\'{e}chet derivative of $J_{2}^{\prime}(w)$. The continuity of $I$ to
$w\in X_{1}$ follows by the same reasoning for $J_{2}^{\prime}(w)$. We write
\[
II=-\sum_{i=1}^{3}\int_{\mathbf{R}^{3}}F_{i}^{\prime}(|\psi\left(  w\right)
|^{2})\left(  \Psi_{2}^{\prime}\left(  w\right)  \phi\right)  \left(  \Psi
_{2}^{\prime}\left(  w\right)  h\right)  \ dx=-\sum_{i=1}^{3}II_{i}%
(w)(\phi,h).
\]
Since $F_{2,3}^{\prime}$ are continuous on $\mathbf{R}$ and satisfy
$|F_{2}^{\prime}(s)|,\ |F_{3}^{\prime}(s)|\leq C|s-1|,$ Lemma \ref{L:Nemitski}
and the smoothness of $\Psi_{2}:X_{1}\rightarrow L^{3}$ imply $F_{2,3}%
^{\prime}(|\psi(w)|^{2})$ is continuous from $X_{1}$ to $L^{3}$, and
consequently the uniform continuity of the quadratic forms $II_{2,3}(w)$ on
$X_{1}$ with respect to $w\in X_{1}$. To see the uniform continuity of the
quadratic forms $II_{1}(w)$ in $w$, we write it more explicitly
\[
II_{1}(w)(\phi,h)=\int_{\mathbf{R}^{3}}\big(\psi^{\prime}(w)h\big)^{T}%
\Big(F_{1}^{\prime}(|\psi\left(  w\right)  |^{2})\psi(w)\psi(w)^{T}%
\Big)\big(\psi^{\prime}(w)\phi\big)dx
\]
where in the above the complex valued $\psi(w),\psi^{\prime}(w)h,\psi^{\prime
}(w)\phi
$ are viewed as 2-dim column vectors. Since $F_{1}^{\prime}$ is supported on
$[0,1-\frac{\beta}{2})$ with $\beta\in(0,1)$ and satisfies $|F_{1}^{\prime}|
\le C (1 + s^{p_{1}-1})$, $p_{1}\in(0,1]$, we have
\[
\Big|F_{1}^{\prime}(|\psi\left(  w\right)  |^{2})\psi(w)\psi(w)^{T}\Big|\leq
C_{p}\big||\psi(w)|^{2}-1\big|^{p},\;\forall\ p\geq0.
\]
As $\Psi_{2}(w)=|\psi(w)|^{2}-1$ is a smooth mapping from $X_{1}$ to
$L^{2}\cap L^{3}$, Lemma \ref{L:Nemitski} implies that $w\rightarrow
F_{1}^{\prime}(|\psi\left(  w\right)  |^{2})\psi(w)\psi(w)^{T}$ is a
continuous mapping from $X_{1}$ to $L^{\frac{3}{2}}$. Therefore, the uniform
continuity with respect to $w$ of the quadratic form $II_{1}(w)$ on $X_{1}$
follows from the smoothness of $\psi:X_{1}\rightarrow\dot{H}^{1}$ and this
completes the proof of the lemma.
\end{proof}

A traveling wave $U_{c}=u_{c}+iv_{c}=\psi\left(  w_{c}\right)  \ $of
(\ref{eqn-GP-generalized}) satisfies the equation
\begin{equation}
-ic\partial_{x_{1}}U_{c}+\Delta U_{c}+F(|U_{c}|^{2})U_{c}=0.
\label{eqn-GP-G-TW}%
\end{equation}
Under (F1)-(F2), for any $0<c<\sqrt{2}$, traveling waves were constructed in
\cite{Maris-annal} as an energy minimizer under the constraint of Pohozaev
type identity. As in the (GP) case, $w_{c}$ is a critical point of the
momentum functional $\bar{E}_{c}\left(  w\right)  =\bar{E}\left(  w\right)
+c\bar{P}\left(  w\right)  $. The second variation functional can be written
in the form
\begin{equation}
\langle\bar{E}_{c}{}^{\prime\prime}(w_{c})\phi,\phi\rangle=\left\langle
L_{c}\left(  K\phi\right)  ,K\phi\right\rangle , \label{quadratic-general}%
\end{equation}
where $K$ is defined in (\ref{definition-K}) and
\begin{equation}
L_{c}:=\left(
\begin{array}
[c]{cc}%
-\Delta-F\left(  \left\vert U_{c}\right\vert ^{2}\right)  -F^{\prime}\left(
\left\vert U_{c}\right\vert ^{2}\right)  2u_{c}^{2} & -c\partial_{x_{1}%
}-2F^{\prime}\left(  \left\vert U_{c}\right\vert ^{2}\right)  u_{c}v_{c}\\
c\partial_{x_{1}}-2F^{\prime}\left(  \left\vert U_{c}\right\vert ^{2}\right)
u_{c}v_{c} & -\Delta-F\left(  \left\vert U_{c}\right\vert ^{2}\right)
-F^{\prime}\left(  \left\vert U_{c}\right\vert ^{2}\right)  2v_{c}^{2}%
\end{array}
\right)  . \label{operator-Lc-g}%
\end{equation}
Assuming that the traveling wave solution $U_{c}=u_{c}+iv_{c}$ satisfies the
decay estimate
\begin{equation}
u_{c}-1=o\left(  \frac{1}{\left\vert x\right\vert ^{2}}\right)  ,v_{c}%
=o\left(  \frac{1}{\left\vert x\right\vert }\right)  , \label{decay-TW-g}%
\end{equation}
and the non-degeneracy condition (\ref{assumption-NDG}) as in the (GP) case,
we can show the same decomposition result for the quadratic form $\langle
\bar{E}_{c}{}^{\prime\prime}(w_{c})\phi,\phi\rangle$, as in Proposition
\ref{prop-quadratic} and Corollary \ref{Cor-quadratic-form}. Then by the proof
of Theorem \ref{thm:stability}, we get the same nonlinear stability criterion
for traveling waves of (\ref{eqn-GP-generalized}). That is,

\begin{theorem}
\label{thm:general-stability} Assume (F1-2). For $0<c<\sqrt{2}$, let $U_{c}$
be a traveling wave solution of (\ref{eqn-GP-generalized}) constructed in
\cite{Maris-annal}. Assume the (\ref{assumption-NDG}) type non-degeneracy
condition:
\[
ker(L_{c})=span\{\partial_{x_{j}}U_{c}\mid j=1,2,3\}\text{.}%
\]
Then the traveling wave $U_{c}$ satisfying $\frac{\partial P(U_{c})}{\partial
c}|_{c=c_{0}}>0\ $is orbitally stable in the same sense (in terms of the
distance $d_{1}$) as in Theorem \ref{thm:stability}.
\end{theorem}

In fact, the above theorem also holds for some cases when $p_{0}=2$ in the
assumption (F2). More precisely, assume \newline\newline(F2') There exists
$C,\alpha_{0},s_{0}>0$, and $0<p_{1}\leq1\leq p_{0}\leq2$, such that
$|F^{\prime}(s)|\leq C(1+s^{p_{1}-1}+s^{p_{0}-1})$ for all $s\geq0$ and
$F(s)\leq-Cs^{\alpha_{0}}$ for all $s>s_{0}$.

\begin{corollary}
\label{C:critical} Assume (F1) and (F2'). For $0<c<\sqrt{2}$, let $U_{c}$ be a
traveling wave solution of (\ref{eqn-GP-generalized}) constructed in
\cite{Maris-annal}, satisfying $\frac{\partial P(U_{c})}{\partial c}%
|_{c=c_{0}}>0$. Assume the (\ref{assumption-NDG}) type non-degeneracy
condition: $ker(L_{c})= span\{\partial_{x_{j}} U_{c}\mid j=1,2,3\}$. Then the
traveling wave $U_{c}$ is orbitally stable.
\end{corollary}

\begin{remark}
This corollary applies to the cubic-quintic nonlinear Schr\"{o}dinger equation
where the nonlinearity corresponds to
\[
F(s)=-\alpha_{1}+\alpha_{3}s-\alpha_{5}s^{2},\quad\alpha_{1,2,3}>0.
\]
For 3D, the cubic-quintic equation is critical and its global existence in the
energy space was shown recently in \cite{killip-et}. For dimension $n\leq4$
and rather general subcritical nonlinear terms, the global existence in the
energy space was shown in \cite{gallo-2008}.
\end{remark}

\begin{remark}
The decay property (\ref{decay-TW-g}) for traveling waves was proved for (GP)
equation in \cite{Gravejat}. It seems possible to use the arguments of
\cite{Gravejat} to get the same decay (\ref{decay-TW-g}) for general nonlinear terms.
\end{remark}

In fact, if $p_{0}=2$, the energy and momentum functional $E$ and $P$ are
still $C^{2}$ on $X_{1}$. Supposed $U_{c}$ is a traveling wave, i.e. a
critical point of the energy-momentum functional $E_{c}$, such that
$E_{c}^{\prime\prime}(U_{c})$ is uniformly positive as in the sense of Lemma
\ref{lemma-positive-quadratic}, then the same proof as the one of Theorem
\ref{thm:stability} applies and we obtain the orbital stability of $U_{c}$.

In assumption (F2), $p_{0}<2$ is assumed so that the existence of traveling
waves is obtained through a constrained minimization approach as in Theorem
1.1 in \cite{Maris-annal}, where the compactness of the embedding is needed.

Fortunately, with assumptions (F1) and (F2'), Corollary 1.2 in
\cite{Maris-annal} applies and thus traveling waves exist through constrained
minimization. The idea is that (F2') allows us to carefully modify the
nonlinearity $F$ to $F_{M}$ such that
\[
F_{M}(s)=F(s),\;\forall\ s\in\lbrack0,s_{1}],\quad F_{M}(s)=-C_{1}s^{\beta
}\;\forall\ s\geq s_{2}%
\]
where $C_{1},\beta,s_{1},s_{2}$ are some constants satisfying $s_{1}\geq
s_{0}$, $s_{2}>>s_{1}$, and $\beta\in(0,2)$. The construction of $F_{M}$
ensures (F1-2) are satisfied, which implies the existence of a constrained
minimizer $U_{c}$ of the energy-momentum functional $E_{c,M}$ associated to
$F_{M}$ and $L_{c,M}\triangleq E_{c,M}^{\prime\prime}(U_{c})$ can be analyzed
as in the above. Moreover, one can prove that the range of $U_{c}$ is
contained in $[0,s_{1}]$. Therefore, $U_{c}$ is also a traveling wave of the
original equation. More details on the existence through the calculus of
variation can be found in \cite{Maris-annal}. Finally due to the fact
$E_{c}^{\prime\prime}(U_{c})=E_{c,M}^{\prime\prime}(U_{c})$ as $F_{M}=F$ on
the range of $U_{c}$, we obtain the uniform positivity of $E_{c}^{\prime
\prime}(U_{c})$ in the sense of Lemma \ref{lemma-positive-quadratic} and the
nonlinear stability follows subsequently.


\section{Instability of traveling waves on the upper branch}

In this section, we prove instability of $3D$ traveling waves obtained via a
constrained variational approach when $\frac{\partial P(U_{c})}{\partial
c}|_{c=c_{0}}<0$. First, we prove linear instability by studying the
linearized problem. Then, instead of passing linear instability to nonlinear
instability, we will prove a much stronger statement by constructing stable
and unstable manifolds near the unstable traveling waves.

\subsection{Linear instability}

In the traveling frame $\left(  t,x-ce_{1}t\right)  $, the nonlinear equation
(\ref{eqn-GP-generalized}) becomes
\begin{equation}
i\partial_{t}U-ic\partial_{x_{1}}U+\Delta U+F(|U|^{2})U=0, \label{eqn-GP-g-TF}%
\end{equation}
where $u\left(  t,x\right)  =U\left(  t,x-ce_{1}t\right)  $.

Near the traveling wave solution $U_{c}=u_{c}+iv_{c}$ satisfying
(\ref{eqn-GP-G-TW}), the linearized equation can be written as
\begin{equation}
\partial_{t}\left(
\begin{array}
[c]{c}%
u_{1}\\
u_{2}%
\end{array}
\right)  =JL_{c}\left(
\begin{array}
[c]{c}%
u_{1}\\
u_{2}%
\end{array}
\right)  , \label{eqn-linearized-u}%
\end{equation}
where
\[
J=\left(
\begin{array}
[c]{cc}%
0 & 1\\
-1 & 0
\end{array}
\right)  ,
\]
and $L_{c}$ is defined by (\ref{operator-Lc-g}).

We construct invariant manifolds by using the nonlinear equation for $w\in
X_{1}$, where $u=\psi\left(  w\right)  $ satisfies the (GP) equation. The
reason is two-fold. First, we need to use the spectral properties of the
quadratic form $\left\langle L_{c}\cdot,\cdot\right\rangle \ $in the space
$X_{1}\ $(Proposition \ref{prop-quadratic}) to prove the exponential dichotomy
of the semigroup $e^{tJL_{c}}~$in Lemma \ref{lemma-dichotomy-U} below. Second,
to ensure that the constructed invariant manifolds lie in the energy space
(See Remark \ref{rmk-reason-coordinate-instability}). Denote $U_{c}%
=\psi\left(  w_{c}\right)  $ and $w_{c}=w_{1c}+iw_{2c}$. Let
\begin{align}
U  &  =\psi(w_{1c}+w_{1},w_{2c}+w_{2})\label{eqn-U}\\
&  =U_{c}+w_{1}-\chi(D)(w_{2c}w_{2}+\frac{w_{2}^{2}}{2})+iw_{2}.\nonumber
\end{align}
Plugging (\ref{eqn-U}) into (\ref{eqn-GP-g-TF}), we get
\begin{align}
\partial_{t}w_{2}  &  =\Delta w_{1}-\Delta\chi(D)(w_{2c}w_{2}+\frac{w_{2}^{2}%
}{2})+c\partial_{x_{1}}w_{2}+[F(|U|^{2})-F(|U_{c}|^{2})]u_{c}\label{U-4}\\
&  +F(|U|^{2})[w_{1}-\chi(D)(w_{2c}w_{2}+\frac{w_{2}^{2}}{2})],\nonumber
\end{align}

\begin{align}
\partial_{t}w_{1}  &  =-\Delta w_{2}+\chi(D)((w_{2c}+w_{2})\partial_{t}%
w_{2})+c\partial_{x_{1}}[w_{1}-\chi(D)(w_{2c}w_{2}+\frac{w_{2}^{2}}%
{2})]\label{U-5}\\
&  +[F(|U_{c}|^{2})-F(|U|^{2})]v_{c}-F(|U|^{2})w_{2}.\nonumber
\end{align}
The above two equations can be written as%
\begin{equation}
i\partial_{t}w-ic\partial_{x_{1}}w+\Delta w=\Psi(w),
\label{eqn-nonlinear--TF-w}%
\end{equation}
where
\begin{align}
\operatorname{Re}\Psi(w)  &  =\Delta\chi(D)(w_{2c}w_{2}+\frac{w_{2}^{2}}%
{2})+[F(|U_{c}|^{2})-F(|U|^{2})]u_{c}\label{U-7}\\
&  -F(|U|^{2})[w_{1}-\chi(D)(w_{2c}w_{2}+\frac{w_{2}^{2}}{2})],\nonumber
\end{align}
and
\begin{align}
\operatorname{Im}\Psi(w)  &  =\chi(D)((w_{2c}+w_{2})\partial_{t}%
w_{2})-c\partial_{x_{1}}\chi(D)(w_{2c}w_{2}+\frac{w_{2}^{2}}{2})\label{U-8}\\
&  +[F(|U_{c}|^{2})-F(|U|^{2})]v_{c}-F(|U|^{2})w_{2}.\nonumber
\end{align}
Instead of linearizing the nonlinear term $\Psi(w)$ at $w=0\ $directly, we
derive the linearized equation of (\ref{eqn-nonlinear--TF-w}) by relating it
with the linearized equation (\ref{eqn-linearized-u}) for $u$. The
linearization of the coordinate mapping $u=\psi\left(  w\right)  \ $at $w_{c}$
yields $u=Kw$, where $K$ is defined by (\ref{definition-K}). Thus, the
linearized equation of (\ref{eqn-nonlinear--TF-w}) at $w=0$ takes the form
\begin{equation}
\partial_{t}\left(
\begin{array}
[c]{c}%
w_{1}\\
w_{2}%
\end{array}
\right)  =K^{-1}JL_{c}K\left(
\begin{array}
[c]{c}%
w_{1}\\
w_{2}%
\end{array}
\right)  \label{eqn-linearized-gp-w}%
\end{equation}
which implies that
\begin{equation}
\left(
\begin{array}
[c]{c}%
w_{1}\\
w_{2}%
\end{array}
\right)  \left(  t\right)  =K^{-1}e^{tJL_{c}}K\left(
\begin{array}
[c]{c}%
w_{1}\\
w_{2}%
\end{array}
\right)  \left(  0\right)  . \label{relation-linearized-eqn}%
\end{equation}
So it suffices to study the spectrum of $JL_{c}$ and the semigroup
$e^{tJL_{c}}$. Note that the traveling wave $U_{c}$ is cylindrical symmetric,
that is, $U_{c}=U_{c}\left(  x_{1},\left\vert x^{\perp}\right\vert \right)  $
with $x^{\perp}=\left(  x_{2},x_{3}\right)  $. We will prove linear
instability in $X_{1}^{s}$, the cylindrical symmetric subspace of $X_{1}$.
Assume the non-degeneracy condition in the cylindrical symmetric space, that
is,
\begin{equation}
\ker L_{c}\cap X_{1}^{s}=\left\{  \partial_{x_{1}}U_{c}\right\}  .
\label{assumption-NDG-s}%
\end{equation}
We have the following analogue of Proposition \ref{prop-quadratic}.

\begin{proposition}
\label{P-decom1} For $0<c<\sqrt{2}$, let $U_{c}=\psi\left(  w_{c}\right)
\ $be a traveling wave solution of (\ref{eqn-GP-generalized}) constructed in
\cite{Maris-annal} and $L_{c}$ be the operator defined by (\ref{operator-Lc-g}%
). Assume (\ref{assumption-NDG-s}). The space $X_{1}^{s}$ is decomposed as a
direct sum%
\[
X_{1}^{s}=N \oplus Z \oplus P,
\]
where $Z=\left\{  \partial_{x_{1}}U_{c}\right\}  $, $N$ is a one-dimensional
subspace such that $\left\langle L_{c}u,u\right\rangle <0$ for $0\neq u\in N$,
and $P$ is a closed subspace such that
\[
\left\langle L_{c}u,u\right\rangle \geq\delta\left\Vert u\right\Vert _{X_{1}%
}^{2}\ \ \text{ for any }u\in P,
\]
for some constant $\delta>0.$
\end{proposition}

The proof is the same as that of Proposition \ref{prop-quadratic}, by
observing that the negative mode constructed in Lemma \ref{le-4} is
cylindrical symmetric. Now we show the linear instability of traveling waves
on the upper branch.

\begin{proposition}
\label{P-linear-insta-G} Let $U_{c}$, $c \in[c_{1}, c_{2}] \subset(0, \sqrt
{2})$, be a $C^{1}$ (with respect to the wave speed $c$) family of traveling
waves of (\ref{eqn-GP-generalized}) in the energy space $X_{0}$. For $c_{0}
\in(c_{1}, c_{2})$, assume

\begin{enumerate}
\item $F \in C^{1}$ on $U_{c_{0}} (\mathbf{R}^{n})$;

\item $L_{c_{0}}$ satisfies (\ref{assumption-NDG-s});

\item $\frac{\partial P(U_{c})}{\partial c}|_{c=c_{0}}<0$;
\end{enumerate}

then there exists $0\neq w_{u}\in X_{1}^{s}$ and $\lambda_{u}>0$, such that
$e^{\lambda_{u}t}w_{u}\left(  x\right)  $ is a solution of
(\ref{eqn-linearized-gp-w}).
\end{proposition}

In particular, this proposition applies to those traveling waves obtained in
\cite{Maris-annal} via a constrained variational approach.

\begin{proposition}
\label{prop-linear-insta} Assume (F1-2) or (F1)-(F2'). For $0<c_{0}<\sqrt{2}$,
let $U_{c_{0}}\ $be a traveling wave solution of (\ref{eqn-GP-generalized})
constructed in \cite{Maris-annal}, satisfying $\frac{\partial P(U_{c}%
)}{\partial c}|_{c=c_{0}}<0$. Assume (\ref{assumption-NDG-s}). Then there
exists a linearly unstable mode of (\ref{eqn-linearized-gp-w}). That is, there
exists $0\neq w_{u}\in X_{1}^{s}$ and $\lambda_{u}>0$, such that
$e^{\lambda_{u}t}w_{u}\left(  x\right)  $ is a solution of
(\ref{eqn-linearized-gp-w}).
\end{proposition}

\begin{proof}
\textbf{of Proposition \ref{P-linear-insta-G}}: By
(\ref{relation-linearized-eqn}), it suffices to show that the operator
$JL_{c_{0}}$ has an unstable eigenvalue in the space $X_{1}^{s}$. The proof is
to modify that of Theorem 5.1 in \cite{gss90}, as explained in Remark
\ref{rmk-g-kernel} below. Define the following subspace of $X_{1}^{s}$ by
\begin{equation}
Y_{1}^{s}=\left\{  u\in X_{1}^{s}\ |\ \left\langle u,J^{-1}\partial_{x_{1}%
}U_{c_{0}}\right\rangle =\left\langle u,J^{-1}\partial_{c}U_{c}|_{c=c_{0}%
}\right\rangle =0\right\}  . \label{definition-space-Y-1-s}%
\end{equation}
We show that the quadratic form $\left\langle L_{c_{0}}\cdot,\cdot
\right\rangle $ restricted to $Y_{1}^{s}$ is non-degenerate. Indeed, any $u\in
X_{1}^{s}$ can be uniquely written as
\begin{equation}
u=a\partial_{x_{1}}U_{c_{0}}+b\partial_{c}U_{c}|_{c=c_{0}}%
+v,\ \label{eqn-decomposition-g-kernel}%
\end{equation}
where $v\in Y_{1}^{s},$
\begin{equation}
a=-\left\langle u,J^{-1}\partial_{c}U_{c}|_{c=c_{0}}\right\rangle
/\frac{\partial P(U_{c})}{\partial c}|_{c=c_{0}}, \label{defn-a}%
\end{equation}
and
\begin{equation}
b=\left\langle u,J^{-1}\partial_{x_{1}}U_{c_{0}}\right\rangle /\frac{\partial
P(U_{c})}{\partial c}|_{c=c_{0}}. \label{defn-b}%
\end{equation}
Here, we use the identity%
\begin{equation}
L_{c}\partial_{c}U_{c}=-P^{\prime}(U_{c})=-J^{-1}\partial_{x_{1}}U_{c}.
\label{relation-Lc-par-c-U-c}%
\end{equation}
Suppose $\left\langle L_{c_{0}}\cdot,\cdot\right\rangle $ is degenerate on
$Y_{1}^{s}$, then there exists $\phi\in Y_{1}^{s}$ such that $\left\langle
L_{c_{0}}\phi,v\right\rangle =0$ for any $v\in Y_{1}^{s}$. This implies that
$\left\langle L_{c_{0}}\phi,u\right\rangle =0$ for any $u\in X_{1}^{s}$, by
the decomposition (\ref{eqn-decomposition-g-kernel}). So $\phi\in\ker
L_{c}\cap X_{1}^{s}$ and by assumption (\ref{assumption-NDG-s})$,\phi
=c\partial_{x_{1}}U_{c_{0}}$ which implies $\phi=0$ since $\partial_{x_{1}%
}U_{c=c_{0}}\notin Y_{1}^{s}.$

Moreover, since $L_{c_{0}}\partial_{x_{1}}U_{c_{0}}=0$,
\eqref{relation-Lc-par-c-U-c} and the definition of $Y_{1}^{s}$ imply a.) the
splitting of $X_{1}^{s}$ into $Y_{1}^{s}$ and $span\{\partial_{x_{1}}U_{c_{0}%
},\ \partial_{c}U_{c}|_{c=c_{0}}\}$ is orthogonal with respect to the
quadratic form $L_{c_{0}}$ and b.) $span\{\partial_{x_{1}}U_{c_{0}}%
,\ \partial_{c}U_{c}|_{c=c_{0}}\}$ is invariant under $JL_{c_{0}}$ and thus so
is $Y_{1}^{s}$, which also imply their invariance under the linearized flow
$e^{tJL_{c_{0}}}$.

Denote $n\left(  L_{c_{0}}|_{X}\right)  $ to be the number of negative modes
of the quadratic form $\left\langle L_{c_{0}}\cdot,\cdot\right\rangle $
restricted to a subspace $X\subset X_{1}^{s}$. We show that $n\left(
L_{c_{0}}|_{Y_{1}^{s}}\right)  =1$. Indeed, for any $u\in X_{1}^{s}$, by
(\ref{eqn-decomposition-g-kernel}) and (\ref{relation-Lc-par-c-U-c}), we have
\[
\left\langle L_{c_{0}}u,u\right\rangle =b^{2}\left\langle L_{c_{0}}%
\partial_{c}U_{c}|_{c=c_{0}},\partial_{c}U_{c}|_{c=c_{0}}\right\rangle
+\left\langle Lv,v\right\rangle .
\]
Since $n\left(  L_{c_{0}}|_{X_{1}^{s}}\right)  =1$ and
\[
\left\langle L_{c_{0}}\partial_{c}U_{c}|_{c=c_{0}},\partial_{c}U_{c}%
|_{c=c_{0}}\right\rangle =-\frac{\partial P(U_{c})}{\partial c}|_{c=c_{0}}>0,
\]
so $n\left(  L_{c_{0}}|_{Y_{1}^{s}}\right)  =1$. Let $Y_{1}^{s}=N\oplus P$,
where on $P$ and $N$, the quadratic form $\left\langle L_{c_{0}}\cdot
,\cdot\right\rangle $ is positive and negative definite respectively, $\dim
N=1$, and $N,P$ are orthogonal in the inner product $\left[  \cdot
,\cdot\right]  :=\left\langle L_{c_{0}}\cdot,\cdot\right\rangle $.

It can be verified that
\[
D\left(  JL_{c_{0}}\right)  =D\left(  L_{c_{0}}\right)  =X_{3}.
\]
Indeed, $JL_{c_{0}},L_{c_{0}}:X_{3}\rightarrow H^{1}$. Since $Y_{1}^{s}$ is
separable, there is an increasing sequence of subspaces $P^{\left(  n\right)
}\subset P$ of odd dimension $n$ such that $\cup\ X^{\left(  n\right)  }$ is
dense in $Y_{1}^{s}$, where $X^{\left(  n\right)  }=N+P^{\left(  n\right)  }$.
We can choose $N\ $\ and $P^{\left(  n\right)  }$to lie in $X_{3}$. Denote by
$\pi^{-},\ \pi^{+}$ and $\pi^{\left(  n\right)  }$ the orthogonal projections
of $Y_{1}^{s}$ to $N,\ P$ and $X^{\left(  n\right)  }$ respectively in the
inner product $\left[  \cdot,\cdot\right]  $. Consider the set
\[
\mathcal{C}=\left\{  u\in Y_{1}^{s}\ |\ \left[  \pi^{-}u,u\right]
=-1,\ \left\langle L_{c_{0}}u,u\right\rangle =0\right\}
\]
and $\mathcal{C}^{n}=\mathcal{C}\cap X^{\left(  n\right)  }$. For $v\in
X^{\left(  n\right)  }$, consider the mapping%
\begin{equation}
f_{n}\left(  v\right)  =\pi^{\left(  n\right)  }\left(  JL_{c_{0}}v\right)
+\left[  \pi^{-}\left(  JL_{c_{0}}v\right)  ,v\right]  v. \label{defn-f-n}%
\end{equation}
In the above definition, we use the observation that $JL_{c_{0}}v\in Y_{1}%
^{s}$ for any $v\in Y_{1}^{s}$. It is easy to check that $\left[  \pi^{-}%
f_{n}\left(  v\right)  ,v\right]  =0$ and
\begin{align*}
\left\langle f_{n}\left(  v\right)  ,L_{c_{0}}v\right\rangle  &  =\left[
f_{n}\left(  v\right)  ,v\right]  =\left[  JL_{c_{0}}v,v\right]  +\left[
\pi^{-}\left(  JL_{c_{0}}v\right)  ,v\right]  \left[  v,v\right] \\
&  =\left\langle JL_{c_{0}}v,L_{c_{0}}v\right\rangle +\left[  \pi^{-}\left(
JL_{c_{0}}v\right)  ,v\right]  \left\langle v,L_{c_{0}}v\right\rangle =0.
\end{align*}
Therefore, $f_{n}$ is a tangent vector field on the manifold $\mathcal{C}^{n}%
$, which is the union of two spheres $S^{n-1}$ and thus has non-vanishing
Euler characteristic. Thus $f_{n}$ must vanish at some $y_{n}\in
\mathcal{C}^{n}$. That is, there is a real scalar$\ a_{n}=-\left[  \pi
^{-}\left(  JL_{c_{0}}v_{n}\right)  ,v_{n}\right]  $, such that
\begin{equation}
\left[  JL_{c_{0}}y_{n},w\right]  =a_{n}\left[  y_{n},w\right]  ,\ \text{for
any }w\in X^{\left(  n\right)  }. \label{eqn-n-eigen}%
\end{equation}
Let $y_{n}=y_{n}^{-}+y_{n}^{+}$, where $y_{n}^{-}\in N$ and $y_{n}^{-}\in
P^{\left(  n\right)  }$. Let $y_{n}^{-}=b_{n}\chi_{-}$ with $\left\langle
L_{c_{0}}\chi_{-},\chi_{-}\right\rangle =-1$, then $\left[  y_{n}^{-}%
,y_{n}^{-}\right]  =-1$ implies that $\left\vert b_{n}\right\vert =1$. We can
normalize $b_{n}=1$. Since
\[
0=\left\langle L_{c_{0}}y_{n},y_{n}\right\rangle =-1+\left\langle L_{c_{0}%
}y_{n}^{+},y_{n}^{+}\right\rangle
\]
and $\left\langle L_{c_{0}}\cdot,\cdot\right\rangle |_{P}$ is positive$,$
$\left\Vert y_{n}^{+}\right\Vert _{X_{1}}$ is uniformly bounded. So
$y_{n}\rightharpoonup y_{\infty}\in Y_{1}^{s}\ $weakly in $X_{1}$. We note
that $y_{\infty}\neq0$ since $\pi^{-}y_{\infty}=\chi_{-}\neq0$. We claim that
$\left\{  a_{n}\right\}  $ is bounded. Suppose otherwise, $a_{n}%
\rightarrow\infty$ when $n\rightarrow\infty$. For any integer $k\in\mathbf{N}$
and a fixed $w\in X^{\left(  k\right)  }$, when $n\geq k$, by
(\ref{eqn-n-eigen}) we have
\begin{align}
\left[  y_{n},w\right]   &  =\frac{1}{a_{n}}\left[  JL_{c_{0}}y_{n},w\right]
=\frac{1}{a_{n}}\left\langle JL_{c_{0}}y_{n},L_{c_{0}}w\right\rangle
\label{eqn-n-eigen-another}\\
&  =-\frac{1}{a_{n}}\left\langle L_{c_{0}}y_{n},JL_{c_{0}}w\right\rangle
=-\frac{1}{a_{n}}\left[  y_{n},JL_{c_{0}}w\right]  .\nonumber
\end{align}
Let $n\rightarrow\infty$ in (\ref{eqn-n-eigen-another}), we have $\left[
y_{\infty},w\right]  =0$. By the density argument, this is also true for any
$w\in Y_{1}^{s}$ and thus $y_{\infty}=0$ since $\left[  \cdot,\cdot\right]  $
is non-degenerate on $Y_{1}^{s}$. This contradiction shows that $\left\{
a_{n}\right\}  $ is bounded. So we can pick a subsequence $\left\{
n_{k}\right\}  \ $such that $a_{n_{k}}\rightarrow a$, for some $a\in
\mathbf{R}.$ For convenience, we still denote the subsequence by $a_{n}$. By
(\ref{eqn-n-eigen}),
\[
-\left[  y_{n},JL_{c_{0}}w\right]  =\left[  JL_{c_{0}}y_{n},w\right]
=a_{n}\left[  y_{n},w\right]  .
\]
Passing to the limit of above, we have
\[
-\left[  y_{\infty},JL_{c_{0}}w\right]  =a\left[  y_{\infty},w\right]  ,
\]
for any fixed $w\in X^{\left(  k\right)  }$. For any $v\in X_{3}\cap Y_{1}%
^{s}$, $JL_{c_{0}}v\in X_{1}$ and by density argument, we have
\begin{equation}
-\left[  y_{\infty},JL_{c_{0}}v\right]  =a\left[  y_{\infty},v\right]  .
\label{eqn-weak-solution-eigen}%
\end{equation}
It is easy to see that (\ref{eqn-weak-solution-eigen}) is also satisfied when
$v\in\left\{  \partial_{x_{1}}U_{c_{0}},\partial_{c}U_{c}|_{c=c_{0}}\right\}
$. Thus by the decomposition (\ref{eqn-decomposition-g-kernel}), the equation
(\ref{eqn-weak-solution-eigen}) is satisfied for any $v\in X_{3}\cap X_{1}%
^{s}$. So
\[
-\left\langle y_{\infty},L_{c_{0}}Jw\right\rangle =a\left\langle y_{\infty
},w\right\rangle
\]
for any $w\in R\left(  L_{c_{0}}\right)  $ which is the orthogonal complement
of $\partial_{x_{1}}U_{c}$. Thus, there exists a constant $d$, such that
$y_{\infty}\in Y_{1}^{s}$ is the weak solution of the equation
\begin{equation}
JL_{c_{0}}y_{\infty}=ay_{\infty}+d\partial_{x_{1}}U_{c_{0}}.
\label{eqn-y-infinity}%
\end{equation}
We must have $a\neq0$, since $0\neq y_{\infty}\in Y_{1}^{s}$ and $Y_{1}%
^{s}\cap\left\{  \partial_{x_{1}}U_{c_{0}},\partial_{c}U_{c}|_{c=c_{0}%
}\right\}  =\varnothing$. By elliptic regularity, $y_{\infty}\in D\left(
JL_{c_{0}}\right)  =X_{3}$ and then $y_{\infty}=\frac{1}{a}\left(  JL_{c_{0}%
}y_{\infty}-\frac{d}{a}\partial_{x_{1}}U_{c_{0}}\right)  \in H^{1}$. So
$y_{\infty}\in H^{3}$. If$\ U_{c}\in1+H^{k}$ for some integer $k$, then it can
be shown that $y_{\infty}\in H^{k}$. Since (\ref{eqn-y-infinity}) implies
that
\[
JL_{c_{0}}\left(  y_{\infty}+\frac{d}{a}\partial_{x_{1}}U_{c_{0}}\right)
=a\left(  y_{\infty}+\frac{d}{a}\partial_{x_{1}}U_{c_{0}}\right)  ,
\]
so $a\neq0$ is an eigenvalue of $JL_{c_{0}}$.

For any nonzero eigenvalue $\lambda$ of $JL_{c_{0}}$ with an eigenfunction
$y$, we must have $u=L_{c_{0}}y\neq0$.\ So we obtain from $L_{c_{0}}Ju=\lambda
u$ that $\lambda$ is also an eigenvalue of $L_{c_{0}}J=-(JL_{c_{0}})^{\ast}$.
Therefore $-\lambda$ is an eigenvalue of $JL_{c_{0}}$ as well. This and the
above argument imply that $\pm a$ are eigenvalues of $JL_{c_{0}}$. This
finishes the proof that $JL_{c_{0}}$ must have a positive eigenvalue.
\end{proof}

\begin{remark}
By the above proof, there also exists a stable eigenvalue $\lambda_{s}<0$ of
$JL_{c_{0}}\ $which gives an exponentially decaying solution $e^{\lambda_{s}%
t}w_{s}(x)$ $(w_{s}(x)\in X_{1}^{s})$ of the linearized equation
(\ref{eqn-linearized-u}). This is due to the Hamiltonian nature of the equation.
\end{remark}

\begin{remark}
\label{rmk-g-kernel}The invariant subspace $Y_{1}^{s}$ is used to remove the
generalized kernel $\left\{  \partial_{x_{1}}U_{c_{0}},\partial_{c}U_{c_{0}%
}\right\}  $ of $L_{c_{0}}$ in $X_{1}^{s}$. This space also plays an important
role in proving the exponential dichotomy of the semigroup $e^{tJL_{c_{0}}}$
below. In \cite{gss87} \cite{gss90}, a general theory was developed for
studying stability of standing waves (traveling waves etc.) of an abstract
Hamiltonian PDE $\frac{du}{dt}=JE^{\prime}\left(  u\right)  $. In this
framework, the symplectic operator $J$ should be invertible in the sense that
$J^{-1}:X\rightarrow X^{\ast}\,$is bounded, where $X$ is the energy space. In
our case, the space $X$ \ is $X_{1}=H^{1}\times\dot{H}^{1},$ $X$ $^{\ast}$ is
$X_{1}^{\ast}=H^{-1}\times\dot{H}^{-1}$ and the operator
\[
J^{-1}=-J=\left(
\begin{array}
[c]{cc}%
0 & -1\\
1 & 0
\end{array}
\right)  .
\]
So $J^{-1}:X_{1}\rightarrow X_{1}^{\ast}\ $is not bounded since $\dot{H}%
^{1}\varsubsetneq H^{-1}$ and we can not apply the theory of \cite{gss87}
\cite{gss90} directly. In \cite{lopes}, an abstract theorem was given for the
case when $J$ is not onto. However, as also commented in \cite{chiron-1d}, it
would take substantial effort to verify some of the assumptions in
\cite{lopes}, particularly the semigroup estimates, for our current case and
the instability of slow traveling waves in Section $5$.

To handle this issue, we modify the proof of linear instability in
\cite{gss90} (Theorem 5.1) to avoid using the invertibility of $J$. Our
modified argument could be used for general Hamiltonian PDEs with a
non-invertible symplectic operator. We do not need to assume the semigroup
estimates as in \cite{lopes}.
\end{remark}

\subsection{Linear exponential dichotomy of semigroup}

To construct invariant manifolds, the first step is to establish the
exponential dichotomy of the linearized semigroup. First, we prove this for
the semigroup generated by $JL_{c}$.

\begin{lemma}
\label{lemma-dichotomy-U}For $0<c<\sqrt{2}$, let $U_{c}\ $be a traveling wave
solution of (\ref{eqn-GP-generalized}) constructed in \cite{Maris-annal} and
$L_{c}$ be the operator defined by (\ref{operator-Lc-g}). Assume
(\ref{assumption-NDG-s}) and $\frac{\partial P(U_{c})}{\partial c}<0$. The
space $X_{1}^{s}$ is decomposed as a direct sum%
\begin{equation}
X_{1}^{s}=E^{u}\oplus E^{cs}, \label{direct-sum-1}%
\end{equation}
satisfying: i) Both $E^{u}=span\left\{  w_{u}\right\}  $ and $E^{cs}$ are
invariant under the linear semigroup $e^{tJL_{c}}$. ii) there exist constants
$M>0$ and $\lambda_{u}>0$, such that
\[
\left\vert e^{tJL_{c}}|_{E^{cs}}\right\vert _{X_{1}}\leq M(1+t),\quad
\forall\;t\geq0\quad\text{ and }\quad|e^{tJL_{c}}|_{E^{u}}|_{X_{1}}\leq
Me^{\lambda_{u}t},\quad\forall\;t\leq0.
\]

\end{lemma}

\begin{proof}
Let $w_{u},w_{s}\in X_{1}^{s}$ be the unstable and stable eigenfunctions of
$JL_{c}$ as constructed in Proposition \ref{prop-linear-insta} and its
subsequent remark. Denote
\[
E^{s}=span\left\{  w_{s}\right\}  , \quad E^{u}=span\left\{  w_{u}\right\}  ,
\quad E^{us}=span\left\{  w_{u},w_{s}\right\}  .
\]
First, we claim that: $\left\langle L_{c}w_{u},w_{s}\right\rangle \neq0$ and
the quadratic form $\left\langle L_{c}\cdot,\cdot\right\rangle |_{E^{us}}%
\ $has one positive and one negative mode. Suppose otherwise $\left\langle
L_{c}w_{u},w_{s}\right\rangle =0$, then $\left\langle L_{c}\cdot
,\cdot\right\rangle |_{E^{us}}$ is identically zero since $\left\langle
Lw_{u},w_{u}\right\rangle =\left\langle Lw_{s},w_{s}\right\rangle =0$ due to
the skew-symmetry of $J$. By Proposition \ref{prop-linear-insta},
$\left\langle L_{c}\cdot,\cdot\right\rangle |_{Y_{1}^{s}}$ is non-degenerate
and has exactly one negative mode. Let $Y_{1}^{s}=N\oplus P$ be such that
$N=\left\{  u_{-}\right\}  \ $with $\left\langle L_{c}u_{-},u_{-}\right\rangle
<0$ and $\left\langle L_{c}\cdot,\cdot\right\rangle |_{P}>0$. Then we can
decompose $w_{u}=a_{1}u_{-}+b_{1}p_{1}$ and $w_{s}=a_{2}u_{-}+b_{2}p_{2}$
where $p_{1},p_{2}\in P$. Since $a_{1},a_{2}\neq0$, we define $\tilde{w}$
$=w_{u}-\frac{a_{2}}{a_{1}}w_{s}\in E^{us}\cap P$. This is a contradiction
since $\tilde{w}\neq0$. Thus $\left\langle L_{c}\cdot,\cdot\right\rangle
|_{E^{us}}$ is represented by the $2\times2$ matrix
\[
\left(
\begin{array}
[c]{cc}%
0 & \left\langle L_{c}w_{u},w_{s}\right\rangle \\
\left\langle L_{c}w_{u},w_{s}\right\rangle  & 0
\end{array}
\right)  ,
\]
which has one positive and one negative mode.

We define the subspace $E^{e}$ by
\[
E^{e}=\left\{  u\in Y_{1}^{s}|\ \left\langle u,L_{c}w_{u}\right\rangle
=\left\langle u,L_{c}w_{s}\right\rangle =0\right\}  .
\]
Let $E_{g}^{\ker}=span\left\{  \partial_{x_{1}}U_{c_{0}},\partial_{c}%
U_{c}\right\}  $ be the generalized kernel of $JL_{c}$ in $X_{1}^{s}$. For any
two solutions $u\left(  t\right)  ,v\left(  t\right)  $ of the linearized
equation $du/dt=JL_{c}u$, the quadratic form $\left\langle u\left(  t\right)
,L_{c}v\left(  t\right)  \right\rangle $ is independent of $t$, since%
\begin{align*}
\frac{d}{dt}\left\langle u\left(  t\right)  ,L_{c}v\left(  t\right)
\right\rangle  &  =\left\langle JL_{c}u,L_{c}v\right\rangle +\left\langle
u,L_{c}JL_{c}v\right\rangle \\
&  =\left\langle JL_{c}u,L_{c}v\right\rangle +\left\langle L_{c}%
u,JL_{c}v\right\rangle =0.
\end{align*}
By using this observation and the invariance of $Y_{1}^{s}$ and $E^{us}$, it
is easy to show that the subspace $E^{e}$ is invariant under the semigroup
$e^{tJL_{c}}$. Furthermore, we show that there exists $C>0$, such that for any
$u\in E^{e}$ and $t\in\mathbf{R,}$
\begin{equation}
\left\Vert e^{tJL_{c}}u\right\Vert _{X_{1}}\leq C\left\Vert u\right\Vert
_{X_{1}}. \label{estimate-elliptic}%
\end{equation}
In fact, we note that any $v\in Y_{1}^{s}$ can be decomposed as
\begin{equation}
v=c_{u}w_{u}+c_{s}w_{s}+v_{1}, \label{decomposition-v}%
\end{equation}
where
\[
c_{u}=\left\langle L_{c}v,w_{s}\right\rangle /\left\langle L_{c}w_{u}%
,w_{s}\right\rangle ,\ c_{s}=\left\langle L_{c}v,w_{u}\right\rangle
/\left\langle L_{c}w_{u},w_{s}\right\rangle
\]
and $v_{1}\in E^{e}$. Thus we have
\[
\left\langle L_{c}\cdot,\cdot\right\rangle |_{Y_{1}^{s}}=\left\langle
L_{c}\cdot,\cdot\right\rangle |_{E^{us}}+\left\langle L_{c}\cdot
,\cdot\right\rangle |_{E^{e}},
\]
and a counting of negative modes on both sides shows that $\left\langle
L_{c}\cdot,\cdot\right\rangle |_{E^{e}}>0$. Then the estimate
(\ref{estimate-elliptic}) follows by the invariance of the quadratic form
$\left\langle L_{c}\cdot,\cdot\right\rangle $. Combining the decompositions
(\ref{decomposition-v}) and (\ref{eqn-decomposition-g-kernel}), we have%
\[
X_{1}^{s}=E^{u}\oplus E^{s}\oplus E^{e}\oplus E_{g}^{\ker}=E^{u}\oplus
E^{cs},
\]
where
\[
E^{cs}=E^{s}\oplus E^{e}\oplus E_{g}^{\ker}.
\]
For any $t\geq0,$ we have
\[
\left\vert e^{tJL_{c}}|_{E^{s}}\right\vert _{X_{1}}\leq Me^{-\lambda_{u}%
t},\ \left\vert e^{tJL_{c}}|_{E_{g}^{\ker}}\right\vert _{X_{1}}\leq M\left(
1+t\right)  ,
\]
and this finishes the proof.
\end{proof}

Next, we prove the exponential dichotomy in $X_{3}^{s}$, the cylindrical
symmetric subspace of $X_{3}\ $which is the domain of $L_{c}$ and $JL_{c}$.

\begin{lemma}
\label{lemma-dichotomy-x3}Under the conditions of Lemma
\ref{lemma-dichotomy-U}, the space $X_{3}^{s}$ can be written as%
\begin{equation}
X_{3}^{s}=E^{u}\oplus E_{3}^{cs},\text{ where }E_{3}^{cs}=X_{3}^{s}\cap E^{cs}
\label{DIRECT-SUM-3}%
\end{equation}
satisfying: i) Both $E^{u}$ and $E_{3}^{cs}$ are invariant under $e^{tJL_{c}}%
$. ii) there exist constants $M>0$ and $\lambda_{u}>0$, such that
\begin{equation}
\left\vert e^{tJL_{c}}|_{E_{3}^{cs}}\right\vert _{X_{3}}\leq M(1+t),\;\forall
\ t\geq0\text{ and }|e^{tJL_{c}}|_{E^{u}}|_{X_{3}}\leq Me^{\lambda_{u}%
t},\;\forall\ t\leq0. \label{estimate-dichotomy-3}%
\end{equation}

\end{lemma}

\begin{proof}
Since the eigenvectors $w_{u},w_{s}\in X_{3}^{s}$, we have $E^{u}
\subset X_{3}^{s}$. The invariance of $E_{3}^{cs}$ clearly follows from the
invariance of $X_{3}^{s}$ and $E^{cs}$ under $e^{tJL_{c}}$. The direct sum
decomposition of $X_{3}^{s}$ is a direct consequence of that of $X_{1}^{s}$.


To complete the proof, we only need to show estimate
(\ref{estimate-dichotomy-3}) on $E_{3}^{cs}$. It is easy to check that the
norm $\left\Vert u\right\Vert _{X_{3}}$ is equivalent to the norm $\left\Vert
u\right\Vert _{X_{1}}+\left\Vert JL_{c}u\right\Vert _{X_{1}}$. So we only need
to estimate the growth of $\left\Vert e^{tJL_{c}}u\right\Vert _{X_{1}%
}+\left\Vert JL_{c}e^{tJL_{c}}u\right\Vert _{X_{1}}$. For any $u\in E_{3}%
^{cs}$, by Lemma \ref{lemma-dichotomy-U}, we have
\[
\left\Vert e^{tJL_{c}}u\right\Vert _{X_{1}}\leq M(1+t)\left\Vert u\right\Vert
_{X_{1}}%
\]
and
\[
\left\Vert JL_{c}e^{tJL_{c}}u\right\Vert _{X_{1}}=\left\Vert e^{tJL_{c}}%
JL_{c}u\right\Vert _{X_{1}}\leq M(1+t)\left\Vert JL_{c}u\right\Vert _{X_{1}}.
\]
This finishes the proof of the lemma.
\end{proof}

By using the relation (\ref{relation-linearized-eqn}), we get the exponential
dichotomy for solutions of the linearized equation (\ref{eqn-linearized-gp-w}).

\begin{corollary}
\label{cor-dichotomy-w}Under the conditions of Lemma \ref{lemma-dichotomy-U},
the space $X_{3}^{s}$ can be decomposed as a direct sum%
\[
X_{3}^{s}=\tilde{E}^{u}\oplus\tilde{E}^{cs},
\]
satisfying: i) Both $\tilde{E}^{u}$ and $\tilde{E}^{cs}$ are invariant under
the linear semigroup $S\left(  t\right)  $ defined by
(\ref{eqn-linearized-gp-w}). ii) there exist constant $M>0$ and $\lambda
_{u}>0$, such that
\[
\left\vert S\left(  t\right)  |_{\tilde{E}^{cs}}\right\vert _{X_{3}}\leq
M(1+t),\quad\forall\;t\geq0\quad\text{ and }\quad|S\left(  t\right)
|_{\tilde{E}^{u}}|_{X_{1}}\leq Me^{\lambda_{u}t},\quad\forall\;t\leq0.
\]

\end{corollary}

\begin{remark}
\label{rmk-dichotomy-proof}The linear exponential dichotomy is the first step
to construct invariant manifolds. In general, it is rather tricky to get the
exponential dichotomy of the semigroup even if its generator has a spectral
gap. This is due to the issue of spectral mapping. More precisely, let
$\sigma\left(  L\right)  $ and $\sigma\left(  e^{L}\right)  $ be the spectra
of the generator $L$ and its exponential $e^{L}$. In general, it is not true
that $\sigma\left(  e^{L}\right)  =e^{\sigma\left(  L\right)  }$. In the
literature, the exponential dichotomy was proved by using resolvent estimates
(\cite{getsezy-latushkin-et-dichotomy}) or compact perturbation theory of
semigroups (\cite{shizuta-83} \cite{vidav}) or dispersive estimates
(\cite{grieg-ohta12} \cite{mitsumachi-instability}). In our current case, it
seems difficult to apply these approaches. In Lemma \ref{lemma-dichotomy-U},
we prove the exponential dichotomy of $e^{tJL_{c}}\ $by using the energy
estimates and the invariant quadratic form $\left\langle L_{c}\cdot
,\cdot\right\rangle $ due to the Hamiltonian structure. This could provide a
general approach to get the exponential dichotomy for lots of Hamiltonian PDEs.
\end{remark}

In the construction of unstable (stable) manifolds, we only need to establish
the exponential dichotomy in the cylindrical symmetric space $X_{1}^{s}%
\ $since the unstable modes are cylindrical symmetric. This also yields
cylindrical symmetric invariant manifolds. By using the non-degeneracy
condition (\ref{assumption-NDG}) and the proof of Lemma
\ref{lemma-dichotomy-U}, we can also get the exponential trichotomy in the
whole space $X_{1}$. This will be important in a future work for the
construction of center manifolds in the energy space.

\begin{lemma}
\label{lemma-trichotomy}For $0<c<\sqrt{2}$, let $U_{c}\ $be a traveling wave
solution of (\ref{eqn-GP-generalized}) constructed in \cite{Maris-annal} and
$L_{c}$ be the operator defined by (\ref{operator-Lc-g}). Assume
\eqref{assumption-NDG} and $\frac{\partial P(U_{c})}{\partial c}|_{c=c_{0}}%
<0$. Then the space $X_{1}$ is decomposed as a direct sum%
\[
X_{1}=E^{u}\oplus E^{c}\oplus E^{s},
\]
satisfying: i) $E^{u},E^{s}$ and $E^{c}$ are invariant under the linear
semigroup $e^{tJL_{c}}$. ii) there exist constant $M>0,\ \lambda_{u}>0$, such
that
\[
\left\vert e^{tJL_{c}}|_{E^{s}}\right\vert _{X_{1}}\leq Me^{-\lambda_{u}%
t},\quad\forall\;t\geq0,\ \ |e^{tJL_{c}}|_{E^{u}}|_{X_{1}}\leq Me^{\lambda
_{u}t},\quad\forall\;t\leq0.
\]
and
\begin{equation}
\ |e^{tJL_{c}}|_{E^{c}}|_{X_{1}}\leq M(1+t),\ \forall\;t\in\mathbf{R,}.
\label{estimate-center}%
\end{equation}

\end{lemma}

\begin{proof}
Denote $E^{s}=\left\{  w_{s}\right\}  $ and $E^{u}=\left\{  w_{u}\right\}  $,
where $w_{u},w_{s}$ are the unstable and stable eigenfunctions of $JL_{c}$.
Let $w^{1}=-\partial_{c}U_{c},\ $then $JL_{c}w^{1}=\partial_{x_{1}}U_{c}$. Let
$w^{2},w^{3}$ be such that
\[
JL_{c}w^{2}=\partial_{x_{2}}U_{c},\ JL_{c}w^{3}=\partial_{x_{3}}U_{c}.
\]
Above two equations are solvable since the kernel of $-L_{c}J=(JL_{c})^{\ast}$
is spanned by $J^{-1}\partial_{x_{i}}U_{c}$, $i=1,2,3$, and
\[
\left\langle J^{-1}\partial_{x_{j}}U_{c},\partial_{x_{i}}U_{c}\right\rangle
=0,\ \ \text{for }i,j=1,2,3,
\]
by the translation invariance of the momentum $\vec{P}\left(  U_{c}\right)
=\left(  \left\langle J^{-1}\partial_{x_{i}}U_{c},U_{c}-1\right\rangle
\right)  $. Denote the generalized kernel of $JL_{c}$ by
\[
E_{g}^{\ker}=span \left\{  \cup_{i=1}^{3}\{ \partial_{x_{i}}U_{c_{0}},w^{i}\}
\right\}  .
\]
Define%
\[
Y_{1}=\left\{  u\in X_{1}|\ \left\langle u,L_{c}w_{u}\right\rangle
=\left\langle u,L_{c}w_{s}\right\rangle =0\right\}  .
\]
Clearly $E_{g}^{ker}\subset Y_{1}$ due to the symmetry of $L_{c}$ and the
skew-symmetry of $J$. Moreover $Y_{1}$ is invariant under $e^{tJL_{c}}$ due to
the invariance of $span\{w^{u},w^{s}\}$ and the invariance of the quadratic
form given by $L_{c}$. Let
\[
E^{e}=\left\{  u\in Y_{1}|\ \left\langle u,J^{-1}w\right\rangle =0,\ \forall
w\in E_{g}^{\ker}\right\}  .
\]
It is straightforward to check that $E^{e}$ is invariant under $e^{tJL_{c}}$
due to the invariance of $Y_{1}$ and $E_{g}^{ker}$.

By the arguments in the proof of Lemma \ref{lemma-dichotomy-U}, we have
$\left\langle L_{c}\cdot,\cdot\right\rangle |_{Y_{1}}\geq0$. This implies
that
\[
\left\langle L_{c}w^{i},w^{i}\right\rangle >0,\ \text{for }i=1,2,3,
\]
by noting that $w^{i}\in Y_{1}$. Indeed, suppose otherwise, then
\[
\left\langle L_{c}w^{i},w^{i}\right\rangle =0=\min_{w\in Y_{1}}\left\langle
L_{c}w,w\right\rangle .
\]
Thus $\left\langle L_{c}w^{i},w\right\rangle =0$ for any $w\in Y_{1},$ and it
follows that $\left\langle L_{c}w^{i},w\right\rangle =0$ for any $w\in X_{1}$.
Thus, $L_{c}w^{i}=0$, a contradiction. So for any $u\in X_{1}$, we can write%
\[
u=c_{u}w_{u}+c_{s}w_{s}+\sum_{i=1}^{3}\left(  a_{i}\partial_{x_{i}}U_{c}%
+b_{i}w^{i}\right)  +v_{1},
\]
where $v_{1}\in E^{e},$
\[
c_{u}=\left\langle L_{c}u,w_{s}\right\rangle /\left\langle L_{c}w_{u}%
,w_{s}\right\rangle ,\ c_{s}=\left\langle L_{c}u,w_{u}\right\rangle
/\left\langle L_{c}w_{u},w_{s}\right\rangle ,
\]
and
\[
a_{i}=-\left\langle u,J^{-1}w^{i}\right\rangle /\left\langle L_{c}w^{i}%
,w^{i}\right\rangle ,\ b_{i}=\left\langle u,J^{-1}\partial_{x_{i}}%
U_{c}\right\rangle /\left\langle L_{c}w^{i},w^{i}\right\rangle .
\]
Here we used the facts
\[
\left\langle w^{j},J^{-1}\partial_{x_{i}} U_{c}\right\rangle = \left\langle
w^{i},J^{-1} w^{j}\right\rangle =0, \quad i\ne j
\]
due to the even or odd symmetry of $w^{i}$ and $\partial_{x_{i}} U_{c}$ in
$x_{j}$. Thus we get the direct sum decomposition%
\[
X_{1}=E^{u}\oplus E^{s}\oplus E^{e}\oplus E_{g}^{\ker}.
\]
By the proof of Lemma \ref{lemma-dichotomy-U}, it follows that the quadratic
form $\left\langle L_{c}\cdot,\cdot\right\rangle |_{E^{e}}$ is positive
definite. This implies that
\[
|e^{tJL_{c}}|_{E^{e}}|_{X_{1}}\leq C,\ \forall\;t\in\mathbf{R,}%
\]
for some constant $C$. Define $E^{c}=E^{e}\oplus E_{g}^{\ker}$. Since
$|e^{tJL_{c}}|_{E_{g}^{\ker}}|_{X_{1}}$has only linear growth, the estimate
(\ref{estimate-center}) follows. This finishes the proof.
\end{proof}

\subsection{Invariant manifolds and orbital instability}

In Appendix 1, we prove that the nonlinear term $\Psi(w)$ in the equation
(\ref{eqn-nonlinear--TF-w}) is $C^{2}(X_{3},X_{3})$. Thus, by the standard
invariant manifold theory for semilinear PDEs (e.g. \cite{bates-jones88,
CL88}), we get the following

\begin{theorem}
\label{thm:invariant manifold}For $0<c_{0}<\sqrt{2}$, let $U_{c_{0}}%
=\psi\left(  w_{c_{0}}\right)  \ $be a traveling wave solution of
(\ref{eqn-GP-generalized}) constructed in \cite{Maris-annal}, satisfying
$\frac{\partial P(U_{c})}{\partial c}|_{c=c_{0}}<0$. Assume in addition $F\in
C^{5}$ in a neighborhood of the set $|U_{c_{0}} (\mathbf{R^{3}})|^{2}$
and the non-degeneracy condition (\ref{assumption-NDG-s}). Then there exists a
unique $C^{2}$ local unstable manifold $W^{u}$ of $w_{c_{0}}$ in $X_{3}^{s}$
which satisfies

\begin{enumerate}
\item It is one-dimensional and tangent to $\tilde{E}^{u}$ at $w_{c_{0}}$.

\item It can be written as the graph of a $C^{2}\ $mapping from a neighborhood
of $w_{c_{0}}$ in $\tilde{E}^{u}$ to $\tilde{E}^{cs}$.

\item It is locally invariant under the flow of the equation
(\ref{eqn-nonlinear--TF-w}), i.e. solutions starting on $W^{u}$ can only leave
$W^{u}$ through its boundary.

\item Solutions starting on $W^{u}$ converges to $w_{c_{0}}$ at the rate
$e^{\lambda_{u}t}$ as $t\rightarrow-\infty$.
\end{enumerate}

The same results hold for local stable manifold of $w_{c_{0}}$ as the equation
(\ref{eqn-nonlinear--TF-w}) is time-reversible.
\end{theorem}

\begin{corollary}
\label{cor-invariant-mfld-u}By using the transformation $u=\psi\left(
w\right)  $, we get the stable and unstable manifolds $\tilde{W}^{u,s}%
=\psi\left(  W^{u,s}\right)  \ $near $U_{c_{0}}\ $in the metric
\[
d_{3}\left(  u,\tilde{u}\right)  =d_{1}\left(  u,\tilde{u}\right)  +\left\Vert
\nabla^{2}\left(  u-\tilde{u}\right)  \right\Vert _{L^{2}}+\left\Vert
\nabla^{3}\left(  u-\tilde{u}\right)  \right\Vert _{L^{2}}%
\]
which is equivalent to the metric $\left\Vert w-\tilde{w}\right\Vert _{X_{3}}$
for $w$. Since $W^{u}$ is one-dimensional, the $d_{3}$ topology and $d_{1}$
topology are equivalent on $W^{u}$. Then an immediate consequence of the above
theorem is the nonlinear instability in $d_{1}$ metric with initial data
slightly perturbed from $U_{c_{0}}$ in $d_{3}$ metric.
\end{corollary}

To compare with the orbital stability result, it is more desirable to get an
orbital instability result as follows.

\begin{corollary}
\label{cor-orbital-instability}Under the assumptions of Theorem
\ref{thm:invariant manifold}, the traveling wave solution $U_{c_{0}}$ is
nonlinearly unstable in the following sense:

$\exists\ \theta,C>0$, such that for any $\delta>0$, there exists a solution
$u_{\delta}\left(  t\right)  $ of equation (\ref{eqn-GP-generalized})
satisfying
\begin{equation}
d_{3}\left(  u_{\delta}\left(  0\right)  ,U_{c_{0}}\right)  \leq\delta,
\label{deviation-initial}%
\end{equation}
and \qquad%
\begin{equation}
\sup_{0<t\leq C\left\vert \ln\delta\right\vert }\inf_{y\in\mathbf{R}^{3}%
}\left\Vert \nabla\left(  u_{\delta,i}\left(  t\right)  -U_{c_{0},i}\left(
\cdot+y\right)  \right)  \right\Vert _{L^{2}}\geq\theta,\ \ \ i=1,2.
\label{deviation-growth}%
\end{equation}
Here, $u_{\delta}\left(  t\right)  =u_{\delta,1}\left(  t\right)
+iu_{\delta,2}\left(  t\right)  $ and $U_{c_{0}}\left(  x\right)  =U_{c_{0}%
,1}\left(  x\right)  +iU_{c_{0},2}\left(  x\right)  $.

\end{corollary}

\begin{proof}
First, we observe that if $u_{g}=u_{g,1}+iu_{g,2}\in H^{3}~$is an unstable
eigenfunction of $JL_{c},$ then $u_{g,i}\neq0$ for $i=1,2$. Suppose otherwise
$u_{g,1}=0$, then from the equation $JL_{c}u_{g}=\lambda_{u}u_{g}$ $\left(
\lambda_{u}>0\right)  ,$ we get
\[
c_{0}\partial_{x_{1}}u_{g,2}-\left(  2F^{\prime}\left(  \left\vert
U_{c}\right\vert ^{2}\right)  u_{c}v_{c}+\lambda_{u}\right)  u_{g,2}=0.
\]
This implies that $u_{g,2}=0,$ thus $u_{g}=0$, a contradiction. Similarly, we
can show that $u_{g,2}\neq0$. The nonlinear instability in $\left\Vert
\nabla\left(  u_{i}-U_{c_{0},i}\right)  \right\Vert _{L^{2}}\ $follows
directly from the existence of unstable manifold and the above observation. To
show orbital instability, we follow the proof of Theorem 6.2 in \cite{gss90}.
We only show the orbital instability in the norm $\left\Vert \nabla\left(
u_{1}-U_{c_{0},1}\right)  \right\Vert _{L^{2}}$, since the proof for
$\left\Vert \nabla\left(  u_{2}-U_{c_{0},2}\right)  \right\Vert _{L^{2}}$ is
the same. Let $u_{g,1}^{\perp}$ be the projection of $u_{g,1}$ onto the space
$Z_{1}^{\perp}$, the orthogonal complement space of $\ Z_{1}=span\left\{
\partial_{x_{i}}U_{c_{0},1},i=1,2,3\right\}  $ in the inner product
$\left\langle \left\langle u_{1},v_{1}\right\rangle \right\rangle =\left(
\nabla u_{1},\nabla v_{1}\right)  $. Fix sufficiently small $\varepsilon_{0}$
and for any $\delta>0$, we can choose the solution $u_{\delta}\left(
t\right)  $ on the unstable manifold $\tilde{W}^{u}$, such that $d_{3}\left(
u_{\delta}\left(  0\right)  ,U_{c_{0}}\right)  \leq\delta,$
\[
\left\Vert \nabla\left(  u_{\delta,1}\left(  t\right)  -U_{c_{0},1}\right)
\right\Vert _{L^{2}}\leq C\varepsilon_{0},\ \text{for }0<t<T_{1}%
\]%
\[
\left\langle \left\langle u_{\delta,1}\left(  T_{1}\right)  -U_{c_{0}%
,1},u_{g,1}^{\perp}\right\rangle \right\rangle \geq\varepsilon_{0},
\]
where $T_{1}=C\left\vert \ln\delta\right\vert $. Here $C$ may depend on
$\varepsilon_{0}$, but is independent of $\delta>0$. Let $h=h(t)\in
\mathbf{R}^{3}$ be such that
\[
\left\Vert \nabla\left(  u_{\delta,1}\left(  t\right)  -U_{c_{0},1}\left(
\cdot+h\right)  \right)  \right\Vert _{L^{2}}\leq2\theta,\quad\theta
=\inf_{y\in\mathbf{R}^{3}}\left\Vert \nabla\left(  u_{\delta,1}\left(
t\right)  -U_{c_{0},1}\left(  \cdot+y\right)  \right)  \right\Vert _{L^{2}}.
\]
Then
\[
\left\Vert \nabla\left(  U_{c_{0},1}\left(  \cdot\right)  -U_{c_{0},1}\left(
\cdot+h\right)  \right)  \right\Vert _{L^{2}}\leq3\left\Vert \nabla\left(
u_{\delta,1}(t)-U_{c_{0},1}\right)  \right\Vert _{L^{2}}\leq2C\varepsilon
_{0},
\]
thus $\left\vert h\right\vert =O\left(  \varepsilon_{0}\right)  $. So we can
write
\[
U_{c_{0},1}\left(  x+h\right)  =U_{c_{0},1}\left(  x\right)  +h\cdot\nabla
U_{c_{0},1}\left(  x\right)  +O\left(  \varepsilon_{0}^{2}\right)  .
\]
This implies that
\begin{align*}
C\theta &  \geq\left\langle \left\langle u_{\delta,1}-U_{c_{0},1}\left(
\cdot+h\right)  ,u_{g,1}^{\perp}\right\rangle \right\rangle \\
&  \geq\left\langle \left\langle u_{\delta,1}-U_{c_{0},1},u_{g,1}^{\perp
}\right\rangle \right\rangle -O\left(  \varepsilon_{0}^{2}\right)
\geq\varepsilon_{0}/2,
\end{align*}
by using the orthogonal property of $u_{g,1}^{\perp}$ and $Z_{1}$. This
finishes the proof.
\end{proof}

\begin{remark}
\label{rmk-reason-coordinate-instability}By using the exponential dichotomy
for the semigroup $e^{tJL_{c}}$ (Lemma \ref{lemma-dichotomy-x3}), we can
construct unstable (stable) manifolds near $U_{c}$ directly from equation
(\ref{eqn-GP-g-TF}) in the space $H^{3}\left(  \mathbf{R}^{3}\right)  \times$
$\dot{H}^{3}\left(  \mathbf{R}^{3}\right)  \,$. However, the functions in
$U_{c}+H^{3}\left(  \mathbf{R}^{3}\right)  \times$ $\dot{H}^{3}\left(
\mathbf{R}^{3}\right)  $ are not guaranteed to be in the energy space $X_{0}$.
To get the invariant manifolds lying on $X_{0}$, we use the coordinate mapping
$U=\psi\left(  w\right)  $ to rewrite the equation (\ref{eqn-GP-g-TF}) as
(\ref{eqn-nonlinear--TF-w}) for $w\in H^{3}\left(  \mathbf{R}^{3}\right)
\times$ $\dot{H}^{3}\left(  \mathbf{R}^{3}\right)  $.
\end{remark}

\begin{remark}
Since the eigenfunctions of $JL_{c}$ actually belongs to $H^{k}$, instead of
constructing the unstable/stable manifolds of traveling waves through the
coordinate change $U=\psi(w)$ and working on (\ref{eqn-nonlinear--TF-w}), one
can also work on (\ref{eqn-GP-g-TF}) directly in the space $U_{c}+H^{k}$. The
details are similar to the proof of Proposition \ref{thm-UM-general} and
Corollary \ref{thm-mfld-madelung}. However, that approach, based on the
improved properties of unstable eigenfunctions, would not be useful when we
construct the center manifolds in the energy space in the forthcoming work.
\end{remark}

\begin{remark}
For (GP) equation, numerical computations (\cite{berloff-roberts-X-stability}
\cite{jones-et-stability}) suggested that $\frac{\partial P(U_{c})}{\partial
c}<0$ iff $c\in\left(  c^{\ast},\sqrt{2}\right)  $ for some $c^{\ast}%
\in\left(  0,\sqrt{2}\right)  $. So for 3D traveling waves of (GP), the
instability sets in at a critical velocity $c^{\ast}$. In the contrast, for
cubic-quintic equation, we have $\frac{\partial P(U_{c})}{\partial c}<0$ and
thus the instability when $c$ is near $0$ and$\ \sqrt{2}$. So there may not
exist a critical speed for instability. The case for small $c$ is proved in
Theorem \ref{thm-instability-madelung} and $\frac{\partial P(U_{c})}{\partial
c}<0$ for $c$ near $\sqrt{2}$ can be seen from the transonic limit
(\cite{lett89} \cite{chiron-maris-kp} \cite{jones-et-stability}) of traveling
waves of (\ref{eqn-GP-generalized}) to solitary waves of
Kadomtsev--Petviashvili (KP) equation.
\end{remark}

\section{Transversal instability of $2D$ traveling waves}

In this section, we prove transversal instability of $2$D traveling waves of
(\ref{eqn-GP-generalized}). Unlike the $3$D instability result (Proposition
\ref{prop-linear-insta} and Theorem \ref{thm:invariant manifold}), we do not
need to assume the non-degeneracy condition (\ref{assumption-NDG}) for the
$2$D traveling waves.

To state the result, first we introduce some notations. Assume $F\in C^{1}
(\mathbf{R}^{+})$. For $0<c<\sqrt{2} ,$ consider the operator $L_{c}$ defined
by (\ref{operator-Lc-g}), where $U_{c}\left(  x_{1}-ct,\left\vert
x_{2}\right\vert \right)  $ is a $2$D traveling wave solution of
(\ref{eqn-GP-generalized}). Then it is easy to show that $L_{c}:\left(
H^{2}\left(  \mathbf{R}^{2}\right)  \right)  ^{2}\rightarrow\left(
L^{2}\left(  \mathbf{R}^{2}\right)  \right)  ^{2}$ is self-adjoint and
\[
\sigma_{ess}\left(  L_{c}\right)  =\sigma_{ess}\left(  L_{c,\infty}\right)
=[0,+\infty),\ \text{for any }c\in\left(  0,\sqrt{2}\right)  ,
\]
where $L_{c,\infty}$ is defined in (\ref{defn-L-infty}). Let $\lambda
_{0}\left(  L_{c}\right)  \ $be the first eigenvalue of $L_{c}$.

\begin{theorem}
\label{thm-linear-transveral-instability} For $0<c<\sqrt{2}$, let
$U_{c}\left(  x_{1}-ct,\left\vert x_{2}\right\vert \right)  $ be a $2$D
traveling wave of (\ref{eqn-GP-generalized}). Suppose $\lambda_{0}\left(
L_{c}\right)  <0$. Let $\lambda_{1}\leq0$ be the second eigenvalue. Then
$U_{c}$ is transversely unstable in the following sense: for any
\[
k\in\left(  \sqrt{-\lambda_{1}},\sqrt{-\lambda_{0}}\right)  ,
\]
there exist an unstable solution
\begin{equation}
e^{\lambda_{u}t+ikx_{3}}u_{g}\left(  x_{1},x_{2}\right)  \text{, with }%
\lambda_{u}>0,u_{g}\in\left(  H^{3}\left(  \mathbf{R}^{2}\right)  \right)
^{2} \label{form-transveral-instability}%
\end{equation}
for the linearized equation (\ref{eqn-linearized-transverse}). If
$k>\sqrt{-\lambda_{0}}$, then no such solution with $\lambda_{u}>0$ exists,
that is, there is spectral stability.
\end{theorem}

\begin{remark}
Denote the momentum by
\[
P\left(  u\right)  =\frac{1}{2}\int_{\mathbf{R}^{2}}\langle i\partial_{x_{1}%
}u,u\rangle\ dx=-\int_{\mathbf{R}^{2}}u_{1}\partial_{x_{1}}u_{2}dx.
\]
When $\frac{\partial P(U_{c})}{\partial c}|_{c=c_{0}}>0$, the instability
condition $\lambda_{0}\left(  L_{c}\right)  <0$ is satisfied by the traveling
wave $U_{c_{0}}$. This is due to the identity
\[
\left\langle L_{c}\partial_{c}U_{c},\partial_{c}U_{c}\right\rangle
=-\left\langle P^{\prime}(U_{c}),\partial_{c}U_{c}\right\rangle =-\frac
{\partial P\left(  U_{c}\right)  }{\partial c},
\]
by (\ref{relation-Lc-par-c-U-c}). Numerical evidences
(\cite{berloff-roberts-crow} \cite{jones-et-stability}) showed that the
condition $\frac{\partial P(U_{c})}{\partial c}>0$ is satisfied for $2$D
traveling waves of (GP). Moreover, 2D traveling waves of (GP) were constructed
in (\cite{betheul-et-minimizer-cmp} \cite{chiron-marisII-12}) as energy
minimizers subject to a fixed momentum. This implies that $L_{c}$ can have at
most one negative eigenvalue, by a similar proof of Lemma \ref{le-5}. So for
any $0<c<\sqrt{2}$, we have $\lambda_{0}\left(  L_{c}\right)  <0$ and
$\lambda_{1}\left(  L_{c}\right)  =0$ for any $2$D traveling wave of (GP). By
Theorem \ref{thm-linear-transveral-instability}, any 2D traveling wave of (GP)
is transversely unstable if and only if the transversal wave number
$k\in\left(  0,\sqrt{-\lambda_{0}}\right)  $. When $k\rightarrow0+$, such
transversal instability had been studied by asymptotic analysis in
\cite{kuznetsov-rasmussen95} \ \cite{berloff-roberts-crow}. In the limit
$c\rightarrow0$, the 2D traveling waves of (GP) consist of an antiparallel
vortex pair (\cite{b-s-1998} \cite{jones-et-symmetry}). In this case, the
mechanism of transversal instability is analogous to the Crow instability of
an antiparallel vortex pair of incompressible fluid (\cite{crow-instability}).
\end{remark}

\begin{proof}
The linearized equation of (\ref{eqn-GP-g-TF}) near $U_{c}\left(  x_{1}%
,x_{2}\right)  $ can be written as
\begin{equation}
\frac{du}{dt}=J\tilde{L}_{c}u, \label{eqn-linearized-transverse}%
\end{equation}
where
\[
\tilde{L}_{c}=L_{c}+\left(
\begin{array}
[c]{cc}%
-\frac{d^{2}}{dx_{3}^{2}} & 0\\
0 & -\frac{d^{2}}{dx_{3}^{2}}%
\end{array}
\right)  .
\]
So to find an unstable solution of the form (\ref{form-transveral-instability}%
) for the linearized equation (\ref{eqn-linearized-transverse}), it is
equivalent to solve the eigenvalue problem $J\left(  L_{c}+k^{2}\right)
u_{g}=\lambda_{u}u_{g}$. Denote $L_{c,k}=L_{c}+k^{2}$, then for $k\in\left(
\sqrt{-\lambda_{1}},\sqrt{-\lambda_{0}}\right)  ,$ the operator $L_{c,k}$ has
one negative eigenvalue, no kernel and the rest of the spectrum is contained
in $\left(  \delta_{0},\infty\right)  $ with $\delta_{0}=k^{2}+\lambda_{1}>0$.
The existence of an unstable eigenvalue of $JL_{c,k}$ follows by the line of
proof of Proposition \ref{prop-linear-insta}, in a much simplified way since
$\sigma\left(  L_{c,k}\right)  $ does not contain $0$. When $k>\sqrt
{-\lambda_{0}}$, the operator $L_{c,k}$ is positive. This implies the
non-existence of unstable modes since any such mode satisfies $\left\langle
L_{c,k}u_{g},u_{g}\right\rangle =0$.
\end{proof}

We now prove nonlinear transversal instability under the instability condition
in Theorem \ref{thm-linear-transveral-instability}. For any $k_{0}\in\left(
\sqrt{-\lambda_{1}},\sqrt{-\lambda_{0}}\right)  $, denote $H^{m}\left(
\mathbf{R}^{2}\times S_{\frac{2\pi}{k_{0}}}\right)  $ to be all functions in
$H^{m}\left(  \mathbf{R}^{2}\times\left[  0,\frac{2\pi}{k_{0}}\right]
\right)  $ which are $\frac{2\pi}{k_{0}}$-periodic in $x_{3}$. Let
\[
X_{1,k_{0}}=\left\{  u\left(  x_{1},x_{2},x_{3}\right)  \in\left(
H^{1}\left(  \mathbf{R}^{2}\times S_{\frac{2\pi}{k_{0}}}\right)  \right)
^{2}\ |\ u\ \text{is odd in }x_{3}\right\}  ,
\]
and
\[
X_{3,k_{0}}=X_{1,k_{0}}\cap\left(  H^{3}\left(  \mathbf{R}^{2}\times
S_{\frac{2\pi}{k_{0}}}\right)  \right)  ^{2}.
\]
From Theorem \ref{thm-linear-transveral-instability}, we have a linearly
unstable mode of the form
\[
e^{\lambda_{u}t}\sin\left(  k_{0}x_{3}\right)  u_{g}\left(  x_{1}%
,x_{2}\right)
\]
in the space $X_{3,k_{0}}$. We will construct unstable manifold near the
traveling wave $U_{c}\left(  x_{1},x_{2}\right)  $ in the space
$\,1+X_{3,k_{0}}$. First, we show the exponential dichotomy of $e^{tJ\tilde
{L}_{c}}$ in the space $X_{3,k_{0}}$.

\begin{lemma}
\label{lemma-dichotomy-periodic}For any
\begin{equation}
k_{0}\in\left(  \max\left\{  \sqrt{-\lambda_{1}},\frac{\sqrt{-\lambda_{0}}}%
{4}\right\}  ,\sqrt{-\lambda_{0}}\right)  , \label{assumption-k0}%
\end{equation}
the space $X_{3,k_{0}}$ is decomposed as a direct sum%
\[
X_{3,k_{0}}=E^{u}\oplus E^{cs},
\]
satisfying: i) Both $E^{u}$ and $E^{cs}$ are invariant under the linear
semigroup $e^{tJ\tilde{L}_{c}}$. ii) there exist constants $M>0$ and
$\lambda_{u}>0$, such that
\[
\left\vert e^{tJ\tilde{L}_{c}}|_{E^{cs}}\right\vert _{H^{3}\left(
\mathbf{R}^{2}\times S_{\frac{2\pi}{k_{0}}}\right)  }\leq M,\quad
\forall\;t\geq0,
\]
and $\quad$%
\[
|e^{tJ\tilde{L}_{c}}|_{E^{u}}|_{H^{3}\left(  \mathbf{R}^{2}\times
S_{\frac{2\pi}{k_{0}}}\right)  }\leq Me^{\lambda_{u}t},\quad\forall\;t\leq0.
\]

\end{lemma}

\begin{proof}
First, we show the exponential dichotomy in the space $X_{1,k_{0}}$. Any
function $u\in X_{1,k_{0}}$ can be written as
\[
u\left(  x_{1},x_{2},x_{3}\right)  =\sum_{j=1}^{\infty}\sin\left(  jk_{0}%
x_{3}\right)  u_{j}\left(  x_{1},x_{2}\right)  ,
\]
and
\begin{equation}
\left\Vert u\right\Vert _{H^{1}\left(  \mathbf{R}^{2}\times S_{\frac{2\pi
}{k_{0}}}\right)  }^{2}=\sum_{j=1}^{\infty}\left(  \left\Vert u_{j}\right\Vert
_{H^{1}\left(  \mathbf{R}^{2}\right)  }^{2}+j^{2}\left\Vert u_{j}\right\Vert
_{L^{2}\left(  \mathbf{R}^{2}\right)  }^{2}\right)  . \label{norm-H1-peiodic}%
\end{equation}
We have
\[
u\left(  t\right)  =e^{tJ\tilde{L}_{c}}u=\sum_{j=1}^{\infty}\sin\left(
jk_{0}x_{3}\right)  \left(  e^{tJL_{c,jk_{0}}}u_{j}\right)  \left(
x_{1},x_{2}\right)  .
\]
By assumption (\ref{assumption-k0}), the operator $L_{c,k_{0}}$ has one
negative eigenvalue and the rest of the spectrum lying in the positive axis,
and $\left\{  L_{c,jk_{0}}\right\}  $ $\left(  j\geq2\right)  $ are positive.
So by the proof of Theorem \ref{thm-linear-transveral-instability}, there
exist a pair of stable and unstable modes of the form $e^{\pm\lambda_{u}%
t}u^{\pm}\ \left(  \lambda_{u}>0\right)  $, where $u^{\pm}=\sin\left(
k_{0}x_{3}\right)  u_{\pm}\left(  x_{1},x_{2}\right)  $. Define $E^{u}%
=span\left\{  u^{+}\right\}  ,\ E^{s}=span\left\{  u^{-}\right\}  ,\ $and
\[
E^{e}=\left\{  u=\sum_{j=1}^{\infty}\sin\left(  jk_{0}x_{3}\right)
u_{j}\left(  x_{1},x_{2}\right)  \in X_{1,k_{0}}\ |\ \left\langle L_{c,k_{0}%
}u_{1},u^{\pm}\right\rangle =0\right\}  .
\]
Then by the arguments in the proof of Lemma \ref{lemma-dichotomy-U}, for any
$u\in E^{e}$, we have
\[
\left\Vert e^{tJL_{c,k_{0}}}u_{1}\right\Vert _{H^{1}\left(  \mathbf{R}%
^{2}\right)  }\leq C\left\Vert u_{1}\right\Vert _{H^{1}\left(  \mathbf{R}%
^{2}\right)  },\text{ for some constant }C,
\]
and by the positivity of $L_{c,jk_{0}}$ $\left(  j\geq2\right)  ,$%
\[
\left\Vert e^{tJL_{c,jk_{0}}}u_{j}\right\Vert _{H^{1}\left(  \mathbf{R}%
^{2}\right)  }^{2}+j^{2}\left\Vert e^{tJL_{c,jk_{0}}}u_{j}\right\Vert
_{L^{2}\left(  \mathbf{R}^{2}\right)  }^{2}\leq C\left(  \left\Vert
u_{j}\right\Vert _{H^{1}\left(  \mathbf{R}^{2}\right)  }^{2}+j^{2}\left\Vert
u_{j}\right\Vert _{L^{2}\left(  \mathbf{R}^{2}\right)  }^{2}\right)
\]
with some constant $C$ independent of $j$. So by (\ref{norm-H1-peiodic})$,$ we
have
\[
\left\Vert e^{tJ\tilde{L}_{c}}u\right\Vert _{H^{1}\left(  \mathbf{R}^{2}\times
S_{\frac{2\pi}{k_{0}}}\right)  }\text{\ }\leq C\left\Vert u\right\Vert
_{H^{1}\left(  \mathbf{R}^{2}\times S_{\frac{2\pi}{k_{0}}}\right)
},\ \text{\ for }u\in E^{e}.
\]
Define $E^{cs}=E^{s}\cup E^{e}$. Then $X_{1,k_{0}}=E^{u}\oplus E^{cs}$ is a
direct sum decomposition for the exponential dichotomy of $e^{tJ\tilde{L}_{c}%
}$. Define
\[
E_{3}^{cs}=\left\{  u\in X_{3,k_{0}}|\ J\tilde{L}_{c}u\in E^{cs}\right\}  .
\]
Then $X_{3,k_{0}}=E^{u}\oplus E_{3}^{cs}$ and the exponential dichotomy follow
by the same argument in the proof of Lemma \ref{lemma-dichotomy-x3}.
\end{proof}

In the equation (\ref{eqn-GP-g-TF}), we let $U=U_{c}+u$, with $u\in
X_{3,k_{0}}.$ Then the equation can be written as
\[
u_{t}=J\tilde{L}_{c}u+\Psi\left(  u\right)  \text{.}%
\]
If $F\in C^{5}\left(  \mathbf{R}\right)  ,\ $it is easy to show that the
nonlinear term $\Psi\left(  u\right)  $ is $C^{2}\left(  X_{3,k_{0}%
},X_{3,k_{0}}\right)  .$ So by using Lemma \ref{lemma-dichotomy-periodic}, we
have the following

\begin{theorem}
\label{thm-transversal-mfld-2d}For $0<c_{0}<\sqrt{2}$, let $U_{c_{0}}\ $be a
$2$D traveling wave solution of (\ref{eqn-GP-generalized}), satisfying
$\frac{\partial P(U_{c})}{\partial c}|_{c=c_{0}}>0$ or more generally
$\lambda_{0}\left(  L_{c}\right)  <0$. For any $k_{0}$ satisfying
(\ref{assumption-k0}), there exists a unique $C^{1}$ local unstable manifold
$W^{u}$ of $U_{c_{0}}$ in $X_{3,k_{0}}$ which satisfies

\begin{enumerate}
\item It is one-dimensional and tangent to $E^{u}$ at $U_{c_{0}}$.

\item It can be written as the graph of a $C^{1}\ $mapping from a neighborhood
of $U_{c_{0}}$ in $E^{u}$ to $E_{3}^{cs}$.

\item It is locally invariant under the flow of the equation
(\ref{eqn-GP-g-TF}).

\item Solutions starting on $W^{u}$ converge to $U_{c_{0}}$ at the rate
$e^{\lambda_{u}t}$ as $t\rightarrow-\infty$.
\end{enumerate}
\end{theorem}

As a corollary of above theorem, we get nonlinear transversal instability of
any $2$D traveling wave of (GP) equation.

\begin{remark}
Assumption \eqref{assumption-k0} ensures that the unstable subspace of the
linearized equation in $X_{3,k_{0}}$ is 1-dimensional. In fact this assumption
can be generalized to
\begin{equation}
\exists\ j_{0}\geq1\text{ such that }j_{0}k_{0}\text{ satisfies
\eqref{assumption-k0}.} \label{E:assumption-k0}%
\end{equation}
Since the subspace corresponding to the $j_{0}$-th mode is decoupled in the
linearized equation, this assumption ensures that there exists a 1-dimensional
unstable subspace in the $j_{0}$-th mode which implies the linear instability
with possibly multiple dimensional unstable subspaces. $\tilde{L}_{c}$ is
uniformly positive in all but finitely many directions, one may prove the
linear exponential dichotomy and the existence of unstable manifolds through a
similar procedure.
\end{remark}

\section{Slow traveling waves of cubic-quintic type equations}

In this section, we assume the nonlinear term of (\ref{eqn-GP-generalized})
satisfies the following:

(H1) $F \in C^{1}([0, \infty))$, $F(r_{0})=0$, and $F^{\prime}(r_{0})<0$,
where $r_{0}$ is a positive constant.

(H2) $\exists~C>0$ such that $\left\vert F^{\prime}\left(  s\right)
\right\vert \leq C\left\vert s\right\vert ^{p_{0}-1},~$for $s\geq1$, where
$p_{0}=\frac{2}{n-2}.$

(H3) $\exists\ r_{1}$ such that $0\leq r_{1}<r_{0}$ and $V(r_{1})<0$, where
$V(r)=\int_{0}^{r}F(s)ds$.

A typical example is the so-called cubic-quintic (or $\psi^{3}-\psi^{5}$)
nonlinear Schr\"{o}dinger equation
\begin{equation}
i\psi_{t}+\Delta\psi-\alpha_{1}\psi+\alpha_{3}\psi|\psi|^{2}-\alpha_{5}%
\psi|\psi|^{4}=0,x\in{\mathbf{R}}^{3}, \label{eqn-cubic-quintic}%
\end{equation}
where $\alpha_{1},\alpha_{3},\alpha_{5}$ are positive constants satisfying
(\ref{condition-3-5}). The main result of this section is to show the
existence and instability of traveling waves with small speeds.

\subsection{Existence of slow traveling waves}

First we recall the result of stationary solutions.

\begin{theorem}
\label{thm-stationary bubble}\cite{de95} Under assumptions (H1)-(H3), there
exists a real-valued function $\phi_{0}\in C^{2}({\mathbf{R}}^{n})$ satisfying

(1)\ $\phi_{0}(x)=\phi_{0}(\left\vert x\right\vert )$ (i.e. $\phi$ is radially symmetric)

(2)\
\begin{equation}
\Delta\phi_{0}+F(\phi_{0}^{2})\phi_{0}=0,\ \text{in }{\mathbf{R}}^{n}%
\ (n\geq2). \label{eqn-stationary}%
\end{equation}

(3)\ $0<\phi_{0}(r)<\sqrt{r_{0}},\ \forall r\in\lbrack0,\infty),\ $and
$\lim_{r\rightarrow\infty}\phi_{0}(r)=\sqrt{r_{0}}$

(4)\ $\phi_{0}^{\prime}(0)=0,\phi_{0}^{\prime}(r)>0\ \forall r\in(0,+\infty)$

(5) There exist $C>0,\delta>0$ such that: $\forall\alpha\in{\mathbf{N}}^{n}$
with $|\alpha|<2$,
\[
|\partial_{x}^{\alpha}(\phi_{0}(x)-\sqrt{r_{0}})|\leq Ce^{-\delta|x|},\forall
x\in{\mathbf{R}}^{n}.
\]

\end{theorem}

The steady solution $\phi_{0}$ constructed above is called a stationary bubble
of (\ref{eqn-GP-generalized}). To simplify notations, below we assume
$r_{0}=1,\ F^{\prime}\left(  1\right)  =-1$. Denote the operator
$A:H^{2}\left(  \mathbf{R}^{n}\right)  \rightarrow L^{2}\left(  \mathbf{R}%
^{n}\right)  $ by
\begin{equation}
A:=-\Delta-F(\phi_{0}^{2})-2F^{\prime}\left(  \phi_{0}^{2}\right)  \phi
_{0}^{2}. \label{defn-operator-A}%
\end{equation}
Note that $A$ is the linearized operator with the steady equation
(\ref{eqn-stationary}). Differentiating (\ref{eqn-stationary}) to $x_{i}$, we
get $\partial_{x_{i}}\phi_{0}\in\ker A$. We state the following non-degeneracy
condition
\begin{equation}
\ker A=\left\{  \partial_{x_{i}}\phi_{0},i=1,\cdots,n\right\}  .
\label{cond-NDG-stationary}%
\end{equation}

First, we study the two and three dimensional cases.

\begin{theorem}
\label{thm-existence-slow} Let $n=2,3.\ $Under assumptions (H1)-(H3), and the
condition (\ref{cond-NDG-stationary}), there exists $b_{0}>0$, such that for
any $c\in\left(  -b_{0},b_{0}\right)  $, there exist $\left(  \rho_{c}%
,\theta_{c}\right)  \in X_{2}^{s}$ such that%
\[
\phi^{c}\left(  x_{1}-ct,x^{\perp}\right)  =\left(  \left(  \rho_{0}+\rho
_{c}\right)  ^{\frac{1}{2}}e^{i\theta_{c}}\right)  \left(  x_{1}-ct,x^{\perp
}\right)
\]
is a cylindrically symmetric traveling wave solution of equation
(\ref{eqn-GP-generalized}). Here,$\ \sqrt{\rho_{0}}=\phi_{0}\left(  r\right)
$ is the stationary solution to (\ref{eqn-GP-generalized}). Moreover, $\left(
\rho_{c},\theta_{c}\right)  $ is $C^{1}$ for $c\in\left(  -b_{0},b_{0}\right)
,\ $%
\[
\Vert\rho_{c}\Vert_{H^{2}}+\Vert\theta_{c}\Vert_{\dot{H}^{2}}\leq K\left\vert
c\right\vert ,\ \text{for some }K>0.
\]

\end{theorem}

For $n=2,$ the non-degeneracy condition (\ref{cond-NDG-stationary}) is proved
for cubic-quintic nonlinearity in Appendix 2.

To prove the existence of traveling waves, we use the hydrodynamic formulation
\eqref{hrdrodynamic nls}.
The traveling wave solution
\[
\psi(x_{1}-ct,x_{\bot})=\sqrt{\rho}e^{i\theta}(x_{1}-ct,x_{\bot})
\]
satisfies
\begin{equation}%
\begin{cases}
-c\theta_{x_{1}}+|\nabla\theta|^{2}-\frac{1}{2}\frac{1}{\rho}\Delta\rho
+\frac{1}{4}\frac{1}{\rho^{2}}|\nabla\rho|^{2}-F(\rho)=0\\
c\rho_{x_{1}}-2\nabla.(\rho\nabla\theta)=0
\end{cases}
\label{eqn-3-5-TW-Madelung}%
\end{equation}
We define $S(\rho,\theta;c)$ to be the left-hand side of
(\ref{eqn-3-5-TW-Madelung}), then (\ref{eqn-3-5-TW-Madelung}) becomes
$S(\rho,\theta;c)=0$. First, we define several function spaces. Define the
spaces\
\[
Z:=L_{r_{\bot}}^{2}\cap\dot{H}_{r_{\bot}}^{-1}\,,\text{with norm }\left\Vert
\cdot\right\Vert _{Z}=\left\Vert \cdot\right\Vert _{2}+\left\Vert
\cdot\right\Vert _{\dot{H}^{-1}},
\]
and $Y:=L_{r_{\bot}}^{2}\times Z$. The energy functional is defined by
(\ref{energy-hrdrodynamic}) and the momentum is%
\[
P(\rho,\theta)=-\frac{1}{2}\int_{\mathbf{R}^{n}}\left(  \rho-1\right)
\theta_{x_{1}}\ dx,
\]
where $(\rho,\theta)\in\left(  \rho_{0},0\right)  +B_{\varepsilon_{0}}$ with
\[
B_{\varepsilon_{0}}=\left\{  (\rho,\theta)\in X_{2}^{s}\ |\ \left\Vert
\rho\right\Vert _{H^{2}}+\left\Vert \theta\right\Vert _{\dot{H}^{2}}%
\leq\varepsilon_{0}\right\}  .
\]
Since $H^{2}\left(  \mathbf{R}^{n}\right)  \hookrightarrow L^{\infty}$
$\left(  \mathbf{R}^{n}\right)  $ for $n=2,3$ and $\rho_{0}\left(  r\right)
\geq\rho_{0}\left(  0\right)  >0$, thus when $\varepsilon_{0}$ is small
enough, the functional $E$ and $P$ are well-defined.

\begin{proof}
[Proof of Theorem \ref{thm-existence-slow}]When $c$ is small enough, we look
for solutions of (\ref{eqn-3-5-TW-Madelung}) in the form $(\rho+\rho
_{0},\theta)$ where $(\rho,\theta)\in B_{\varepsilon_{0}}$ with $\varepsilon
_{0}$ small enough such that $\rho+\rho_{0}>0$. First, we note the following
variational structure of (\ref{eqn-3-5-TW-Madelung}):
\[
S(\rho+\rho_{0},\theta;c)=2D_{\left(  \rho,\theta\right)  }\left(  E(\rho
+\rho_{0},\theta)+cP(\rho+\rho_{0},\theta)\right)  .
\]
Since the functionals $E,P$ are translation invariant to $x_{1}$, the above
implies that
\begin{equation}
\left\langle S(\rho+\rho_{0},\theta;c),\left(
\begin{array}
[c]{c}%
\partial_{x_{1}}\left(  \rho+\rho_{0}\right) \\
\partial_{x_{1}}\theta
\end{array}
\right)  \right\rangle =0 \label{eqn-var-transl-invaraint}%
\end{equation}
for any $(\rho,\theta)\in B_{\varepsilon_{0}}$. Define $K_{0}=$ span$\left\{
(\partial_{x_{1}}\rho_{0}(x),0)\right\}  $. Let $K_{0}^{\bot}$ be the
orthogonal complement of $K_{0}$ in $Y$,\ and $\Pi^{\bot}:Y\mapsto K_{0}%
^{\bot}$ be the $L^{2}\ $orthogonal projection. We solve the equation
\begin{equation}
\Pi^{\bot}S(\rho_{0}+\rho,\theta;c)=0,\ \ (\rho,\theta)\in K_{0}^{\bot},
\label{eqn-reduced-TW}%
\end{equation}
near $(0,0;0)\ $by the implicit function theorem. The linearized operator of
$S$ with respect to $(\rho,\theta)\ $at $(0,0;0)$ is
\begin{equation}
D_{\left(  \rho,\theta\right)  }S(\rho_{0},0;0):=M_{0}=%
\begin{pmatrix}
M_{1} & 0\\
0 & M_{2}%
\end{pmatrix}
:\ \ X_{2}^{s}\mapsto Y\ , \label{defn-operator-M}%
\end{equation}
where
\begin{align}
M_{1}  &  =-\nabla\cdot(\frac{\nabla}{2\rho_{0}})-\frac{1}{2}\frac{1}{\rho
_{0}^{3}}|\nabla\rho_{0}|^{2}+\frac{\Delta\rho_{0}}{2\rho_{0}^{2}}-F^{\prime
}(\rho_{0}),\label{defn-M1-M2}\\
M_{2}  &  =-2\nabla\cdot(\rho_{0}\nabla).\nonumber
\end{align}
The linearized mapping of $\Pi^{\bot}S(\rho_{0}+\rho,\theta;c)|_{K_{0}^{\perp
}\cap X_{2}^{s}}$ at $(0,0;0)$ is $\Pi^{\bot}M_{0}|_{X_{2}^{s}\cap K_{0}%
^{\bot}}=M_{0}|_{X_{2}^{s}\cap K_{0}^{\bot}}$.\ It can be checked that for any
$\rho\in H^{2},$
\begin{equation}
M_{1}\rho=A\left(  \frac{\rho}{2\sqrt{\rho_{0}}}\right)  \frac{1}{\sqrt
{\rho_{0}}}, \label{relation-M1}%
\end{equation}
which also follows from (\ref{relation-2nd-variation-madelung}) below. So by
the assumption (\ref{cond-NDG-stationary}), $\ $%
\[
\ker M_{1}=span\left\{  \sqrt{\rho_{0}}\partial_{x_{1}}\phi_{0}\right\}
=span\left\{  \partial_{x_{1}}\rho_{0}\right\}  ,\ \text{on }H_{r_{\bot}}%
^{2}.
\]
Moreover, $\left(  M_{2}\theta,\theta\right)  >0$ for any $\theta\in\dot
{H}^{1}$. Thus $\ker M=\left\{  (\partial_{x_{1}}\rho_{0}(x),0)\right\}
=K_{0}$. By Lemma \ref{lemma-inverse-bound} below, the operator $M_{0}%
:X_{2}^{s}\cap K_{0}^{\bot}\rightarrow K_{0}^{\bot}$ is bounded with a bounded
inverse. Moreover, it is easy to show that $S(\rho+\rho_{0},\theta;c)\in
C^{1}\left(  B_{\varepsilon_{0}}\times\mathbf{R;}Y\right)  $. Thus by the
Implicit Function Theorem (\cite{chow-hale}), there exists a neighborhood
$B_{\delta_{0}}\times\left(  -b_{0},b_{0}\right)  $ of $\left(  0,0;0\right)
\ $in $\left(  X_{2}^{s}\cap K_{0}^{\bot}\right)  \times\mathbf{R}$ such that
\[
\left(  \rho\left(  c\right)  ,\theta\left(  c\right)  \right)  :\left(
-b_{0},b_{0}\right)  \rightarrow B_{\delta_{0}}%
\]
is the unique solution to (\ref{eqn-reduced-TW}) near $\left(  0,0;0\right)  $
which is $C^{1}$ in $c$. Moreover, as implied by the proof of IFT, we have
\[
\left\Vert \rho\left(  c\right)  \right\Vert _{H^{2}}+\left\Vert \theta\left(
c\right)  \right\Vert _{\dot{H}^{2}}\leq K\left\Vert S(\rho_{0}%
,0;c)\right\Vert _{Y}\leq K\left\vert c\right\vert ,
\]
for some constant $K$. We claim that $\left(  \rho\left(  c\right)
,\theta\left(  c\right)  \right)  $ solves the original problem, that is,
$S(\rho\left(  c\right)  +\rho_{0},\theta\left(  c\right)  ;c)=0$. Indeed, by
the equation (\ref{eqn-reduced-TW}), we have
\[
S(\rho\left(  c\right)  +\rho_{0},\theta\left(  c\right)  ;c)=k(\partial
_{x_{1}}\rho_{0}(x),0)
\]
for some constant $k$. We claim that $k=0$. Suppose otherwise $k\neq0$, then
by (\ref{eqn-var-transl-invaraint}),
\[
\left\langle \left(
\begin{array}
[c]{c}%
\partial_{x_{1}}\rho_{0}\\
0
\end{array}
\right)  ,\left(
\begin{array}
[c]{c}%
\partial_{x_{1}}\left(  \rho\left(  c\right)  +\rho_{0}\right) \\
\partial_{x_{1}}\theta\left(  c\right)
\end{array}
\right)  \right\rangle =0,
\]
or $\left\Vert \partial_{x_{1}}\rho_{0}\right\Vert _{L^{2}}^{2}+O\left(
c\right)  =0$ which is a contradiction. This finishes the proof of the Theorem.
\end{proof}

It remains to show that the operator $M|_{X_{2}^{s}\cap K_{0}^{\bot}}\ $has a
bounded inverse. We study this in a more general setting. For $0<c_{0}%
<\sqrt{2},\ $suppose $\left(  \rho_{c_{0}},\theta_{c_{0}}\right)  $ is a
traveling wave solution satisfying (\ref{eqn-3-5-TW-Madelung}) and $\min
\rho_{c_{0}}>0$. The linearized operator of $S(\rho+\rho_{c_{0}},\theta
+\theta_{c_{0}};c)$ at $\left(  0,0;c_{0}\right)  $ is
\begin{equation}
D_{\left(  \rho,\theta\right)  }S(\rho_{c_{0}},\theta_{c_{0}};c_{0}%
):=M_{c_{0}}=%
\begin{pmatrix}
M_{11} & M_{12}\\
M_{21} & M_{22}%
\end{pmatrix}
:\ \ X_{2}^{s}\mapsto Y\ . \label{defn-M}%
\end{equation}
Here,
\begin{equation}
M_{11}=-\nabla\cdot(\frac{\nabla}{2\rho_{c_{0}}})-\frac{1}{2}\frac{1}%
{\rho_{c_{0}}^{3}}|\nabla\rho_{c_{0}}|^{2}+\frac{\Delta\rho_{c_{0}}}%
{2\rho_{c_{0}}^{2}}-F^{\prime}(\rho_{c_{0}}), \label{defn-M11}%
\end{equation}

\[
M_{22}=-2\nabla\cdot(\rho_{c_{0}}\nabla),
\]
and%
\begin{align}
M_{21}  &  =c_{0}\partial_{x_{1}}-2\nabla\cdot\left(  \nabla\theta_{c_{0}%
}\cdot\right)  ,\ \label{defn-M12}\\
M_{12}  &  =M_{21}^{\ast}=-c_{0}\partial_{x_{1}}+2\nabla\theta_{c_{0}}%
\cdot\nabla.\nonumber
\end{align}
Define $K_{c_{0}}=\left\{  \left(  \partial_{x_{1}}\rho_{c_{0}},\partial
_{x_{1}}\theta_{c_{0}}\right)  \right\}  $. Let $K_{c_{0}}^{\bot}$ be the
orthogonal complement of $K_{c_{0}}$ in $Y$,\ and $\Pi_{c_{0}}^{\bot}:Y\mapsto
K_{c_{0}}^{\bot}$ be the orthogonal projection. Note that $\Pi_{c_{0}}^{\bot
}M_{c_{0}}|_{K_{c_{0}}^{\bot}\cap X_{2}^{s}}=M_{c_{0}}|_{K_{c_{0}}^{\bot}\cap
X_{2}^{s}}.$

\begin{lemma}
\label{lemma-inverse-bound}Assume
\begin{equation}
\ker M_{c_{0}}=span\left\{  \left(  \partial_{x_{1}}\rho_{c_{0}}%
,\partial_{x_{1}}\theta_{c_{0}}\right)  \right\}  \text{ on }X_{2}^{s}.
\label{cond-NDG-c0}%
\end{equation}
Then there exists $\gamma>0$, such that for any $(\rho,\theta)\in K_{c_{0}%
}^{\bot}\cap X_{2}^{s}$,
\begin{equation}
\left\Vert M_{c_{0}}(\rho,\theta)\right\Vert _{Y}\geq\gamma\left\Vert
(\rho,\theta)\right\Vert _{X_{2}^{s}}. \label{estimate-inverse-bound}%
\end{equation}
In particular, $M_{c_{0}}|_{K_{c_{0}}^{\bot}\cap X_{2}^{s}}:K_{c_{0}}^{\bot
}\cap X_{2}^{s}\rightarrow K_{c_{0}}^{\bot}$ is invertible and
\[
\left\Vert M_{c_{0}}|_{K_{c_{0}}^{\bot}\cap X_{2}^{s}}^{-1}\right\Vert
\leq\gamma^{-1}.
\]

\end{lemma}

\begin{proof}
We follow the arguments of the proof of Proposition 2.3 in \cite{fw86}.
Suppose (\ref{estimate-inverse-bound}) is not true. Then there exists a
sequence
\[
\psi_{n}=\left(  \rho_{n},\theta_{n}\right)  \in K_{c_{0}}^{\bot}\cap
X_{2}^{s},\ \ n=1,2,\cdots,
\]
such that $\left\Vert \psi_{n}\right\Vert _{X_{2}}=1$ and $\left\Vert
M_{c_{0}}\psi_{n}\right\Vert _{Y}\rightarrow0$ when $n\rightarrow\infty$.
Since $\left\Vert \psi_{n}\right\Vert _{X_{2}}=1$, we may assume (by passing
to a subsequence) that $\psi_{n}\rightarrow\psi_{\infty}$ weakly in $X_{2}%
^{s}\,$\ for some $\psi_{\infty}\in X_{2}^{s}$. The fact that $\left\{
\psi_{n}\right\}  $ is orthogonal to $K_{c_{0}}$ implies that $\psi_{\infty
}\in K_{c_{0}}^{\bot}\cap X_{2}^{s}$. The weak convergence of $\psi_{n}$ to
$\psi_{\infty}$ in $X_{2}^{s}$ implies the weak convergence of $M_{c_{0}}%
\psi_{n}$ to $M_{c_{0}}\psi_{\infty}$ in $L^{2}$. It follows from $\left\Vert
M_{c_{0}}\psi_{n}\right\Vert _{Y}\rightarrow0$ that $M_{c_{0}}\psi_{\infty}%
=0$. Thus $\psi_{\infty}=0$ since $\psi_{\infty}$ is orthogonal to $\ker
M_{c_{0}}=K_{c_{0}}$.

Denote the operator
\begin{equation}
M_{c_{0}}^{\infty}=%
\begin{pmatrix}
M_{11}^{\infty} & -c_{0}\partial_{x_{1}}\\
c_{0}\partial_{x_{1}} & M_{22}%
\end{pmatrix}
,\text{ with }M_{11}^{\infty}=-\nabla.(\frac{\nabla}{2\rho_{c_{0}}})+\frac
{1}{\rho_{c_{0}}}. \label{defn-M-infinity}%
\end{equation}
We claim that: $\left\Vert M_{c_{0}}^{\infty}\psi_{n}\right\Vert
_{Y}\rightarrow0$ when $n\rightarrow\infty$. First, $\psi_{n}\rightarrow0$
weakly in $X_{2}$ implies that $\rho_{n}\rightarrow0$ weakly in $H^{2}$ and
$\theta_{n}\rightarrow0$ weakly in $\dot{H}^{2}$. So for any bounded function
$a\left(  x\right)  $ decaying at infinity, we have $a\left(  x\right)
\rho_{n},\ a\left(  x\right)  \nabla\rho_{n},\ a\left(  x\right)  \nabla
\theta_{n}\rightarrow0$ strongly in $L^{2}$, since the restriction of
$\rho_{n},\nabla\rho_{n},\nabla\theta_{n}$ to a bounded domain implies strong
convergence. Thus, we have
\[
\left\Vert \left(  M_{11}-M_{11}^{\infty}\right)  \rho_{n}\right\Vert _{L^{2}%
}\rightarrow0,
\]%
\[
\left\Vert \left(  M_{21}-c_{0}\partial_{x_{1}}\right)  \rho_{n}\right\Vert
_{L^{2}}=2\left\Vert \nabla\cdot\left(  \nabla\theta_{c_{0}}\rho_{n}\right)
\right\Vert _{L^{2}}\rightarrow0,
\]%
\[
\left\Vert \left(  M_{12}+c_{0}\partial_{x_{1}}\right)  \theta_{n}\right\Vert
_{L^{2}}=2\left\Vert \nabla\theta_{c_{0}}\cdot\nabla\theta_{n}\right\Vert
_{L^{2}}\rightarrow0,
\]
and
\[
\left\Vert \left(  M_{21}-c_{0}\partial_{x_{1}}\right)  \rho_{n}\right\Vert
_{\dot{H}^{-1}}\leq C\left\Vert \nabla\theta_{c_{0}}\rho_{n}\right\Vert
_{L^{2}}\rightarrow0.
\]
This shows that $\left\Vert \left(  M_{c_{0}}-M_{c_{0}}^{\infty}\right)
\psi_{n}\right\Vert _{Y}\rightarrow0$ and thus $\left\Vert M_{c_{0}}^{\infty
}\psi_{n}\right\Vert _{Y}\rightarrow0$. By Lemma \ref{lemma-M-infinity} below,
there exists $\eta>0$ such that
\[
\left\Vert M_{c_{0}}^{\infty}\psi_{n}\right\Vert _{Y}\geq\eta\left\Vert
\psi_{n}\right\Vert _{X_{2}}=\eta.
\]
This contradiction proves the lemma.
\end{proof}

\begin{lemma}
\label{lemma-M-infinity}Assume $0<c_{0}<\sqrt{2}$ and $\inf\rho_{c_{0}}\left(
x\right)  =\delta_{0}>0$. Then there exists $\eta>0$ such that
\begin{equation}
\left\Vert M_{c_{0}}^{\infty}\psi\right\Vert _{Y}\geq\eta\left\Vert
\psi\right\Vert _{X_{2}}, \label{estimate-M-infinity}%
\end{equation}
for any $\psi\in X_{2}^{s}$.
\end{lemma}

\begin{proof}
Take any $\psi=\left(  \rho,\theta\right)  \in X_{2}^{s}.$ First, we estimate
$\left\Vert \psi\right\Vert _{H^{1}\times\dot{H}^{1}}$ as in the proof of
Lemma \ref{le-6}. Since $0<c_{0}<\sqrt{2}$, we can choose $0<a_{0}<1$ such
that $2-\frac{c_{0}^{2}}{a_{0}^{2}}>0$. Then,
\begin{align}
\left\langle M_{c_{0}}^{\infty}\psi,\psi\right\rangle  &  =\int_{\mathbf{R}%
^{n}}\left[  \frac{1}{2\rho_{c_{0}}}\left\vert \nabla\rho\right\vert
^{2}+\frac{1}{\rho_{c_{0}}}\rho^{2}-2c_{0}\rho\partial_{x_{1}}\theta
+2\rho_{c_{0}}\left\vert \nabla\theta\right\vert ^{2}\right]
dx\label{estimate-quadratic-m-infty}\\
&  =\int_{\mathbf{R}^{n}}[\frac{1}{2\rho_{c_{0}}}\left\vert \nabla
\rho\right\vert ^{2}+\frac{1}{\rho_{c_{0}}}\left(  1-a_{0}^{2}\right)
\rho^{2}+2\rho_{c_{0}}\left\vert \nabla^{\perp}\theta\right\vert
^{2}\nonumber\\
&  \ \ \ \ \ \ +(2-\frac{c_{0}^{2}}{a_{0}^{2}})\rho_{c_{0}}\left\vert
\partial_{x_{1}}\theta\right\vert ^{2}+\left(  \frac{a_{0}\rho}{\sqrt
{\rho_{c_{0}}}}-\frac{c_{0}}{a_{0}}\partial_{x_{1}}\theta\sqrt{\rho_{c_{0}}%
}\right)  ^{2}]\ dx\nonumber\\
&  \geq\eta_{0}\left(  \left\Vert \rho\right\Vert _{H^{1}}^{2}+\left\Vert
\theta\right\Vert _{\dot{H}^{1}}^{2}\right)  ,\text{ for some }\eta
_{0}>0\text{. }\nonumber
\end{align}
So
\[
\eta_{0}\left(  \left\Vert \rho\right\Vert _{H^{1}}^{2}+\left\Vert
\theta\right\Vert _{\dot{H}^{1}}^{2}\right)  \leq\left\Vert M_{c_{0}}^{\infty
}\psi\right\Vert _{L^{2}\times\dot{H}^{-1}}\left(  \left\Vert \rho\right\Vert
_{L^{2}}+\left\Vert \theta\right\Vert _{\dot{H}^{1}}\right)  ,
\]
and thus
\[
\frac{1}{2}\eta_{0}\left(  \left\Vert \rho\right\Vert _{H^{1}}+\left\Vert
\theta\right\Vert _{\dot{H}^{1}}\right)  \leq\left\Vert M_{c_{0}}^{\infty}%
\psi\right\Vert _{L^{2}\times\dot{H}^{-1}}.
\]
From the standard elliptic estimates, there exists $C>0$ such that
\[
\left\Vert \nabla^{2}\rho\right\Vert _{L^{2}}+\left\Vert \nabla^{2}%
\theta\right\Vert _{L^{2}}\leq C\left(  \left\Vert \rho\right\Vert _{H^{1}%
}+\left\Vert \theta\right\Vert _{\dot{H}^{1}}+\left\Vert M_{c_{0}}^{\infty
}\psi\right\Vert _{L^{2}\times L^{2}}\right)  .
\]
Combining above two inequalities, we get (\ref{estimate-M-infinity}).
\end{proof}

For dimension $n\geq4,\ $we need to study the equation in the function space
of higher regularity. Choose $k>\frac{n}{2}$ such that $H^{k}\left(
\mathbf{R}^{n}\right)  \hookrightarrow L^{\infty}\left(  \mathbf{R}%
^{n}\right)  $. Let $Y_{k}=H_{r_{\bot}}^{k-2}\times\left(  H_{r_{\bot}}%
^{k-2}\cap\dot{H}_{r_{\bot}}^{-1}\right)  ,$We construct traveling waves near
stationary bubbles by solving the equation $S(\rho+\rho_{0},\theta;c)=0$ in
the space $X_{k}^{s}$. Assuming $F\in C^{k-1}$, from Lemma
\ref{lemma-inverse-bound}, by bootstrapping we get the estimate
\[
\left\Vert M_{c_{0}}(\rho,\theta)\right\Vert _{Y_{k}}\geq\gamma\left\Vert
(\rho,\theta)\right\Vert _{X_{k}}%
\]
and thus $M_{c_{0}}|_{K_{c_{0}}^{\bot}\cap X_{k}^{s}}:K_{c_{0}}^{\bot}\cap
X_{k}^{s}\rightarrow K_{c_{0}}^{\bot}\cap Y_{k}$ is invertible. Then by the
same proof of Theorem \ref{thm-existence-slow}, we get the existence of slow
traveling waves near $\rho_{0}$ for $n\geq4$ in the space $X_{k}^{s}.\ $

\begin{remark}
The two non-degeneracy conditions (\ref{cond-NDG-c0}) and
(\ref{assumption-NDG-s}) are equivalent. This can be seen from the relation of
operators $M_{c_{0}}$ and $L_{c_{0}}$. For a traveling wave $U_{c_{0}}%
=\sqrt{\rho_{c_{0}}}e^{i\theta_{c_{0}}}$ with no vortices, denote the matrix
operator
\begin{equation}
T_{c_{0}}=\left(
\begin{array}
[c]{cc}%
\frac{1}{2}\frac{1}{\sqrt{\rho_{c_{0}}}}\cos\theta_{c_{0}} & -\sqrt
{\rho_{c_{0}}}\sin\theta_{c_{0}}\\
\frac{1}{2}\frac{1}{\sqrt{\rho_{c_{0}}}}\sin\theta_{c_{0}} & \sqrt{\rho
_{c_{0}}}\cos\theta_{c_{0}}%
\end{array}
\right)  . \label{definition T_c}%
\end{equation}
Then
\begin{equation}
M_{c_{0}}=2T_{c_{0}}^{t}L_{c_{0}}T_{c_{0}}.
\label{relation-2nd-variation-madelung}%
\end{equation}
Since $T_{c_{0}}$ is obviously an isomorphism of $X$, we have
\[
T_{c_{0}}\left(  \ker M_{c_{0}}\right)  =\ker L_{c_{0}}%
\]
which implies the equivalence of (\ref{cond-NDG-c0}) and
(\ref{assumption-NDG-s}). To show (\ref{relation-2nd-variation-madelung}), we
note that: 1) $M_{c_{0}}$ and $L_{c_{0}}$ are from the second variation of the
energy-momentum functional $2\left(  E+cP\right)  $ in $\left(  \rho
,\theta\right)  \ $and $E+cP$ in $\left(  u,v\right)  $ respectively$\ $and 2)
the first order variations of $\left(  u,v\right)  $ and $\left(  \rho
,\theta\right)  $ are related by the matrix $T_{c_{0}}$.
\end{remark}

\begin{remark}
The existence of slow traveling waves for cubic-quintic type equations was
proved for $n\geq4$ in \cite{maris-4d-slow} by using the critical point
theory, and for $n=2,3$ in an unpublished manuscript of Z. Lin
(\cite{lin-note-99}) by using the hydrodynamic formulation and
Liapunov-Schmidt reduction. The proof we give here adapts the formulation of
(\cite{lin-note-99}), but it is much simpler and works for any dimension
$n\geq2$. The new observation is to use the variational structure of the
traveling wave equation (\ref{eqn-3-5-TW-Madelung}) in hydrodynamic variables
to reduce it to equation (\ref{eqn-reduced-TW}) which is solved by the
implicit function theorem. Moreover, as a corollary of the proof we get the
local uniqueness and differentiability of the traveling wave branch.
\end{remark}

\bigskip

\subsection{Continuation of traveling waves}

By using Lemma \ref{lemma-inverse-bound} and the proof of Theorem
\ref{thm-existence-slow}, we get the following result on the continuation of
traveling waves without vortices (i.e. $\left\vert U_{c}\right\vert \neq0$).

\begin{proposition}
For $n\geq2$, fix $k>\frac{n}{2}$ and assume $F\in C^{k-1}$, $0<c_{0}<\sqrt
{2}$, $\left(  \rho_{c_{0}},\theta_{c_{0}}\right)  $ is a cylindrically
symmetric traveling wave of (\ref{eqn-3-5-TW-Madelung}) satisfying $\inf
\rho_{c_{0}}\left(  x\right)  >0$ and the non-degeneracy condition
(\ref{cond-NDG-c0}). Then $\exists$ $\varepsilon_{0}>0$, such that for%
\[
c\in\left(  -\varepsilon_{0}+c_{0},\varepsilon_{0}+c_{0}\right)
\subset\left(  0,\sqrt{2}\right)  ,
\]
there exists a locally unique $C^{1}$ solution curve $\left(  \rho_{c}%
,\theta_{c}\right)  $ of (\ref{eqn-3-5-TW-Madelung}), with $\left(  \rho
_{c}-\rho_{c_{0}},\theta_{c}\right)  \in X_{k}^{s}$. That is,
\[
\phi^{c}\left(  x_{1}-ct,,x^{\perp}\right)  =\left(  \sqrt{\rho_{c}}%
e^{i\theta_{c}}\right)  \left(  x_{1}-ct,,x^{\perp}\right)
\]
are the only traveling wave solutions of (\ref{eqn-GP-generalized}) near
$\left(  \rho_{c_{0}},\theta_{c_{0}}\right)  .$
\end{proposition}

For $n=3$, we can prove the continuation of general traveling waves even with
vortices, under the non-degeneracy condition (\ref{assumption-NDG-s}). Instead
of using the hydrodynamic formulation, this is achieved by using the original
equation (\ref{eqn-GP-G-TW}). First, we need an analogue of Lemma
\ref{lemma-inverse-bound}. We still use $X,\ Y$ for the cylindrical symmetric
spaces defined before.

\begin{lemma}
\label{lemma-inverse-bound-3d} For $n=3$ and $0\leq c_{0}<\sqrt{2},$ let
$U_{c_{0}}=u_{c_{0}}+iv_{c_{0}}$ be a traveling wave solution of
(\ref{eqn-GP-G-TW}) satisfying the decay condition (\ref{decay-TW-g}). Let
$L_{c_{0}}:X_{2}^{s}\rightarrow Y$ be the operator defined in
(\ref{operator-Lc-g}). Assume (\ref{assumption-NDG-s}), i.e.,
\[
\ker L_{c_{0}}=\bar{K}_{c_{0}}=\left\{  \partial_{x_{1}}U_{c_{0}}\right\}
,\text{ on }X_{2}^{s},
\]
and denote $\bar{K}_{c_{0}}^{\bot}$ to be the $L^{2}$ orthogonal complement of
$\bar{K}_{c_{0}}$ in $Y$. Then, there exists $\gamma>0$, such that
\[
\left\Vert L_{c_{0}}\phi\right\Vert _{Y}\geq\gamma\left\Vert \phi\right\Vert
_{X_{2}^{s}},\text{ for any }\phi\in\bar{K}_{c_{0}}^{\bot}\cap X_{2}^{s}.
\]
In particular, $L_{c_{0}}:\bar{K}_{c_{0}}^{\bot}\cap X_{2}^{s}\rightarrow
\bar{K}_{c_{0}}^{\bot}$ is invertible and
\[
\left\Vert \Big(L_{c_{0}}|_{\bar{K}_{c_{0}}^{\bot}\cap X_{2}^{s}}%
\Big)^{-1}\right\Vert \leq\gamma^{-1}.
\]

\end{lemma}

\begin{proof}
The proof is almost the same as that of Lemma \ref{lemma-inverse-bound}. So we
only point out some key points in the proof. For any sequence $\left\{
\psi_{n}\right\}  \in X_{2}^{s}$ with $\left\Vert \psi_{n}\right\Vert _{X_{2}%
}=1$ and $\psi_{n}\rightarrow0$ weakly in $X_{2}^{s}$, we show that
\[
\left\Vert \left(  L_{c_{0}}-L_{c_{0},\infty}\right)  \psi_{n}\right\Vert
_{Y}\rightarrow0,
\]
where
\[
L_{c_{0},\infty}:=\left(
\begin{array}
[c]{cc}%
-\Delta+2 & -c_{0}\partial_{x_{1}}\\
c_{0}\partial_{x_{1}} & -\Delta
\end{array}
\right)  .
\]
Let $L_{c_{0}}=\left(  L^{ij}\right)  $ and $L_{c_{0},\infty}=\left(
L_{\infty}^{ij}\right)  ,\ i,j=1,2$, and $\psi_{n}=u_{n}+iv_{n}$. By
(\ref{decay-TW-g}),
\[
L^{ij}-L_{\infty}^{ij}=a^{ij}\left(  x\right)  =o\left(  \frac{1}{\left\vert
x\right\vert }\right)  ,\text{ }a^{22}\left(  x\right)  =o\left(  \frac
{1}{\left\vert x\right\vert ^{2}}\right)  .
\]
Then, since $u_{n}\rightarrow0$ weakly in $H^{2},$
\[
\left\Vert \left(  L^{11}-L_{\infty}^{11}\right)  u_{n}\right\Vert _{L^{2}%
}=\left\Vert a^{11}\left(  x\right)  u_{n}\right\Vert _{L^{2}}\rightarrow0,
\]%
\[
\left\Vert \left(  L^{21}-L_{\infty}^{21}\right)  u_{n}\right\Vert _{L^{2}%
}=\left\Vert a^{21}\left(  x\right)  u_{n}\right\Vert _{L^{2}}\rightarrow0,
\]%
\[
\left\Vert \left(  L^{21}-L_{\infty}^{21}\right)  u_{n}\right\Vert _{\dot
{H}^{-1}}\leq\left\Vert \left\vert x\right\vert a^{21}\left(  x\right)
u_{n}\right\Vert _{L^{2}}\rightarrow0,
\]
by the local compactness of $H^{2}\hookrightarrow L^{2}$. Since $v_{n}%
\rightarrow0$ weakly in $\dot{H}^{2},$ we have
\[
\left\Vert \left(  L^{12}-L_{\infty}^{12}\right)  v_{n}\right\Vert _{L^{2}%
}=\left\Vert a^{12}\left(  x\right)  v_{n}\right\Vert _{L^{2}}\rightarrow0,
\]%
\[
\left\Vert \left(  L^{22}-L_{\infty}^{22}\right)  v_{n}\right\Vert _{L^{2}%
}=\left\Vert a^{22}\left(  x\right)  v_{n}\right\Vert _{L^{2}}\rightarrow0,
\]%
\[
\left\Vert \left(  L^{22}-L_{\infty}^{22}\right)  v_{n}\right\Vert _{\dot
{H}^{-1}}\leq\left\Vert \left\vert x\right\vert a^{22}\left(  x\right)
v_{n}\right\Vert _{L^{2}}\rightarrow0,
\]
by the arguments in the proof of Lemma \ref{lemma-compactness}.

By Lemma \ref{le-6}, there exists $\eta_{0}>0$ such that
\[
\left\langle L_{c_{0},\infty}\phi,\phi\right\rangle \geq\eta_{0}\left\Vert
\phi\right\Vert _{H^{1}\times\dot{H}^{1}}^{2},\text{ for any }\phi\in
H^{1}\times\dot{H}^{1}.
\]
Then by the same proof of Lemma \ref{lemma-M-infinity}, for some $\eta>0,$
\[
\left\Vert L_{c_{0},\infty}\phi\right\Vert _{Y}\geq\eta\left\Vert
\phi\right\Vert _{X_{2}},\ \ \text{for any }\phi\in X_{2}^{s}.
\]
The rest of the proof is the same as Lemma \ref{lemma-inverse-bound}.
\end{proof}

\begin{theorem}
\label{thm-continuation-3d}For $0<c_{0}<\sqrt{2}$, assume $U_{c_{0}}%
=\psi\left(  w_{c_{0}}\right)  $ is a cylindrical symmetric $3$D$\ $traveling
wave solution of (\ref{eqn-GP-generalized}) satisfying the non-degeneracy
condition (\ref{assumption-NDG-s}). Then $\exists$ $\varepsilon_{0}>0$, such
that for%
\[
c\in\left(  -\varepsilon_{0}+c_{0},\varepsilon_{0}+c_{0}\right)
\subset\left(  0,\sqrt{2}\right)  ,
\]
there exists a locally unique $C^{1}$ solution curve $U_{c}=\psi\left(
w_{c}\right)  \ $of (\ref{eqn-GP-G-TW}) near $U_{c_{0}}$, where $w_{c}\in
H_{r_{\bot}}^{2}({\mathbf{R}}^{3},{\mathbf{R}})\times\dot{H}_{r_{\bot}}%
^{2}({\mathbf{R}}^{3},{\mathbf{R}})\ $and $\left\Vert w_{c}-w_{c_{0}%
}\right\Vert =O\left(  \left\vert c-c_{0}\right\vert \right)  $.
\end{theorem}

\begin{proof}
By Lemma \ref{le-2}, for a traveling wave solution $U_{c}=\psi\left(
w_{c}\right)  ,$ it is equivalent to solve $(\bar{E}_{c})^{\prime}(w_{c})=0$
for $w_{c}$. Define $K_{c_{0}}=span\left\{  \partial_{x_{1}}w_{c_{0}}\right\}
$. Let $K_{c_{0}}^{\bot}$ be the $L^{2}\ $orthogonal complement of $K_{c_{0}}$
in $Y$,\ and $\Pi_{c_{0}}^{\bot}:Y\mapsto K_{c_{0}}^{\bot}$ be the orthogonal
projection. Let $\tilde{K}_{c_{0}}^{\bot}~$be orthogonal complement of
$K_{c_{0}}$ in $X_{2}^{s}$, in the inner product $\left[  \cdot,\cdot\right]
=\left(  K\cdot,K\cdot\right)  ,$ where the operator%
\[
K\left(
\begin{array}
[c]{c}%
\phi_{1}\\
\phi_{2}%
\end{array}
\right)  =\left(
\begin{array}
[c]{c}%
\phi_{1}-\chi(D)\left(  v_{c}\phi_{2}\right) \\
\phi_{2}%
\end{array}
\right)
\]
is defined in (\ref{definition-K}). We use the implicit function theorem to
find solutions $\left(  w_{c_{0}}+w,c\right)  $ near $\left(  w_{c_{0}}%
,c_{0}\right)  $ of the equation%
\[
\Pi_{c_{0}}^{\bot}\bar{E}_{c}{}^{\prime}(w_{c_{0}}+w)=0,\text{ }w\in\tilde
{K}_{c_{0}}^{\bot}\text{. }%
\]
The linearized operator with respect to $w$ of the left hand side above at
$\left(  w_{c_{0}},c_{0}\right)  $ is
\[
\Pi_{c_{0}}^{\bot}\bar{E}_{c_{0}}{}^{\prime\prime}(w_{c_{0}})|_{\tilde
{K}_{c_{0}}^{\bot}}=\bar{E}_{c_{0}}{}^{\prime\prime}(w_{c_{0}})|_{\tilde
{K}_{c_{0}}^{\bot}}:\tilde{K}_{c_{0}}^{\bot}\rightarrow K_{c_{0}}^{\bot},
\]
which will be shown to be invertible below. In fact, by
(\ref{relation-quadratic forms}), we have
\[
\bar{E}_{c_{0}}{}^{\prime\prime}(w_{c_{0}})=K^{\ast}L_{c_{0}}K,
\]
where $K^{\ast}$ is given by (\ref{defn-K*}). So by (\ref{assumption-NDG-s}),%
\[
\ker\bar{E}_{c_{0}}{}^{\prime\prime}(w_{c_{0}})=K_{c_{0}}=\left\{
\partial_{x_{1}}w_{c_{0}}\right\}
\]
and
\[
\Pi_{c_{0}}^{\bot}\bar{E}_{c_{0}}{}^{\prime\prime}(w_{c_{0}})|_{\tilde
{K}_{c_{0}}^{\bot}}=\bar{E}_{c_{0}}{}^{\prime\prime}(w_{c_{0}})|_{\tilde
{K}_{c_{0}}^{\bot}}.
\]
By the definition of $\tilde{K}_{c_{0}}^{\bot}$, $\phi\in\tilde{K}_{c_{0}%
}^{\bot}$ iff $K\phi\in\bar{K}_{c_{0}}^{\bot}\cap X_{2}^{s}$ where $\bar
{K}_{c_{0}}$ is defined in Lemma \ref{lemma-inverse-bound-3d}. By Lemma
\ref{lemma-inverse-bound-3d}, there exists $\gamma>0$, such that
\begin{equation}
\left\Vert L_{c_{0}}K\phi\right\Vert _{Y}\geq\gamma\left\Vert K\phi\right\Vert
_{X_{2}},\text{ for any }\phi\in\tilde{K}_{c_{0}}^{\bot}.
\label{inequality-L-c0}%
\end{equation}
It is easy to show that the mappings $K:X_{2}^{s}\rightarrow X_{2}^{s}$ and
$K^{\ast}:Y\rightarrow Y$ are isometric, that is, there exist $C_{1},C_{2}>0,$
such that
\[
C_{1}\left\Vert \phi\right\Vert _{X_{2}}\leq\left\Vert K\phi\right\Vert
_{X_{2}}\leq C_{2}\left\Vert \phi\right\Vert _{X_{2}}%
\]
and
\[
C_{1}\left\Vert \phi\right\Vert _{Y}\leq\left\Vert K^{\ast}\phi\right\Vert
_{Y}\leq C_{2}\left\Vert \phi\right\Vert _{Y}.
\]
Thus by (\ref{inequality-L-c0}), there exists some $\gamma_{1}>0,$ such that
\[
\left\Vert K^{\ast}L_{c_{0}}K\phi\right\Vert _{Y}\geq\gamma_{1}\left\Vert
\phi\right\Vert _{X_{2}},\text{ for any }\phi\in\tilde{K}_{c_{0}}^{\bot}.
\]
That is, the operator
\[
\bar{E}_{c_{0}}{}^{\prime\prime}(w_{c_{0}})|_{\tilde{K}_{c_{0}}^{\bot}%
}=K^{\ast}L_{c_{0}}K\phi|_{\tilde{K}_{c_{0}}^{\bot}}:\tilde{K}_{c_{0}}^{\bot
}\rightarrow K_{c_{0}}^{\bot}%
\]
has a bounded inverse. The rest of the proof is the same as that of Theorem
\ref{thm-existence-slow}. So we skip it.
\end{proof}

\subsection{Instability of slow traveling waves}

In this subsection, we prove the instability of slow traveling waves
constructed in Theorem \ref{thm-existence-slow}. The approach is the same as
developed in Section 3. The linearized equation is
(\ref{eqn-linearized-madelung}). To find unstable eigenvalues of $JM_{c}$, we
first study the quadratic form $\left\langle M_{c}u,u\right\rangle $, where
$u=\left(  \rho,\theta\right)  \in X_{1}^{s}$, the cylindrical symmetric
subspace of $X_{1}=H^{1}\left(  \mathbf{R}^{n}\right)  \times\dot{H}%
^{1}\left(  \mathbf{R}^{n}\right)  $.


\begin{proposition}
\label{prop-quadratic-madelung} Assume (H1)-(H3) and the non-degeneracy
condition (\ref{cond-NDG-stationary}) on the stationary bubble. Then $\exists$
$a_{0}\in\left(  0,\sqrt{2}\right)  $, such that for any $0\leq c<a_{0}%
,~$there exists a traveling wave solution $U_{c}=\sqrt{\rho_{c}}e^{i\theta
_{c}}$ of (\ref{eqn-GP-generalized}) without vortices and
(\ref{assumption-NDG-s}) is satisfied. Moreover, the space $X_{1}^{s}$ can be
decomposed as a direct sum%
\[
X_{1}^{s}=N\oplus Z \oplus P,
\]
where $Z=\left\{  \left(  \partial_{x_{1}}\rho_{c},\partial_{x_{1}}\theta
_{c}\right)  \right\}  $, $N$ is a one-dimensional subspace such that
$\left\langle M_{c}u,u\right\rangle <0$ for $0\neq u\in N$, and $P$ is a
closed subspace such that
\[
\left\langle M_{c}u,u\right\rangle \geq\delta\left\Vert u\right\Vert _{X_{1}%
}^{2},\ \forall\text{ }u\in P,
\]
for some constant $\delta>0.$
\end{proposition}

\begin{proof}
The proof is similar to that of Proposition \ref{prop-quadratic}, so we only
sketch it. The existence of traveling waves is shown in Theorem
\ref{thm-existence-slow}, for $c\in\lbrack0,b_{0}),\ b_{0}>0$. Define the
operator
\[
\tilde{M}_{c}:=\tilde{G}\circ M_{c}\circ\tilde{G}:L_{r_{\bot}}^{2}\rightarrow
L_{r_{\bot}}^{2},
\]
where $\tilde{G}$ is defined in (\ref{defn-G-tilde})$.$We will show that:
there exists $a_{0}>0$, such that when $c\in\lbrack0,a_{0}),$

(i) $\tilde{M}_{c}:L^{2}\rightarrow L^{2}$ is self-adjoint and bounded.

(ii) $\tilde{L}_{c}$ has one-dimensional cylindrical symmetric negative
eigenspace,
\[
\ker\tilde{M}_{c}\cap L_{r_{\bot}}^{2}=\left\{  \tilde{G}^{-1}\left(
\partial_{x_{1}}\rho_{c},\partial_{x_{1}}\theta_{c}\right)  \right\}
\]
and the rest of the spectrum are positive.

The conclusions in the Proposition follow from the above properties of the
operator $\tilde{M}_{c}$. Denote $\tilde{M}_{c}^{\infty}:=\tilde{G}\circ
M_{c}^{\infty}\circ\tilde{G},\ $where $M_{c}^{\infty}$ is defined in
(\ref{defn-M-infinity}). Then it is easy to see that $\tilde{M}_{c}^{\infty}$
is bounded and self-adjoint, and by the estimate
(\ref{estimate-quadratic-m-infty}), the essential spectrum of $\tilde{M}%
_{c}^{\infty}\subset\lbrack\delta_{0},\infty)$ for some $\delta_{0}>0$. We
show$\ \tilde{M}_{c}$ is compact perturbation of $\tilde{M}_{c}^{\infty}$.
Indeed, $\tilde{M}_{c}-\tilde{M}_{c}^{\infty}=\left(  \tilde{M}_{ij}\right)
$, where $M_{22}=0,$
\[
\tilde{M}_{11}=(-\Delta+1)^{-\frac{1}{2}}a_{1}\left(  x\right)  (-\Delta
+1)^{-\frac{1}{2}},
\]%
\[
\tilde{M}_{21}=-2(-\Delta)^{-\frac{1}{2}}\nabla\cdot\left(  \vec{a}_{2}\left(
x\right)  (-\Delta+1)^{-\frac{1}{2}}\right)  ,\ \tilde{M}_{12}=\tilde{M}%
_{21}^{\ast},
\]
with
\[
a_{1}\left(  x\right)  =\frac{1}{2}\frac{1}{\rho_{c}^{3}}|\nabla\rho_{c_{0}%
}|^{2}-\frac{\Delta\rho_{c_{0}}}{2\rho_{c_{0}}^{2}}-F^{\prime}(\rho_{c_{0}%
})-\frac{1}{\rho_{c_{0}}}\rightarrow0,
\]%
\[
\vec{a}_{2}\left(  x\right)  =\nabla\theta_{c_{0}}\rightarrow0,
\]
when $\left\vert x\right\vert \rightarrow\infty$. Thus by the local
compactness of $H^{1}\hookrightarrow L^{2}$, the operators $\tilde{M}_{11},$
$\tilde{M}_{21}$ and $\tilde{M}_{12}$ are compact. So by the perturbation
theory of self-adjoint operators, $\tilde{M}_{c}$ is self-adjoint and bounded,
with its essential spectrum in $[\delta_{0},\infty)$. By Lemma
\ref{lemma-spectra-M0} below, $\ker\tilde{M}_{0}\cap L_{r_{\bot}}^{2}=\left\{
\tilde{G}^{-1}\left(  \partial_{x_{1}}\rho_{0},0\right)  \right\}  $ and
$\tilde{M}_{0}$ has exactly one negative eigenvalue which is simple with
radially symmetric eigenfunction. Since the traveling wave solution $\left(
\rho_{c},\theta_{c}\right)  $ is $C^{1}$ for $c\in\lbrack-b_{0},b_{0})$, the
discrete spectrum of $\tilde{M}_{c}$ are continuous in $c$. Thus, there exists
$a_{0}>0$, such that for $c\in\lbrack0,a_{0})$, $\tilde{M}_{c}$ has
one-dimensional kernel in $L_{r_{\bot}}^{2}$ spanned by $\tilde{G}^{-1}\left(
\partial_{x_{1}}\rho_{c},\partial_{x_{1}}\theta_{c}\right)  \ $and exactly one
negative eigenvalue which is simple with a cylindrically symmetric eigenfunction.
\end{proof}

\begin{lemma}
\label{lemma-spectra-M0}Assume (H1)-(H3). Let $\phi_{0}=\sqrt{\rho_{0}}$ be a
stationary bubble of (\ref{eqn-GP-generalized}) obtained in \cite{de95} (via
constrained minimization) satisfying (\ref{cond-NDG-stationary}).Then
\[
\ker\tilde{M}_{0}\cap L_{r_{\bot}}^{2}=span\left\{  \tilde{G}^{-1}\left(
\partial_{x_{1}}\rho_{0},0\right)  \right\}
\]
and $\tilde{M}_{0}$ has exactly one negative eigenvalue which is simple and
with radially symmetric eigenfunction.
\end{lemma}

\begin{proof}
Since $\tilde{M}_{0}=\tilde{G}\circ M_{0}\circ\tilde{G}$, it is equivalent to
show that%
\[
\ker M_{0}\cap L_{r_{\bot}}^{2}=\left\{  \left(  \partial_{x_{1}}\rho
_{0},0\right)  \right\}
\]
and $M_{0}$ has only one negative eigenvalue which is simple with a radially
symmetric eigenfunction. Recall that
\[
M_{0}=%
\begin{pmatrix}
M_{1} & 0\\
0 & M_{2}%
\end{pmatrix}
,
\]
where
\[
M_{1}=A\left(  \frac{1}{2\sqrt{\rho_{0}}}\cdot\right)  \frac{1}{\sqrt{\rho
_{0}}},\ M_{2}=-2\nabla\cdot(\rho_{0}\nabla).
\]
Since $M_{2}>0$, by (\ref{cond-NDG-stationary}) the cylindrically symmetric
kernel of $M_{1}$ is spanned by $\partial_{x_{1}}\rho_{0}$. It remains to show
that the operator $A$ has only one negative eigenvalue which is simple with a
radially symmetric eigenfunction. This property was shown in Lemmas 2.1 and
2.2 of \cite{maris-4d-slow} for $n\geq3$. The proof for $n=2$ is almost the
same and we sketch it below. By Lemma 3.3 of \cite{de95} or Lemma 2.1 of
\cite{maris-4d-slow}, $A$ has at least one negative eigenvalue with radial
symmetric eigenfunction. It was shown in \cite{berestychi-et-83-2d}
\cite{alves-et-2012} that $\phi_{0}$ minimizes the functional
\[
T\left(  u\right)  =\frac{1}{2}\int_{\mathbf{R}^{2}}\left\vert \nabla
u\right\vert ^{2}\ dx
\]
subject to the constraint
\[
I\left(  u\right)  =\int_{\mathbf{R}^{2}}V\left(  |u|^{2}\right)  dx=0.
\]
By the arguments in the proof of Lemma \ref{le-5} or the proof of Lemma 2.2 in
\cite{maris-4d-slow}, it follows that $A$ has at most one-dimensional negative
eigenspace. Thus $A$ has exactly one-dimensional negative eigenspace.
\end{proof}

\begin{remark}
In Appendix 2, we prove the non-degeneracy condition
(\ref{cond-NDG-stationary}) for cubic-quintic nonlinearity and $n=2$. For
$n\geq3$, the condition (\ref{cond-NDG-stationary}) was proved in
\cite{maris-4d-slow} for nonlinearity satisfying some additional condition
\ ((H5) in P. 1209 of \cite{maris-4d-slow}). However, our computation
indicates that this additional condition appears to be not satisfied by the
cubic-quintic nonlinearity.
\end{remark}

Much as Proposition \ref{P-linear-insta-G} and Lemma \ref{lemma-dichotomy-x3},
we have the same linear instability criterion $\frac{\partial P(U_{c}%
)}{\partial c}<0$ and the subsequent linear exponential dichotomy.

\begin{proposition}
\label{P-Linear-Insta-nV} Let $U_{c}=\sqrt{\rho_{c}}e^{i\theta_{c}}$, $c
\in[c_{1}, c_{2}]$, be a $C^{1}$ (with respect to the wave speed $c$) family
of traveling waves of (\ref{eqn-GP-generalized}). For $c_{0} \in(c_{1},
c_{2})$, assume

\begin{enumerate}
\item $U_{c_{0}}(x) \ne0$, for all $x\in\mathbf{R}^{n}$;

\item non-degeneracy condition (\ref{cond-NDG-c0}) is satisfied;

\item $F \in C^{1}$ on $U_{c_{0}} (\mathbf{R}^{n})$;

\item $M_{c_{0}}$ satisfies the decomposition result stated in Proposition
\ref{prop-quadratic-madelung}, and

\item $\frac{\partial P(U_{c})}{\partial c}|_{c=c_{0}}<0$;
\end{enumerate}

then there exists $w_{u}\in X_{1}^{s}$ and $\lambda_{u}>0$, such that
$e^{\lambda_{u}t}w_{u}\left(  x\right)  $ is a solution of
(\ref{eqn-linearized-madelung}). Moreover, the linearized semigroup
$e^{tJM_{c_{0}}}$ also has an exponential dichotomy in the space $X_{3}^{s}$.
\end{proposition}

After these preparations, we show the linear instability of slow traveling
waves of (\ref{eqn-GP-generalized}) with cubic-quintic type nonlinear terms.

\begin{theorem}
\label{thm-instability-madelung}Assume (H1)-(H3) and
(\ref{cond-NDG-stationary}). For any $n\geq2$, $\exists$ $\varepsilon_{0}>0$,
such that for all $0\leq c<\varepsilon_{0}$, the traveling wave solutions
$U_{c}=\sqrt{\rho_{c}}e^{i\theta_{c}}\ $constructed in Theorem
\ref{thm-existence-slow} are linearly unstable in the following sense: there
exists an unstable solution $e^{\lambda_{u}t}w_{u}\left(  x\right)  $ with
\[
w_{u}=\left(  \rho_{u},\theta_{u}\right)  \in X_{3}^{s},\ \ \lambda_{u}>0,
\]
of the linearized equation (\ref{eqn-linearized-madelung}). Moreover, the
linearized semigroup $e^{tJM_{c}}$ also has an exponential dichotomy in the
space $X_{3}^{s}$.
\end{theorem}

\begin{proof}
Since the traveling wave branch $U_{c}\ $constructed in Theorem
\ref{thm-existence-slow} is $C^{1}$ for $c\in\left(  -b_{0},b_{0}\right)  $,
according to Proposition \ref{P-Linear-Insta-nV}, it is reduced to show that
$\frac{\partial P(U_{c})}{\partial c}|_{c=0}<0$. We note that
\[
\frac{1}{2}M_{c}\partial_{c}(\rho_{c},\theta_{c})=-P^{\prime}(\rho_{c}%
,\theta_{c})=-\frac{1}{2}J^{-1}\partial_{x_{1}}(\rho_{c},\theta_{c}),
\]
since $(\rho_{c},\theta_{c})$ satisfies $\left(  E^{\prime}+cP^{\prime
}\right)  (\rho_{c},\theta_{c})=0$ and $\left(  E^{\prime\prime}%
+cP^{\prime\prime}\right)  (\rho_{c},\theta_{c})=\frac{1}{2}M_{c}$. Here, we
use $^{\prime}$ to denote the functional derivative in $\left(  \rho
,\theta\right)  $. Thus%
\begin{align*}
\frac{\partial P(U_{c})}{\partial c}|_{c=0}  &  =\langle P^{\prime}(\rho
_{c},\theta_{c}),\partial_{c}(\rho_{c},\theta_{c})\rangle|_{c=0}\\
&  =\frac{1}{2}\left(  \partial_{c}(\rho_{c},\theta_{c}),\left(
0,\partial_{x_{1}}\rho_{0}\right)  \right)  =-\frac{1}{2}\left(  M_{2}%
^{-1}\partial_{x_{1}}\rho_{0},\partial_{x_{1}}\rho_{0}\right)  <0
\end{align*}
since $M_{2}=-2\nabla.(\rho_{0}\nabla)>0$.

By Proposition \ref{prop-quadratic-madelung}, we can show the exponential
dichotomy for the semigroup $e^{tJM_{c}}$ in the space $X_{3}^{s}$, as in
Lemma \ref{lemma-dichotomy-x3}. The proof is the same as in Lemmas
\ref{lemma-dichotomy-U} and \ref{lemma-dichotomy-x3}, thus we skip it.
\end{proof}

Due to the presence of derivative terms in the nonlinearity of the
hydrodynamic equation
\eqref{hrdrodynamic nls}, it is much easier to obtain the unstable manifolds
and thus the nonlinear instability of traveling waves by working with the
original form of the nonlinear Schr\"{o}dinger equation
\eqref{eqn-GP-generalized}, which is semilinear, based on the linear
instability obtained in the above theorems. We first state the following
proposition.

\begin{proposition}
\label{thm-UM-general} For any dimension $n\geq1$, let $k>\frac{n}{2}$ be an
integer and $\Omega_{\bot}\subset\mathbf{R}^{n-1}$ be a smooth domain.
Consider \eqref{eqn-GP-generalized} for $x\in\Omega=\mathbf{R}\times
\Omega_{\bot}$ subject to homogeneous Dirichlet, Neumann, or periodic boundary
condition on $\partial\Omega$. Suppose $F\in C^{k+l}$ and $U_{c}=u_{c}+iv_{c}$
be a traveling wave of \eqref{eqn-GP-generalized} on $\Omega$ such that
\[
u_{c}-1\in H^{k}(\Omega)\text{ and }v_{c}\in\dot{H}^{k}(\Omega).
\]
Assume the linearized flow $e^{tJL_{c}}$, where $L_{c}$ is defined in
\eqref{operator-Lc-g}, has an exponential dichotomy in $H^{k}(\Omega)\times
H^{k}(\Omega)$, i.e. there exist closed subspaces $E^{cs,u}\subset
H^{k}(\Omega)\times H^{k}(\Omega)$ and $\lambda_{u,cs},M\geq0$ such that
$\lambda_{u}>\lambda_{cs}$, $e^{tJL_{c}}E^{u,cs}=E^{u,cs}$, and
\[
\left\vert e^{tJL_{c}}|_{E^{u}}\right\vert \leq Me^{\lambda_{u}t},\;\forall
t\leq0\;\text{ and }\;\left\vert e^{tJL_{c}}|_{E^{cs}}\right\vert \leq
Me^{\lambda_{cs}t},\;\forall t\geq0.
\]
Then there exists a unique $C^{l}$ locally invariant unstable manifold
$W^{u}\subset U_{c}+H^{k}(\Omega)\times H^{k}(\Omega)$ of $U_{c}$ in the sense
as described in Theorem \ref{thm:invariant manifold}.
\end{proposition}

\begin{remark}
If $\Omega_{\bot}$ and $U_{c}$ are invariant under certain symmetries like
rotations or reflections, then one may work in the subspace of $H^{k}(\Omega)$
with the same symmetries and thus the unstable manifold would consist of
functions in $U_{c}+H^{k}(\Omega)\times H^{k}(\Omega)$ with the same symmetries.
\end{remark}

The proposition can be obtained from the general theorems in
\cite{bates-jones88, CL88} simply based on the observation that, under the
above assumptions, the dynamic equation of $z(t,x)=U(t,x-ct\vec{\mathbf{e}%
}_{1})-U_{c}(x-ct\vec{\mathbf{e}}_{1})$, where $U(t,x)$ is a solution of
\eqref{eqn-GP-generalized}, has the linear part $JL_{c}z$ and its nonlinear
part defines a $C^{l}$ transformation from $H^{k}(\Omega)\times H^{k}(\Omega)$
to itself. As a corollary, we have

\begin{corollary}
\label{thm-mfld-madelung}Let $n\geq2$ and $U_{c}=\sqrt{\rho_{c}}e^{i\theta
_{c}}$ be a traveling wave of \eqref{eqn-GP-generalized}, radially symmetric
in $x_{\bot}$ and linearly unstable with the linear exponential dichotomy for
$(\rho,\theta)\in H_{r_{\bot}}^{k}(\mathbf{R}^{n})\times\dot{H}_{r_{\bot}}%
^{k}(\mathbf{R}^{n})$ $\left(  k\geq2\right)  $, including those proved in
Theorem \ref{thm-instability-madelung}. Assume $F\in C^{k+l}$, then there
exists a unique $C^{l}$ local unstable (stable) manifolds $W^{u}\left(
W^{s}\right)  $ of
$U_{c}$ in $U_{c}+H_{r_{\bot}}^{k}(\mathbf{R}^{n})\times H_{r_{\bot}}%
^{k}(\mathbf{R}^{n})$.

\end{corollary}

\begin{proof}
In order to prove the corollary, we only need to establish the linear
exponential dichotomy of $e^{tJL_{c}}$ in $H_{r_{\bot}}^{k}\times H_{r_{\bot}%
}^{k}$. Since $JL_{c}$ and $JM_{c}$, where $M_{c}$ is defined in
(\ref{defn-M}), are conjugate through $T_{c}$ defined in (\ref{definition T_c}%
) and $T_{c}$ is an isomorphism on $H_{r_{\bot}}^{k}\times H_{r_{\bot}}^{k}$,
we only need to obtain the exponential dichotomy of $e^{tJM_{c}}$ on
$H_{r_{\bot}}^{k}\times H_{r_{\bot}}^{k}$. Based on Theorem
\ref{thm-instability-madelung}, it is straightforward to repeat the arguments
to derive the exponential dichotomy of $e^{tJM_{c}}\ $in $H_{r_{\bot}}%
^{k}\times\dot{H}_{r_{\bot}}^{k}$. The form of $JM_{c}$ implies that its
unstable and stable eigenfunctions in $H_{r_{\bot}}^{k}\times\dot{H}_{r_{\bot
}}^{k}\ $actually belong to $H_{r_{\bot}}^{k}\times H_{r_{\bot}}^{k}$, that
is, $E^{u},E^{s}\subset H_{r_{\bot}}^{k}\times H_{r_{\bot}}^{k}$. Indeed, for
an unstable eigenvalue $\lambda_{u}>0$, the eigenfunction $\left(  \rho
_{u},\theta_{u}\right)  \in H_{r_{\bot}}^{k}\times\dot{H}_{r_{\bot}}^{k}$
satisfies
\[
\theta_{u}=-\frac{1}{\lambda_{u}}\left(  M_{11}\rho_{u}+M_{12}\theta
_{u}\right)  \in L^{2}\text{,}%
\]
and the same is true for the stable eigenfunction. Here, $M_{11}$ and $M_{12}$
are defined in (\ref{defn-M11})-(\ref{defn-M12}) and we use the observation
that $M_{12}\theta$ contains only $\partial\theta(t)$, instead of $\theta(t)$
itself. It is easy to see that $\tilde{E}^{cs}=E^{cs}\cap H_{r_{\bot}}%
^{k}\times H_{r_{\bot}}^{k}$ is a closed subspace of $H_{r_{\bot}}^{k}\times
H_{r_{\bot}}^{k}$, invariant under $e^{tJM_{c}}$, and $H_{r_{\bot}}^{k}\times
H_{r_{\bot}}^{k}=E^{u}\oplus\tilde{E}^{cs}$. To complete the proof, we only
need to obtain the following growth estimate of $|\theta(t)|_{L^{2}}$ where
$z(t)=\left(  \rho\left(  t\right)  ,\theta\left(  t\right)  \right)
=e^{tJM_{c}}z^{0}$ for $z^{0}\in\tilde{E}^{cs}$. When $z^{0}=\left(
\rho\left(  0\right)  ,\theta\left(  0\right)  \right)  \in\tilde{E}^{cs}$, we
have that for any $t\geq0,\ $
\begin{align*}
\left\vert \theta(t)\right\vert _{L^{2}}  &  \leq\left\vert \theta
(0)\right\vert _{L^{2}}+\int_{0}^{t}\left\vert \theta_{t}(s)\right\vert ds\\
&  \leq\left\vert \theta(0)\right\vert _{L^{2}}+\int_{0}^{t}\left(  \left\vert
M_{11}\rho\left(  s\right)  \right\vert _{L^{2}}+\left\vert M_{12}%
\theta\left(  s\right)  \right\vert _{L^{2}}\right)  ds\\
&  \leq\left\vert \theta(0)\right\vert _{L^{2}}+C\int_{0}^{t}\left(
\left\vert \rho\left(  s\right)  \right\vert _{H^{2}}+\left\vert \theta\left(
s\right)  \right\vert _{\dot{H}^{1}}\right)  ds\\
&  \leq\left\vert \theta(0)\right\vert _{L^{2}}+C\int_{0}^{t}Me^{\lambda
_{cs}s}|z^{0}|_{H^{k}\times\dot{H}^{k}}ds\\
&  \leq M^{\prime}(1+t)e^{\lambda_{cs}t}|z^{0}|_{H^{k}\times H^{k}},
\end{align*}
In the above, we use the special form of $M_{12}\ $again and the exponential
dichotomy of the linear equation $z_{t}(t)=JM_{c}z(t)$ in $H_{r_{\bot}}%
^{k}\times\dot{H}_{r_{\bot}}^{k}$. Therefore the exponential dichotomy of
$e^{tJM_{c}}\ $holds in $H_{r_{\bot}}^{k}\times H_{r_{\bot}}^{k}$ with
$\lambda_{u}$ and any $\tilde{\lambda}_{cs}\in(\lambda_{cs},\lambda_{u})$.
\end{proof}

As a corollary, we get nonlinear orbital instability for slow traveling waves
in dimension $n\geq2$.

\begin{remark}
\label{rmk-instability-3-5}The instability of stationary bubbles was proved in
\cite{de95}. It is possible to show the instability of slow traveling waves by
certain perturbation argument. However, the instability proof of Theorem
\ref{thm-instability-madelung} contains more information than what can be
obtained from a perturbation theory. First, it yields the exponential
dichotomy of the semigroup which is essential for constructing invariant
manifolds. The unstable manifold theorem automatically implies the optimal
orbital instability result, that is, the instability is measured in a weak
norm with initial deviation in a strong norm and the growth is exponentially
fast. By contrast, the nonlinear instability proof in \cite{de95} used an
abstract theorem of \cite{henry-et-83} (see \cite{shatah-strauss00} for a
similar theorem). It does not require the exponential dichotomy or the precise
growth estimate of the semigroup, but the instability was proved in a strong
norm $H^{k}$ $\left(  k>\frac{n}{2}\right)  $ and no estimate of the growth
time scale was given. Second, the proof of Theorem
\ref{thm-instability-madelung} actually gives a instability criterion
$\frac{\partial P(U_{c})}{\partial c}<0$ under the non-degeneracy condition
(\ref{assumption-NDG-s}). In particular, the instability of stationary
solution persists until the first travel speed $c\ $at which either
$\frac{\partial P(U_{c})}{\partial c}=0$ or the condition
(\ref{assumption-NDG-s}) fails.
\end{remark}

\section{Extensions and future problems}

In this section, we discuss the extensions of the results in previous
sections. We also mention some remaining issues on stability of traveling
waves of (\ref{eqn-GP-generalized}).

In the one-dimensional case, when $0<c<\sqrt{2}$, the traveling waves $U_{c}$
of (\ref{eqn-GP-generalized}) is nonvanishing. In this case, a sharp stability
criterion was obtained in \cite{lin-bubble-1d} by using the hydrodynamic
formulation and the theory of \cite{gss87}. The traveling waves are stable if
and only if $\frac{d}{dc}P\left(  U_{c}\right)  >0$, where
\[
P\left(  u\right)  =-\int_{\mathbf{R}}\operatorname{Im}\left(  \bar
{u}u^{\prime}\right)  \left(  1-\frac{1}{\left\vert u\right\vert ^{2}}\right)
dx\text{. }%
\]
However, in \cite{lin-bubble-1d} nonlinear instability was only proved in the
energy space and without any estimate of the growth time scale, as in
\cite{gss87} \cite{bona-et-87}, where the linear instability problem was
bypassed and the nonlinear instability was proved by a contradiction argument.
Recently, the nonlinear orbital instability with exponential growth was proved
in \cite{chiron-1d} by studying the linearized problem. In Remark
\ref{remark-gap-chiron-1d} below, we comment on some possible gap in the proof
of \cite{chiron-1d}. By using the methods in Sections 3 and 5.3, we can prove
the existence of unstable (stable) manifolds near $U_{c}\ $and thus obtain the
optimal nonlinear instability result when $\frac{d}{dc}P\left(  U_{c}\right)
<0.$ In fact, for $1$D traveling waves, the spectral property of the quadratic
form $\left\langle M_{c}\cdot,\cdot\right\rangle \ $as in Proposition
\ref{prop-quadratic-madelung} was essentially proved in \cite{lin-bubble-1d}.
By using this spectral property, the linear instability when $\frac{d}%
{dc}P\left(  U_{c}\right)  <0\ $and the exponential dichotomy of the linear
semigroup $e^{tJM_{c}}$ can be proved by the same approach as in Section 3.
Then the unstable (stable) manifolds can be constructed via Theorem
\ref{thm-UM-general} as in corollary \ref{thm-mfld-madelung}.

Another extension is to prove the transversal instability of any $1$D
traveling waves of (\ref{eqn-GP-generalized}) by the approach of Section 4.
Let
\begin{equation}
L_{c}:=\left(
\begin{array}
[c]{cc}%
-\partial_{x}^{2}-F\left(  \left\vert U_{c}\right\vert ^{2}\right)
-F^{\prime}\left(  \left\vert U_{c}\right\vert ^{2}\right)  2u_{c}^{2} &
-c\partial_{x}-2F^{\prime}\left(  \left\vert U_{c}\right\vert ^{2}\right)
u_{c}v_{c}\\
c\partial_{x}-2F^{\prime}\left(  \left\vert U_{c}\right\vert ^{2}\right)
u_{c}v_{c} & -\partial_{x}^{2}-F\left(  \left\vert U_{c}\right\vert
^{2}\right)  -F^{\prime}\left(  \left\vert U_{c}\right\vert ^{2}\right)
2v_{c}^{2}%
\end{array}
\right)  . \label{operator-L_c-1d}%
\end{equation}
When $0<c<\sqrt{2}$, the operator $L_{c}$ has exactly one negative eigenvalue
$\lambda_{0}<0$ and the second eigenvalue is zero. Since by the analogue of
the formula (\ref{relation-2nd-variation-madelung}) in $1$D, the number of
negative modes of the quadratic form with $L_{c}$ equals that of the quadratic
form $\left\langle M_{c}\cdot,\cdot\right\rangle .$ The latter number was
shown to be the one in \cite{lin-bubble-1d}. Thus, by the same proof of
Theorem \ref{thm-linear-transveral-instability}, when $k\in\left(
0,\sqrt{-\lambda_{0}}\right)  $ the $1$D traveling wave is transversal
unstable with the transversal period $\frac{2\pi}{k}$ and is transversely
stable when $k\geq\sqrt{-\lambda_{0}}$. Similar to the proof of Lemma
\ref{lemma-dichotomy-periodic} and Theorem \ref{thm-transversal-mfld-2d}, we
can construct unstable (stable) manifolds from the above transversal
instability. For (GP), when $c=0$, the stationary solution $U_{0}=\tanh\left(
\frac{x}{\sqrt{2}}\right)  $ vanishes at $x=0$, the operator
\[
L_{0}=\left(
\begin{array}
[c]{cc}%
-\partial_{x}^{2}-1+3U_{0}^{2} & 0\\
0 & -\partial_{x}^{2}-1+U_{0}^{2}%
\end{array}
\right)
\]
can be verified to have exactly one negative eigenvalue. So the above
discussions on transversal instability are also valid for $U_{0}$. We should
note that in \cite{rousett-et-transversal}, the linear transversal instability
of $1$D traveling waves of (GP) was shown for $k$ near $\sqrt{-\lambda_{0}}$
by a different method. Our results are novel in the following two aspects: 1)
locating the sharp interval for unstable transversal wave numbers; 2)
constructing the unstable (stable) manifolds.

\begin{remark}
\label{remark-gap-chiron-1d}In \cite{chiron-1d}, the proof of nonlinear
orbital instability of 1D traveling waves (Theorem 4 and Corollary 3) consists
of several steps. For linear instability when $\frac{d}{dc}P\left(
U_{c}\right)  <0$, the author cited the result in \cite{benzoni-instability}.
Then, an abstract result (Theorem B.3 in \cite{chiron-1d}) is used to get the
growth estimate (and actually exponential dichotomy) of the semigroup
$e^{tJL_{c}}$ in $H^{1}\left(  \mathbf{R}\right)  \times H^{1}\left(
\mathbf{R}\right)  $, where the operator $L_{c}$ is defined in
(\ref{operator-L_c-1d}). By using this semigroup estimate, the nonlinear
orbital instability follows since the equation (\ref{eqn-GP-generalized}) is
semilinear in $H^{1}\times H^{1}$. However, the operator $L_{c}$ does not
satisfy Assumption (A) for Theorem B.3, which requires that $\sigma
_{ess}\left(  L_{c}\right)  =[\delta_{0},+\infty)$ for some $\delta_{0}>0$.
Since by (\ref{relation-2nd-variation-madelung}) we have $\sigma_{ess}\left(
L_{c}\right)  =\sigma_{ess}\left(  M_{c}\right)  =[0,+\infty)$, where $M_{c}$
is the 1D version of the operator defined in (\ref{defn-M}). Such a lack of
spectral gap is exactly one of the main difficulty for studying stability of
traveling waves of (\ref{eqn-GP-generalized}) with nonvanishing condition at
infinity. To overcome this difficulty in 1D, first we get the exponential
dichotomy of $e^{tJM_{c}}$ on $H^{k}\times\dot{H}^{k}$ $\left(  k\geq2\right)
\ $where $M_{c}$ has a spectral gap. Then we lift this exponential dichotomy
of $e^{tJM_{c}}$ to $H^{k}\times H^{k}$ by noting that the unstable
eigenfunction lies in this space. The exponential dichotomy of $e^{tJL_{c}}$
on $H^{k}\times H^{k}$ then follows.
\end{remark}

Theorems \ref{thm:stability}, \ref{thm:invariant manifold} and Corollary
\ref{C:critical} in Sections 2 and 3 give a completed theory for the orbital
stability (instability) of $3$D traveling waves. We briefly comment on the
extensions to higher dimensions $n\geq4$. Assume the nonlinear term $F\left(
u\right)  $ in (\ref{eqn-GP-generalized}) satisfies:

(F1) $F\in C^{1}(\mathbf{R})$, $C^{2}$ in a neighborhood of $1$, $F(1)=0$ and
$F^{\prime}(1)=-1$.

(F2) There exists $0\leq p_{1}\leq1\leq p_{0}<\frac{2}{n-2}$ such that
$|F^{\prime}(s)|\leq C(1+s^{p_{1}-1}+s^{p_{0}-1})\ $for all $s\geq0$.\newline

Under the assumption (F1)-(F2), as in 3D case, the traveling waves have been
constructed in \cite{Maris-annal} by minimizing the energy subject to Pohozaev
type constraint, for $n\geq4$. We can prove the spectral property for the
quadratic form $\left\langle L_{c}\cdot,\cdot\right\rangle \ $as in
Proposition \ref{prop-quadratic}, in the space $X_{1}=H^{1}\left(
\mathbf{R}^{n}\right)  \times\dot{H}^{1}\left(  \mathbf{R}^{n}\right)  $. This
allows us to prove a linear instability criterion that $\frac{\partial
P(U_{c})}{\partial c}|_{c=c_{0}}<0$ where
\[
P\left(  u\right)  =\frac{1}{2}\int_{\mathbf{R}^{n}}\langle i\partial_{x_{1}%
}u,u-1\rangle\ dx=-\int_{\mathbf{R}^{n}}\left(  u_{1}-1\right)  \partial
_{x_{1}}u_{2}dx.
\]
To pass to nonlinear results, first we note that when $n\geq4$, it was shown
in \cite{maris-4d-slow} that the energy-momentum functional $E_{c}=E+cP$ is
$C^{2}$ on the space $1+X_{1}$. Thus, by the proof of Theorem
\ref{thm:stability}, we can prove that the traveling waves with $\frac
{\partial P(U_{c})}{\partial c}|_{c=c_{0}}>0$ is orbitally stable in the
distance $\left\Vert u-U_{c}\right\Vert _{X_{1}}.\ $Moreover, when $n=4$, it
can be shown that the energy space
\begin{align*}
X_{0}  &  =\left\{  u\ |\nabla u\in L^{2}\left(  \mathbf{R}^{n}\right)
,\ V\left(  u\right)  \in L^{1}\left(  \mathbf{R}^{n}\right)  \right\} \\
&  =\left\{  u\ |\nabla u\in L^{2}\left(  \mathbf{R}^{n}\right)
,\ 1-\left\vert u\right\vert ^{2}\in L^{1}\left(  \mathbf{R}^{n}\right)
\right\}
\end{align*}
exactly consists of functions of the form $\left\{  c\left(  1+w\right)
\ |\ c\in\mathbb{S}^{1},\ w\in X_{1}\right\}  $. So when $n=4$, $\frac
{\partial P(U_{c})}{\partial c}|_{c=c_{0}}>0$ is sharp for orbital stability
in the energy space. When $n>4$, the energy space $X_{0}$ might be strictly
larger than the set $\left\{  c\left(  1+X_{1}\right)  \right\}  $. To show
the orbital stability in $X_{0}$ for $n>4$, we need to find a coordinate
mapping $u=g\left(  w\right)  ,w\in X_{1}$ for $u\in X_{0}$, as for the $3$D
case. To construct unstable (stable) manifolds under the instability criterion
$\frac{\partial P(U_{c})}{\partial c}|_{c=c_{0}}<0$, we first note that the
exponential dichotomy is still true in $X_{1}$ and then in $X_{k}=H^{k}%
\times\dot{H}^{k}$ for any integer $k>1$, by the proof of Lemmas
\ref{lemma-dichotomy-U} and \ref{lemma-dichotomy-x3}. Then we write
$U=U_{c}+w$ $\left(  w\in X_{k}\right)  \ $in (\ref{eqn-GP-g-TF}). It can be
shown that, assuming $F\in C^{k+2}$, the nonlinear term $F\left(  \left\vert
U\right\vert ^{2}\right)  U\in C^{2}\left(  X_{k},X_{k}\right)  $ for $k$
large. Therefore, the unstable (stable) manifolds can be constructed in the
space $U_{c}+X_{k}$, which is contained in the energy space by
\cite{maris-4d-slow} as mentioned earlier.

We notice that equation (GP) for $n=4$ is just the borderline case and does
not satisfy (F2) for nonlinear stability (keep in mind (F1--2) are not needed
for unstable manifolds). Exactly as in Corollary \ref{C:critical}, we may
instead assume \newline\newline(F2') There exists $C, \alpha_{0}, s_{0}>0$,
and $0< p_{1} \le1 \le p_{0} \le\frac2{n-2}$, such that $|F^{\prime}(s)|\leq
C(1 + s^{p_{1}-1} + s^{p_{0}-1})$ for all $s\geq0$ and $F(s) \le
-Cs^{\alpha_{0}}$ for all $s>s_{0}$.\newline\newline Following the same
argument for Corollary \ref{C:critical}, we obtain the nonlinear instability
of traveling waves obtained in \cite{Maris-annal}.

Now we discuss the $2$D case. When the traveling wave $U_{c}$ has no vortices
($U_{c}\neq0$), we can use the Madelung transform to derive the instability
criterion $\frac{\partial P(U_{c})}{\partial c}|_{c=c_{0}}<0$ and construct
stable (unstable) manifolds. See Theorems \ref{thm-instability-madelung} and
\ref{thm-mfld-madelung}. However, things get more tricky for traveling waves
with vortices. The Madelung transform is not applicable. Also, it is improper
to use the base space $X_{1}$ to study the linearized problem
(\ref{eqn-linearized-u}) for two reasons. First, the Hardy's inequality
(\ref{hardy}) is not valid for $n=2$. So we can not even define the quadratic
form $\left\langle L_{c}\cdot,\cdot\right\rangle $ on $X_{1}$. Secondly, due
to the oscillations at infinity of functions in $X_{0}$, we don't have a
manifold structure of $X_{0}$ with the base $X_{1}$. In \cite{Gerard1-energy},
a manifold structure of $X_{0}$ is given with the base space
\[
X^{\prime}=H_{\mathbf{R}}^{1}\left(  \mathbf{R}^{2}\right)  \times\left(
X_{\mathbf{R}}^{1}\left(  \mathbf{R}^{2}\right)  +H_{\mathbf{R}}^{1}\left(
\mathbf{R}^{2}\right)  \right)  ,
\]
where $X_{\mathbf{R}}^{1}=\left\{  u\in L^{\infty}|\ \nabla u\in
L^{2}\right\}  $. However, it is unclear how to use $X^{\prime}$ in place of
$X_{1}$ in our approach.

In an ongoing work, we construct center manifolds near the unstable $3$D
traveling waves of (GP) as proved in Proposition \ref{prop-linear-insta}. The
linear exponential trichotomy of the semigroup $e^{tJL_{c}}$ has been
established for the space $X_{1}\ $in Lemma \ref{lemma-trichotomy}. This
trichotomy can also be lifted to the space $X_{3}$, similar to Lemma
\ref{lemma-dichotomy-x3}. Then we can construct center manifold in the orbital
neighborhood of $w_{c}$ in $X_{3}$. However, it is more desirable to construct
center manifold in the energy space $X_{1}$. We note that in Lemma
\ref{lemma-trichotomy} it is shown that the second variation of the
energy-momentum is positive definite when restricted on the center space and
modulo the generalized kernel. So the construction of center manifold in
$X_{1}\ $would imply the orbital stability restricted there and also the local
uniqueness of the center manifold. Together with the unstable (stable)
manifolds in Theorem \ref{thm:invariant manifold}, these will give a foliation
of the local dynamics near the unstable $3$D traveling waves.

\section{Appendix}

\subsection*{Appendix 1}

In this appendix, we prove the $C^{2}$ smoothness of the nonlinear term
$\Psi(w)$ (defined in (\ref{U-7}) and (\ref{U-8})) on $X_{3}$. Indeed, we will
show that $\Psi\in C^{2}\left(  X_{3},H^{3}\left(  \mathbf{R}^{3}%
;\mathbf{C}\right)  \right)  .$

Let $U_{c}=u_{c}+iv_{c}$ be a finite energy traveling wave solution of
equation (\ref{eqn-GP-generalized}), that is, $|U_{c}|^{2}-1,\ \nabla U_{c}\in
L^{2}(\mathbf{R}^{3})$. Let $U_{c}=\psi(w_{c})$, then $w_{c}\in X_{1}%
=H^{1}\times\dot{H}^{1}$. Moreover, by the proof of Lemma 5.5 of
\cite{Maris-annal}, $u_{c}-1,v_{c}\in\dot{H}^{3}(\mathbf{R}^{3})$. So it
follows from the definition of the coordinate mapping $\psi$ that
\[
w_{1c}\in H^{3}(\mathbf{R}^{3}),\ w_{2c}\in\dot{H}^{3}(\mathbf{R}^{3}).
\]

\begin{lemma}
\label{le-7} Assume that $F\in C^{5}(\mathbf{R})$ and $F\left(  1\right)  =0$.
Then $\Psi\in C^{2}(X_{3},H^{3})$.
\end{lemma}

\begin{remark}
For the construction of unstable (stable) manifolds, we work on a small
$X_{3}$ neighborhood of a linearly unstable traveling wave $U_{c}$. By Sobolev
embedding, $X_{3}\hookrightarrow L^{\infty}\left(  \mathbf{R}^{3}\right)  \,$,
so we only need to assume the smoothness of $F$ in a finite interval $\left[
\min\left\vert U_{c}\right\vert ^{2} -\varepsilon_{0},\max\left\vert
U_{c}\right\vert ^{2} +\varepsilon_{0}\right]  $ for some $\varepsilon_{0}>0$.
In particular, for traveling waves with no vortices $\left(  U_{c}%
\neq0\right)  $, we do not need to assume the smoothness of $F\left(
s\right)  $ near $s=0.~$
\end{remark}

\begin{proof}
In the sequel, let $C(\Vert w\Vert_{X_{3}})$ be a constant depending
(increasingly) on $\Vert w\Vert_{X_{3}}$.\ In the proof, we will use the
following basic facts:

(i) Let $n\in\mathbb{N}$, $F\in C^{3+n}(\mathbf{R})$ with $F(0)=0$, and $g\in
H^{3}(\mathbf{R}^{3})$. Then
\begin{equation}
F(g)\in H^{3}(\mathbf{R}^{3})\text{ and }F\in C^{n}\left(  H^{3},H^{3}\right)
. \label{moser-composition}%
\end{equation}

(ii)
\begin{equation}
\Vert fg\Vert_{H^{3}}\leq C\Vert f\Vert_{H^{3}}\Vert g\Vert_{\dot{H}^{3}%
},\ \forall\ f\in H^{3}(\mathbf{R}^{3}),g\in\dot{H}^{3}(\mathbf{R}^{3}).
\label{ineq-5}%
\end{equation}

(iii)
\begin{equation}
\Vert fg\Vert_{\dot{H}^{3}}\leq C\Vert f\Vert_{\dot{H}^{3}}\Vert g\Vert
_{\dot{H}^{3}},\ \forall\ f,g\in\dot{H}^{3}(\mathbf{R}^{3}), \label{ineq-4}%
\end{equation}

(iv) Let $\chi\in C_{0}^{\infty}(\mathbf{R}^{3},[0,1]),\ $then $\forall
\ f,g\in\dot{H}^{3}(\mathbf{R}^{3})$,
\begin{equation}
\Vert\Delta\chi(D)(fg)\Vert_{H^{3}}\leq C\Vert(|\xi|+|\xi|^{3})\widehat{fg}%
\Vert_{L^{2}}\leq C\Vert f\Vert_{\dot{H}^{3}}\Vert g\Vert_{\dot{H}^{3}}
\label{ineq-6}%
\end{equation}
and
\begin{equation}
\Vert D^{\alpha}\chi(D)(fg)\Vert_{\dot{H}^{3}}\leq C\Vert f\Vert_{\dot{H}^{3}%
}\Vert g\Vert_{\dot{H}^{3}},\ \forall\ 0\leq|\alpha|<+\infty. \label{ineq-7}%
\end{equation}
Here, (i) is by Moser's composition inequality, (ii)-(iv) can be shown by
Sobolev embedding and Fourier transforms.

\textbf{Step 1. }Show $\operatorname{Re}\Psi(w)\in C^{2}(X_{3},H^{3}).$

Let $\tilde{F}(s)=F(s+1)$, then $\tilde{F}\in C^{5}(\mathbf{R})$ and
$\tilde{F}(0)=0$. Denote $U=\psi\left(  w+w_{c}\right)  ,$ we write
\begin{align*}
Re\Psi(w)  &  =\Delta\chi(D)(w_{2c}w_{2}+\frac{w_{2}^{2}}{2})+[F(|U_{c}%
|^{2})-F(|U|^{2})]u_{c}\\
&  \ \ \ -F(|U|^{2})[w_{1}-\chi(D)(w_{2c}w_{2}+\frac{w_{2}^{2}}{2})]\\
&  =F(|U_{c}|^{2})u_{c}+\Delta\chi(D)(w_{2c}w_{2}+\frac{w_{2}^{2}}{2})\\
&  \ \ \ \ -F(|U|^{2})\left[  u_{c}+w_{1}-\chi(D)(w_{2c}w_{2}+\frac{w_{2}^{2}%
}{2})\right]  .\\
&  =\tilde{F}(|U_{c}|^{2}-1)u_{c}+\Psi_{1}\left(  w\right)  +\tilde{F}\left(
\Psi_{2}\left(  w+w_{c}\right)  \right)  \Psi_{3}\left(  w\right)  ,
\end{align*}
where%
\[
\Psi_{1}\left(  w\right)  =\Delta\chi(D)(w_{2c}w_{2}+\frac{w_{2}^{2}}{2}),
\]%
\begin{align*}
\Psi_{2}\left(  w\right)   &  =\left\vert \psi\left(  w\right)  \right\vert
^{2}-1=\left(  1+w_{1}-\chi(D)(\frac{w_{2}^{2}}{2})\right)  ^{2}+w_{2}^{2}-1\\
&  =\left(  w_{1}-\chi(D)(\frac{w_{2}^{2}}{2})\right)  ^{2}+\left(
1-\chi(D)\right)  w_{2}^{2}+2w_{1},
\end{align*}
and
\[
\Psi_{3}\left(  w\right)  =u_{c}+w_{1}-\chi(D)(w_{2c}w_{2}+\frac{w_{2}^{2}}%
{2}).
\]
Since $\tilde{F}\in C^{2}\left(  H^{3},H^{3}\right)  ,$ by (\ref{ineq-4}) and
(\ref{ineq-5}) it suffices to show that $\Psi_{1},\Psi_{2}\in C^{2}\left(
X^{3},H^{3}\right)  $ and $\Psi_{3}\in C^{2}\left(  X^{3},\dot{H}^{3}\right)
.$ It follows from (\ref{ineq-6}) and (\ref{ineq-7}) that $\Psi_{1}\in
C^{\infty}\left(  X^{3},H^{3}\right)  $ and $\Psi_{3} -1\in C^{\infty}\left(
X^{3},\dot{H}^{3}\right)  $. Let
\[
\Psi_{2}\left(  w\right)  =\left(  \Psi_{4}\left(  w\right)  \right)
^{2}+\Psi_{5}\left(  w\right)  ,
\]
where
\[
\Psi_{4}\left(  w\right)  =w_{1}-\chi(D)(\frac{w_{2}^{2}}{2}),\ \Psi
_{5}\left(  w\right)  =\left(  1-\chi(D)\right)  w_{2}^{2}+2w_{1}.
\]
By (\ref{ineq-2}), (\ref{ineq-5}) and (\ref{ineq-7}), for any$\ f,g\in\dot
{H}^{3}(\mathbf{R}^{3})$,
\[
\left\Vert \left(  1-\chi(D)\right)  \left(  fg\right)  \right\Vert _{H^{3}%
}\leq C\Vert f\Vert_{\dot{H}^{3}}\Vert g\Vert_{\dot{H}^{3}}.
\]
This implies that $\Psi_{5}\in C^{\infty}\left(  X^{3},H^{3}\right)  $. By
(\ref{ineq-2}),
\[
\left\Vert \chi(D)\left(  fg\right)  \right\Vert _{L^{4}}\leq C\Vert
f\Vert_{\dot{H}^{1}}\Vert g\Vert_{\dot{H}^{1}}.
\]
Combining with (\ref{ineq-7}), this implies that
\[
\Psi_{4}\left(  w\right)  \in C^{\infty}\left(  X^{3},L^{4}\cap\dot{H}%
^{3}\right)
\]
and thus $\left(  \Psi_{4}\right)  ^{2}\in C^{\infty}\left(  X^{3}%
,H^{3}\right)  \ $by (\ref{ineq-4}). This finishes the proof for
$\operatorname{Re}\Psi(w).$

\textbf{Step 2. }Show $\operatorname{Im}\Psi(w)\in C^{2}(X_{3},H^{3})$.

We write
\begin{align}
\operatorname{Im}\Psi(w)  &  =\chi(D)((w_{2c}+w_{2})\partial_{t}%
w_{2})-c\partial_{x_{1}}\chi(D)(w_{2c}w_{2}+\frac{w_{2}^{2}}{2}%
)\label{eqn-imiginary}\\
&  \ \ \ \ \ +[F(|U_{c}|^{2})-F(|U|^{2})]v_{c}-F(|U|^{2})w_{2}\nonumber\\
&  =F(|U_{c}|^{2})v_{c}+\Psi_{6}\left(  w\right)  +\Psi_{7}\left(  w\right)
-\tilde{F}\left(  \Psi_{2}\left(  w+w_{c}\right)  \right)  \left(  v_{c}%
+w_{2}\right)  ,\nonumber
\end{align}
where%
\[
\Psi_{6}\left(  w\right)  =\chi(D)((w_{2c}+w_{2})\Psi_{8}\left(  w\right)  ),
\]%
\[
\Psi_{8}\left(  w\right)  =\partial_{t}w_{2}=\Delta w_{1}+c\partial_{x_{1}%
}w_{2}-\operatorname{Re}\Psi(w),
\]
and%
\[
\Psi_{7}\left(  w\right)  =-c\partial_{x_{1}}\chi(D)(w_{2c}w_{2}+\frac
{w_{2}^{2}}{2}).
\]
By the proof in Step 1, the last term in (\ref{eqn-imiginary}) is in
$C^{2}(X_{3},H^{3})$ and $\Psi_{8}\in C^{2}\left(  X_{3},H^{1}\right)  .$
Since for any $f\in\dot{H}^{1},g\in H^{1},$%
\[
\left\Vert \chi(D)(fg)\right\Vert _{H^{3}}\leq C\left\Vert fg\right\Vert
_{L^{2}}\leq C\left\Vert f\right\Vert _{\dot{H}^{1}}\left\Vert g\right\Vert
_{H^{1}},
\]
so $\Psi_{6}\in C^{2}\left(  X_{3},H^{3}\right)  $. For any $f,g\in\dot{H}%
^{1}$, by (\ref{ineq-1}) we have
\begin{align*}
&  \left\Vert \partial_{x_{1}}\chi(D)\left(  fg\right)  \right\Vert _{L^{2}}\\
&  =\left\Vert \chi(D)\left(  f\partial_{x_{1}}g+g\partial_{x_{1}}f\right)
\right\Vert _{L^{2}}\leq C\left\Vert f\partial_{x_{1}}g+g\partial_{x_{1}%
}f\right\Vert _{L^{\frac{3}{2}}}\\
&  \leq C\left(  \left\Vert f\right\Vert _{L^{6}}\left\Vert \partial_{x_{1}%
}g\right\Vert _{L^{2}}+\left\Vert g\right\Vert _{L^{6}}\left\Vert
\partial_{x_{1}}f\right\Vert _{L^{2}}\right)  \leq C\left\Vert f\right\Vert
_{\dot{H}^{1}}\left\Vert g\right\Vert _{\dot{H}^{1}}.
\end{align*}
Combined with (\ref{ineq-7}), above implies that
\[
\left\Vert \partial_{x_{1}}\chi(D)\left(  fg\right)  \right\Vert _{H^{3}}\leq
C\Vert f\Vert_{\dot{H}^{3}}\Vert g\Vert_{\dot{H}^{3}},\
\]
for any$\ f,g\in\dot{H}^{3}(\mathbf{R}^{3}).$ Thus $\Psi_{7}\in C^{\infty
}\left(  X_{3},H^{3}\right)  $. This finishes the proof for $\operatorname{Im}%
\Psi(w)$.
\end{proof}

\subsection*{Appendix 2}

In this Appendix, we show the non-degeneracy condition
(\ref{cond-NDG-stationary}) of stationary bubbles $\phi_{0}$ of the
cubic-quintic equation for $N=2$. Consider the cubic-quintic nonlinear
Schr\"{o}dinger equation (\ref{eqn-cubic-quintic}). Denote
\[
F(s)=-\alpha_{1}+\alpha_{3}s-\alpha_{5}s^{2},\ \ \ \alpha_{1},\alpha
_{3},\alpha_{5}>0,
\]
with
\begin{equation}
\frac{\alpha_{1}\alpha_{5}}{\alpha_{3}^{2}}\in(\frac{3}{16},\frac{1}{4}).
\label{eq:1.2}%
\end{equation}
Set $\rho_{0}=\frac{\alpha_{3}+\sqrt{\alpha_{3}^{2}-4\alpha_{1}\alpha_{5}}%
}{2\alpha_{5}}$, then $F(\rho_{0})=0$ and $F^{\prime}(\rho_{0})<0$. Define
\begin{equation}
g(s)=\left\{
\begin{array}
[c]{ll}%
-F((\sqrt{\rho_{0}}-s)^{2})(\sqrt{\rho_{0}}-s),\ 0\leq s\leq\sqrt{\rho_{0}}, &
\\
0,\ s\geq\sqrt{\rho_{0}}, & \\
-g(-s),\ s\leq0, &
\end{array}
\right.  . \label{defn-g}%
\end{equation}
According to Theorem 2.1 of \cite{de95}, if (\ref{eq:1.2}) holds, the
following semilinear elliptic equation
\begin{equation}
-\Delta u=g(u),\ \ u\in H^{1}(\mathbb{R}^{N}),u\neq0, \label{eq:1.3}%
\end{equation}
has a positive radially symmetric, decreasing solution $Q\left(  \left\vert
x\right\vert \right)  \in(0, \sqrt{\rho_{0}})$, which is usually called a
\textit{ground state}. Then $\phi_{0}=\sqrt{\rho_{0}}-Q\left(  \left\vert
x\right\vert \right)  $ is a stationary bubble of (\ref{eqn-cubic-quintic})
with the the nonzero boundary condition $|\phi_{0}|\rightarrow\sqrt{\rho_{0}}$
as $|x|\rightarrow\infty$. See Theorem \ref{thm-stationary bubble} for more
properties of $\phi_{0}$.


\begin{theorem}
For $N=2,\ $let $Q\left(  \left\vert x\right\vert \right)  $ be the ground
state of (\ref{eq:1.3}) with the cubic-quintic nonlinear term $g\left(
u\right)  $ defined in (\ref{defn-g}). Consider the operator
\[
L_{0}=-\Delta-g^{\prime}(Q):H^{2}\left(  \mathbf{R}^{2}\right)  \rightarrow
L^{2}\left(  \mathbf{R}^{2}\right)  .
\]
Then
\begin{equation}
\ker L_{0}=span \left\{  \partial_{x_{1}}Q,\partial_{x_{2}}Q\right\}  .
\label{condition-NDG-ground}%
\end{equation}

\end{theorem}

First, we check some properties of the cubic-quintic nonlinear term.

\begin{lemma}
\label{lemma-appendix2-g} Let $\frac{\alpha_{1}\alpha_{5}}{\alpha_{3}^{2}}%
\in(\frac{3}{16},\frac{1}{4})$, then $g$ satisfies the following conditions
\newline(G1) $g(0)=0$ and $g^{\prime}(0)<0$, \newline(G2) there exists
$u_{0}\in(0,\sqrt{\rho_{0}})$ such that $g(u_{0})=0$, $g^{\prime}(u_{0})>0$,
$g(u)<0$ for all $0<u<u_{0}$, and $g(u)>0$ for all $u_{0}<u<\sqrt{\rho_{0}}$.
\newline Furthermore, there exists $c_{0}\in(\frac{3}{16},\frac{21}{100})$ and
$u_{1}\in(u_{0},\sqrt{\rho_{0}})$ such that if $\frac{\alpha_{1}\alpha_{5}%
}{\alpha_{3}^{2}}=c_{0}$, then $g$ satisfies \newline(G3) $g^{\prime}%
(u_{1})=0$, $g^{\prime}(u)<0$ for all $u_{1}<u<\sqrt{\rho_{0}}$, $g^{\prime
}(u)>0$ for all $u_{0}<u<u_{1}$ and $G(u_{1})=\int_{0}^{u_{1}}g(s)ds=0$,
$G(u)<0$ for all $0<u<u_{1}$, and $G(u)>0$ for all $u_{1}<u\leq\sqrt{\rho_{0}%
}$. \newline(G4) for any $\beta>0$, $\Phi_{\beta}(u)=\beta(ug^{\prime
}(u)-g(u))-2g(u)$ has exactly one zero in $(u_{0},\sqrt{\rho_{0}})$.
\end{lemma}

\begin{proof}
By the definition of $g$, we have for $u\in\lbrack0,\sqrt{\rho_{0}})$,
\[
g^{\prime}(u)=(-1)[\alpha_{1}-3\alpha_{3}(\sqrt{\rho_{0}}-u)^{2}+5\alpha
_{5}(\sqrt{\rho_{0}}-u)^{4}],
\]%
\[
g^{\prime\prime}(u)=-6\alpha_{3}(\sqrt{\rho_{0}}-u)+20\alpha_{5}(\sqrt
{\rho_{0}}-u)^{3}.
\]
Let $\rho_{1}=\frac{\alpha_{3}-\sqrt{\alpha_{3}^{2}-4\alpha_{1}\alpha_{5}}%
}{2\alpha_{5}}$, $\tilde{\rho}_{0}=\frac{3\alpha_{3}+\sqrt{9\alpha_{3}%
^{2}-20\alpha_{1}\alpha_{5}}}{10\alpha_{5}}$, $\tilde{\rho}_{1}=\frac
{3\alpha_{3}-\sqrt{9\alpha_{3}^{2}-20\alpha_{1}\alpha_{5}}}{10\alpha_{5}}$.
For all $\frac{\alpha_{1}\alpha_{5}}{\alpha_{3}^{2}}\in(\frac{3}{16},\frac
{1}{4})$, we have
\[
\tilde{\rho}_{1}<\rho_{1}<\tilde{\rho}_{0}<\rho_{0}.
\]
Choose $u_{0}=\sqrt{\rho_{0}}-\sqrt{\rho_{1}}$. Then (G1)(G2) hold.

Choose $u_{1}=\sqrt{\rho_{0}}-\sqrt{\tilde{\rho}_{1}}$. Then $u_{1}\in
(u_{0},\sqrt{\rho_{0}})$,
\[
g^{\prime}(u_{1})=(-1)[\alpha_{1}-3\alpha_{3}\tilde{\rho}_{1}+5\alpha
_{5}\tilde{\rho}_{1}^{2}]=0,
\]
$g^{\prime}(u)<0$ for all $u_{1}<u<\sqrt{\rho_{0}}$ and $g^{\prime}(u)>0$ for
all $u_{0}\leq u<u_{1}$.

Let $c=\frac{\alpha_{1}\alpha_{5}}{\alpha_{3}^{2}}$. If $c\in(\frac{3}%
{16},\frac{21}{100})$, we have $g^{\prime\prime}(u)<0$ for all $u_{0}\leq
u<\sqrt{\rho_{0}}$. Then for any $\beta>0$,
\[
\Phi_{\beta}^{\prime}(u)=\beta ug^{\prime\prime}(u)-2g^{\prime}(u)<0,\ \forall
u\in\lbrack u_{0},u_{1}].
\]
Moreover, $\Phi_{\beta}(u_{0})>0$, $\Phi_{\beta}(u_{1})<0$ and $\Phi_{\beta
}(u)<0$ for all $u_{1}<u<\sqrt{\rho_{0}}$.

By direct calculations, we have
\begin{align*}
G(u_{1})  &  =\frac{-\alpha_{3}^{3}}{2\alpha_{5}^{2}}\{\frac{3-\sqrt{9-20c}%
}{10}(\frac{14c}{15}-\frac{9-9\sqrt{9-20c}}{100})\\
&  -\frac{1+\sqrt{1-4c}}{24}[8c-(1+\sqrt{1-4c})]\}\triangleq\frac{-\alpha
_{3}^{3}}{2\alpha_{5}^{2}}h(c).
\end{align*}
Since $h(\frac{3}{16})>0$, $h(\frac{21}{100})<0$, there exists $c_{0}\in
(\frac{3}{16},\frac{21}{100})$ such that $G(u_{1})=0$ if $\frac{\alpha
_{1}\alpha_{5}}{\alpha_{3}^{2}}=c_{0}$. Then (G3)(G4) hold.
\end{proof}

Let $u(\alpha,r)$ be the solution of the initial value problem
\begin{equation}
\left\{
\begin{array}
[c]{ll}%
u^{\prime\prime}(r)+\frac{N-1}{r}u^{\prime}(r)+g(u)=0 & \\
u(0)=\alpha,\ u^{\prime}(0)=0. &
\end{array}
\right.  \label{eq:1.4}%
\end{equation}
Then $\phi(\alpha,r)=\frac{\partial u(\alpha,r)}{\partial\alpha}$ solves
\begin{equation}
\left\{
\begin{array}
[c]{ll}%
\phi^{\prime\prime}(r)+\frac{N-1}{r}\phi^{\prime}(r)+g^{\prime}(u(\alpha
,r))\phi=0 & \\
\phi(0)=1,\ \phi^{\prime}(0)=0. &
\end{array}
\right.  \label{eq:1.5}%
\end{equation}
Let $Q\left(  \left\vert x\right\vert \right)  =u(\alpha_{0},\left\vert
x\right\vert )$ be a ground state of (\ref{eq:1.3}). To show the
non-degeneracy condition
\[
\ker L_{0}=\left\{  \partial_{x_{1}}Q,\cdots,\partial_{x_{N}}Q\right\}  ,
\]
it suffices to show that the function $\phi(\alpha_{0},r)$ does not vanish at
infinity. (See \cite{weinstein85} or \cite{maris-4d-slow} for the proof). When
$N=2,\ $such a result is provided in the following lemma, which was motivated
by \cite{Lin} and \cite{Jang}.

\begin{lemma}
\label{lemma-1-appendix2} Suppose that (G1)-(G4) hold and let $u(\alpha
_{0},r)$ be a ground state of (\ref{eq:1.3}), then $\lim\limits_{r\rightarrow
+\infty}\phi(\alpha_{0},r)\neq0$ when $N=2.$
\end{lemma}

\begin{proof}
To simplify notations, we denote $u(\alpha_{0},r),\phi(\alpha_{0},r)$ by
$u(r),\phi(r)$ respectively. Since $u(r)$ is a ground state of (\ref{eq:1.3}),
then $u(r)>0$, $u^{\prime}(r)<0$ for all $r>0$ and $u_{1}<u(0)<\sqrt{\rho_{0}%
}$. Moreover, by (G1) it follows from Lemma 6 of \cite{Kwong} that $\phi$
becomes monotone for large $r$. Therefore $\lim\limits_{r\rightarrow+\infty
}\phi(r)$ exists. In order to prove this lemma, we suppose to the contrary
that $\lim\limits_{r\rightarrow+\infty}\phi(r)=0$. \newline

Claim 1. $\phi$ has exactly one zero in $(0,+\infty)$. \newline Let
$A_{0}=-\partial_{r}^{2}-\frac{N-1}{r}\partial_{r}-g^{\prime}(u(r))$. From
$A_{0}u^{\prime}=-\left(  N-1\right)  r^{-2}u^{\prime}$, we deduce that the
first eigenvalue of $A_{0}$ is negative. By Proposition B.1 of \cite{Fr}, the
second eigenvalue of $A_{0}$ is nonnegative. Since $A_{0}\phi=0$ and
$\lim\limits_{r\rightarrow+\infty}\phi(r)=0$, $0$ must be the second
eigenvalue of $A_{0}$. Thus, $\phi$ has exactly one zero $z_{1}\in$
$(0,+\infty)$.

Claim 2. Let $r_{0}\in(0,+\infty)$ be such that $u(r_{0})=u_{0}$, then
$0<z_{1}<r_{0}$. Here, $u_{0}$ is defined in (G2). \newline For $\beta\geq0$,
let $v_{\beta}(r)=ru^{\prime}(r)+\beta u(r)$. Then $v_{\beta}$ solves
\begin{equation}
v_{\beta}^{\prime\prime}(r)+\frac{N-1}{r}v_{\beta}^{\prime}(r)+g^{\prime
}(u)v_{\beta}=\Phi_{\beta}(u), \label{eq:1.6}%
\end{equation}
where $\Phi_{\beta}(u)=\beta(ug^{\prime}(u)-g(u))-2g(u)$. By (\ref{eq:1.6})
and Green's Theorem, for any $0\leq r_{1}<r_{2}$, we have
\begin{align}
\int_{r_{1}}^{r_{2}}r^{N-1}\Phi_{\beta}(u)\phi dr  &  =r_{2}^{N-1}[\phi
(r_{2})v_{\beta}^{\prime}(r_{2})-v_{\beta}(r_{2})\phi^{\prime}(r_{2}%
)]\label{eq:1.7}\\
&  -r_{1}^{N-1}[\phi(r_{1})v_{\beta}^{\prime}(r_{1})-v_{\beta}(r_{1}%
)\phi^{\prime}(r_{1})].\nonumber
\end{align}
Set $H(\beta)=\phi(r_{0})v_{\beta}^{\prime}(r_{0})-v_{\beta}(r_{0}%
)\phi^{\prime}(r_{0})$. By the proof of lemma 2.8 in \cite{Jang}, we deduce
that $H(0)>0$. Then from (\ref{eq:1.7}) we get
\begin{equation}
\int_{0}^{r_{0}}r^{N-1}\Phi_{0}(u)\phi dr=r_{0}^{N-1}H(0)>0. \label{eq:1.7a}%
\end{equation}
By (G2) we know that $\Phi_{0}(u)=-2g(u)<0$ for all $u_{0}<u<\sqrt{\rho_{0}}$.
If $\phi>0$ on $[0,r_{0})$, it is impossible by (\ref{eq:1.7a}). Thus we must
have $\phi(r_{0})<0$ and $0<z_{1}<r_{0}$.

Claim 3. $\theta(r)=\frac{-ru^{\prime}(r)}{u(r)}$ is increasing in $(0,r_{0}%
)$. \newline For the proof of this claim, we need $N=2$. In fact, by
(\ref{eq:1.4}) we have $(-ru^{\prime}(r))^{\prime}=rg(u(r))$ for $N=2$. Thus,
from (G2) we know that $(-ru^{\prime}(r))^{\prime}>0$ in $(0,r_{0})$. Since
$u(r)$ is decreasing in $(0,+\infty)$, we get that $\theta(r)=\frac
{-ru^{\prime}(r)}{u(r)}$ is increasing in $(0,r_{0})$. \newline Set $\beta
_{0}=\frac{-z_{1}u^{\prime}(z_{1})}{u(z_{1})}$, then $\beta_{0}>0$ and
$v_{\beta_{0}}(z_{1})=0$. From (\ref{eq:1.7}), we get
\begin{equation}
\int_{0}^{z_{1}}r^{N-1}\Phi_{\beta_{0}}(u)\phi dr=0. \label{eq:1.8}%
\end{equation}
By (G2)(G3), we have $\Phi_{\beta_{0}}(u)<0$ for all $u_{1}\leq u<\sqrt
{\rho_{0}}$. Note that $\phi>0$ on $[0,z_{1})$, $u^{\prime}(r)<0$ for all
$r>0$ and $u_{1}<u(0)<\sqrt{\rho_{0}}$, then from (\ref{eq:1.8}) we deduce
that $u(z_{1})<u_{1}$ and $\Phi_{\beta_{0}}(u(z_{1}))>0$. Furthermore, by (G4)
we have $\Phi_{\beta_{0}}(u(r))>0$ for all $r\in(z_{1},r_{0})$. Since $\phi<0$
on $(z_{1},r_{0}]$, we have
\begin{equation}
\int_{z_{1}}^{r_{0}}r^{N-1}\Phi_{\beta_{0}}(u)\phi dr<0. \label{eq:1.9}%
\end{equation}
On the other hand, from (\ref{eq:1.7}) we get
\begin{equation}
\int_{z_{1}}^{r_{0}}r^{N-1}\Phi_{\beta_{0}}(u)\phi dr=r_{0}^{N-1}[\phi
(r_{0})v_{\beta_{0}}^{\prime}(r_{0})-v_{\beta_{0}}(r_{0})\phi^{\prime}%
(r_{0})]. \label{eq:1.10}%
\end{equation}

Claim 4. $\phi(r_{0})v_{\beta_{0}}^{\prime}(r_{0})-v_{\beta_{0}}(r_{0}%
)\phi^{\prime}(r_{0})>0$. \newline By Claim 4 and (\ref{eq:1.9})
(\ref{eq:1.10}), we get a contradiction. \newline\textit{proof of Claim 4.}
Let $H(\beta)=\phi(r_{0})v_{\beta}^{\prime}(r_{0})-v_{\beta}(r_{0}%
)\phi^{\prime}(r_{0})$, then
\begin{equation}
H(\beta)=H(0)+\beta\lbrack\phi(r_{0})u^{\prime}(r_{0})-\phi^{\prime}%
(r_{0})u(r_{0})]. \label{eq:1.11}%
\end{equation}
We show $H(\beta_{0})>0$ in two cases. \newline Case 1. $\phi(r_{0})u^{\prime
}(r_{0})-\phi^{\prime}(r_{0})u(r_{0})\geq0$. In this case, by $H(0)>0$ and
(\ref{eq:1.11}) we obviously have $H(\beta_{0})>0$. \newline Case 2.
$\phi(r_{0})u^{\prime}(r_{0})-\phi^{\prime}(r_{0})u(r_{0})<0$. Since
$\phi(r_{0})<0,u^{\prime}(r_{0})<0,u(r_{0})>0$, we must have $\phi^{\prime
}(r_{0})>0$.

Let $b_{1}=\frac{-r_{0}u^{\prime}(r_{0})}{u(r_{0})}$. Since $z_{1}<r_{0}$ and
$\theta(r)=\frac{-ru^{\prime}(r)}{u(r)}$ is increasing in $(0,r_{0})$, we have
$\beta_{0}<b_{1}$. Then by $\phi(r_{0})u^{\prime}(r_{0})-\phi^{\prime}%
(r_{0})u(r_{0})<0$ and (\ref{eq:1.11}) we get $H(\beta_{0})>H(b_{1})$. Note
that $v_{N-2}^{\prime}(r_{0})=-r_{0}g(u_{0})=0$ by (G2) and $v_{b_{1}}%
(r_{0})=0$. If $b_{1}\geq N-2$, we have $v_{b_{1}}^{\prime}(r_{0})\leq
v_{N-2}^{\prime}(r_{0})=0$ and
\[
H(b_{1})=\phi(r_{0})v_{b_{1}}^{\prime}(r_{0})-v_{b_{1}}(r_{0})\phi^{\prime
}(r_{0})=\phi(r_{0})v_{b_{1}}^{\prime}(r_{0})\geq0.
\]
This finishes the proof of the lemma for $N=2$.
\end{proof}

\subsection*{Appendix 3}

Consider a function in the form of $U(x_{1},x_{\bot})=U(x_{1},r_{\bot})$,
where $x_{\bot}=(x_{2},x_{3})$ and $r_{\bot}=|x_{\bot}|$, and assume $\nabla
U\in H^{s}(\mathbf{R}^{3})$, $s>1$, not necessarily an integer. In this
appendix, we prove
\[
\frac{1}{r_{\bot}}\partial_{r_{\bot}}U\in L^{2}(\mathbf{R}^{3})\;\text{ and
}\;\partial_{r_{\bot}}U\in H^{1}(\mathbf{R}^{3})
\]
which are needed in Lemma \ref{le-4} to show that the Hessian $L_{c}$ of the
energy functional has a negative mode.

Due to the density of Schwartz class functions, we will work on Schwartz class
functions, but keep tracking of the norms carefully. Denote $\vec{\mathbf{e}%
}_{\perp}=\frac{1}{r_{\bot}}\left(  0,x_{\bot}\right)  $, then%

\[
\partial_{r_{\bot}}U(x_{1},r_{\bot})=\nabla_{x_{\bot}}U(x_{1},x_{\bot}%
)\cdot\frac{x_{\bot}}{r_{\bot}}=DU(x_{1},x_{\bot})\cdot\vec{\mathbf{e}}%
_{\perp}%
\]
and%
\[
\partial_{x_{1}}\partial_{r_{\bot}}U(x_{1},r_{\bot})=D^{2}U(x_{1},x_{\bot
})\left(  \vec{\mathbf{e}}_{1},\vec{\mathbf{e}}_{\perp}\right)  .
\]
Moreover, since $\partial_{r_{\bot}}U(x_{1},x_{\bot})$ is radial in $x_{\bot}%
$, its gradient in $x_{\bot}$ must be in the radial direction and thus
\[
\nabla_{x_{\bot}}\partial_{r_{\bot}}U(x_{1},x_{\bot})=D^{2}U(x_{1},x_{\bot
})\left(  \vec{\mathbf{e}}_{\perp},\vec{\mathbf{e}}_{\perp}\right)
\frac{x_{\perp}}{r_{\bot}}.
\]
Therefore $\partial_{r_{\bot}}U\in H^{1}(\mathbf{R}^{3})$ is obvious.
Computing higher order derivatives in a similar fashion and applying an
interpolation argument if $s$ is not an integer, one can prove $\partial
_{r_{\bot}}U\in H^{s}(\mathbf{R}^{3})$.

To show $\frac{1}{r_{\bot}}\partial_{r_{\bot}}U\in L^{2}(\mathbf{R}^{3})$, we
first observe that the radial symmetry of $U$ implies its linearization at
$x_{\bot}=0$ is also a radially symmetric linear function, which can only be
$0$, and thus
\[
\nabla_{x_{\bot}}U(x_{1},0)=0\implies\partial_{r_{\bot}}U(x_{1},0)=0.
\]
Fix $x_{1}$, on the one hand, one may estimate by using the Cauchy-Schwarz
inequality
\[%
\begin{split}
&  \big(\partial_{r_{\bot}}U(x_{1},r_{\bot})\big)^{2}=2\int_{0}^{r_{\bot}%
}\partial_{r_{\bot}}U(x_{1},r_{\bot}^{\prime})\partial_{r_{\bot}r_{\bot}%
}U(x_{1},r_{\bot}^{\prime})dr_{\bot}^{\prime}\\
\leq &  Cr_{\bot}^{\frac{p-2}{p}}|\nabla U(x_{1},\cdot)|_{L^{\infty
}(\mathbf{R}^{2})}\left(  \int_{0}^{r_{\bot}}r_{\bot}^{\prime}\left\vert
\partial_{r_{\bot}r_{\bot}}U\left(  x_{1},r_{\bot}^{\prime}\right)
\right\vert ^{p}dr_{\bot}^{\prime}\right)  ^{\frac{1}{p}}\\
\leq &  Cr_{\bot}^{\frac{p-2}{p}}|\nabla U(x_{1},\cdot)|_{H^{s}(\mathbf{R}%
^{2})}|D^{2}U(x_{1},\cdot)|_{L^{p}(\{|x_{\bot}|<r_{\bot}\})}\leq Cr_{\bot
}^{\frac{p-2}{p}}|\nabla U(x_{1},\cdot)|_{H^{s}(\mathbf{R}^{2})}^{2}%
\end{split}
\]
for some $p>2$. Integrating in $x_{1}$
we obtain
\[
\int_{\mathbf{R}}\big(\partial_{r_{\bot}}U(x_{1},r_{\bot})\big)^{2}dx_{1}\leq
Cr_{\bot}^{1-\frac{2}{p}}|\nabla U|_{H^{s}(\mathbf{R}^{3})}^{2}.
\]
On the other hand, via integration by parts, we have%

\[%
\begin{split}
&  \int_{\tilde{r}_{\bot}}^{1}\frac{1}{r_{\bot}^{\prime}}\big(\partial
_{r_{\bot}}U(x_{1},r_{\bot}^{\prime})\big)^{2}dr_{\bot}^{\prime}%
=-2\int_{\tilde{r}_{\bot}}^{1}(\log r_{\bot}^{\prime})\partial_{r_{\bot}%
}U(x_{1},r_{\bot}^{\prime})\partial_{r_{\bot}r_{\bot}}U(x_{1},r_{\bot}%
^{\prime})dr_{\bot}^{\prime}\\
&  \hspace{3.1in}-(\log\tilde{r}_{\bot})\big(\partial_{r_{\bot}}U(x_{1}%
,\tilde{r}_{\bot})\big)^{2}.
\end{split}
\]
Integrating it with respect to $x_{1}$, letting $\tilde{r}_{\bot}%
\rightarrow0+$, and using the above inequality, we obtain
\[
|\frac{1}{r_{\bot}}\partial_{r_{\bot}}U|_{L^{2}(\{|x_{\bot}|<1\})}^{2}%
=-2\int_{|x_{\bot}|<1}(\log r_{\bot})\partial_{r_{\bot}}U(x_{1},r_{\bot
})\partial_{r_{\bot}r_{\bot}}U(x_{1},r_{\bot})dx
\]
Splitting the integrand on the right side into the product of $r_{\bot
}^{-\frac{1}{2}}\partial_{r_{\bot}}U$, $r_{\bot}^{\frac{1}{p}}\partial
_{r_{\bot}r_{\bot}}U$, and $r_{\bot}^{\frac{1}{2}-\frac{1}{p}}\log r_{\bot}$
and applying the H\"{o}lder inequality first with indices $\frac{1}{2}$,
$\frac{1}{p}$, and $\frac{1}{2}-\frac{1}{p}$ to the integral in $x_{\bot}$ and
then the Cauchy-Schwarz inequality to the $x_{1}$ integral, we have
\[%
\begin{split}
&  |\frac{1}{r_{\bot}}\partial_{r_{\bot}}U|_{L^{2}(\{|x_{\bot}|<1\})}^{2}\\
\leq &  C\int_{\mathbf{R}}|\frac{1}{r_{\bot}}\partial_{r_{\bot}}U(x_{1}%
,\cdot)|_{L^{2}(\{|x_{\bot}|<1\})}|D^{2}U(x_{1},\cdot)|_{L^{p}(\{|x_{\bot
}|<1\})}dx_{1}\\
\leq &  C|\frac{1}{r_{\bot}}\partial_{r_{\bot}}U|_{L^{2}(\{|x_{\bot}%
|<1\})}|\nabla U|_{H^{s}(\mathbf{R}^{3})}%
\end{split}
\]
Therefore we obtain
\[
|\frac{1}{r_{\bot}}\partial_{r_{\bot}}U|_{L^{2}(\{|x_{\bot}|<1\})}\leq
C|\nabla U|_{H^{s}(\mathbf{R}^{3})}.
\]
As the estimate is trivially true on $\{|x_{\bot}|>1\}$,
the proof is complete.

\begin{center}
{\Large Acknowledgement}
\end{center}

This work is supported partly by a US NSF grant DMS-1411803, a Simons
Fellowship during 2013-2014 (Lin), a Chinese NSF grant NSFC-11271360 (Wang),
and a US NSF grant DMS-1362507 (Zeng).


\begin{thebibliography}{99}                                                                                               %


\bibitem {gp-physical03}Abid, M.; Huepe, C.; Metens, S.; Nore, C.; Pham, C.
T.; Tuckerman, L. S. and Brachet, M. E., \textit{Gross-Pitaevskii dynamics of
Bose-Einstein condensates and superfluid turbulence}, Fluid Dynamics Research
\textbf{33}, 5-6 (2003), 509-544.

\bibitem {alves-et-2012}Alves, Claudianor O.; Souto, Marco A. S.; Montenegro,
Marcelo \textit{Existence of a ground state solution for a nonlinear scalar
field equation with critical growth}. Calc. Var. Partial Differential
Equations \textbf{43} (2012), no. 3-4, 537--554.

\bibitem {ambrosetti-primer-NA}Ambrosetti, Antonio; Prodi, Giovanni, \textit{A
primer of nonlinear analysis}. Cambridge Studies in Advanced Mathematics, 34.
Cambridge University Press, Cambridge, 1993.

\bibitem {ba93}Barashenkov, I.~V. and Panova, E.~Y., \textit{Stability and
evolution of the quiesent and traveling solitonic bubbles}, Physica D
\textbf{69 }(1993), 114-134.

\bibitem {ba89}Barashenkov, I.~V.; Gocheva, A.~D.; Makhankov, V.~G. and
I.~V.~Puzynin, \textit{Stability of the soliton-like bubbles}, Physica D
\textbf{34 }(1989), 240-254.

\bibitem {lett89}Barashenkov, I.~V.; Boyadjiev, T.~L.; Puzynin, I.~V.~and
Zhanlav, T., \textit{Stability of moving bubbles in a system of interacting
bosons}, Phys.Lett.A. \textbf{135 }(1989), 125-128.

\bibitem {bates-jones88}Bates, P. and Jones, C., \textit{Invariant manifolds
for semilinear partial differential equations}, Dynamics Reported, \textbf{2}
(1989), 1--38.

\bibitem {benzoni-instability}Benzoni-Gavage, S., \textit{Spectral transverse
instability of solitary waves in Korteweg fluids.} J. Math. Anal. Appl.
\textbf{361} (2010), 338-357.

\bibitem {berestychi-et-83-2d}Berestycki, H., Gallou\"{e}t, T., Kavian, O.:
\textit{Equations de Champs scalaires euclidiens non lin\'{e}aires dans le
plan}. C. R. Acad. Sci. Paris Ser. I Math. \textbf{297} (1983), 307--310.

\bibitem {berloff-roberts-review01}Berloff, N. G. and Roberts, P. H.,
\textit{Nonlinear Schrodinger equation as a model of superfluid helium},\ in
"Quantized Vortex Dynamics and Superfluid Turbulence" edited by C.F. Barenghi,
R.J. Donnelly and W.F. Vinen, Lecture Notes in Physics, volume 571,
Springer-Verlag, 2001.

\bibitem {berloff-roberts-crow}Berloff, N. G. and Roberts, P. H.,
\textit{Motions in a Bose condensate IX. Crow instability of antiparallel
vortex pairs},\ Journal of Physics A: Mathematics and General, \textbf{34},
(2001), 10057-10066.

\bibitem {berloff-roberts-X-stability}Berloff, N. G. and Roberts, P. H.,
\textit{Motions in a Bose condensate: X. New results on stability of
axisymmetric solitary waves of the Gross-Pitaevskii equation}, Journal of
Physics A: Mathematics and General, 37, (2004), 11333.

\bibitem {b-s-1998}B\'{e}thuel, Fabrice and Saut, Jean-Claude,
\textit{Traveling waves for the Gross-Pitaevskii equation I}, Ann. Inst. Henri
Poincar\'{e}, Physique Th\'{e}orique, \textbf{70}(2), (1999), 147--238.

\bibitem {B-G-S}B\'{e}thuel, Fabrice; Gravejat, Philippe; Saut, Jean-Claude,
\textit{Existence and properties of traveling waves for the Gross-Pitaevskii
equation}, \textit{Contamporary Mathematics}, \textbf{473} (2008), 55-103.

\bibitem {betheul-et-minimizer-cmp}B\'{e}thuel, Fabrice; Gravejat, Philippe;
Saut, Jean-Claude \textit{Traveling waves for the Gross-Pitaevskii equation.
II}, Comm. Math. Phys. \textbf{285} (2009), no. 2, 567--651.

\bibitem {betheu-et-stability-1d}B\'{e}thuel, Fabrice; Gravejat, Philippe;
Saut, Jean-Claude; Smets, Didier \textit{Orbital stability of the black
soliton for the Gross-Pitaevskii equation}. Indiana Univ. Math. J. \textbf{57}
(2008), no. 6, 2611--2642.

\bibitem {bona-et-87}Bona, J. L.; Souganidis, P. E.; Strauss, W. A.
\textit{Stability and instability of solitary waves of Korteweg-de Vries
type}. Proc. Roy. Soc. London Ser. A \textbf{411} (1987), no. 1841, 395--412.

\bibitem {brezis-book}Brezis, Haim \textit{Functional analysis, Sobolev spaces
and partial differential equations}. Universitext. Springer, New York, 2011.

\bibitem {Lin}Chen, Chiun-Chuan and Lin, Chang-Shou, \textit{Uniqueness of the
ground state solutions of }$-\Delta u+f(u)=0,\ in\ R^{n},n\geq3$\textit{,}
\textit{Commun. in PDE.}, \textbf{16} (1991), 1549-1572.

\bibitem {chiron-1d}Chiron, David, \textit{Stability and instability for
subsonic traveling waves of the nonlinear Schr\"{o}dinger equation in
dimension one}. Anal. PDE \textbf{6} (2013), no. 6, 1327--1420.

\bibitem {chiron-maris-kp}Chiron, D.; Mari\c{s}, M.; \textit{Rarefaction
Pulses for the Nonlinear Schr\"{o}dinger Equation in the Transonic Limit.}
Comm. Math. Phys. \textbf{326} (2014), no. 2, 329--392.

\bibitem {chow-hale}Chow, Shui Nee; Hale, Jack K. \textit{Methods of
bifurcation theory}, Springer-Verlag, New York-Berlin, 1982.

\bibitem {CL88}Chow, S-N and Lu K., \textit{Invariant manifolds for flows in
Banach spaces}, J. Differential Equations, \textbf{74} (1988), 285--317.

\bibitem {crow-instability}Crow, S. C., \textit{Stability theory for a pair of
trailing vortices}, AIAA Journal, \textbf{8} (1970), 2172-2179.

\bibitem {de95}de Bouard, Anne, \textit{Instability of stationary bubbles},
SIAM. J. Math. Anal., \textbf{26 }(1995), 566-582.

\bibitem {fw86}Floer, Andreas; Weinstein, Alan \textit{Nonspreading wave
packets for the cubic Schr\"{o}dinger equation with a bounded potential}. J.
Funct. Anal. \textbf{69} (1986), no. 3, 397--408.

\bibitem {Fr}Fr\"{o}hlich, J.; Gustafson, S.; Jonsson, B. L. G.; Sigal, I. M.
\textit{Solitary wave dynamics in an external potential}. Comm. Math. Phys.
\textbf{250} (2004), no. 3, 613--642.

\bibitem {gallo-2008}Gallo, Cl\'{e}ment, \textit{The Cauchy problem for
defocusing nonlinear Schr\"{o}dinger equations with non-vanishing initial data
at infinity}. Comm. Partial Differential Equations \textbf{33} (2008), no.
4-6, 729--771.

\bibitem {getsezy-latushkin-et-dichotomy}Gesztesy, F.; Jones, C. K. R. T.;
Latushkin, Y.; Stanislavova, M. \textit{A spectral mapping theorem and
invariant manifolds for nonlinear Schr\"{o}dinger equations}. Indiana Univ.
Math. J. 49 (2000), no. 1, 221--243.

\bibitem {gerard-cauchy}G\'{e}rard, P. \textit{The Cauchy problem for the
Gross-Pitaevskii equation}. Ann. Inst. H. Poincar\'{e} Anal. Non Lin\'{e}aire
\textbf{23} (2006), no. 5, 765--779.

\bibitem {Gerard1-energy}G\'{e}rard, Patrick \textit{The Gross-Pitaevskii
equation in the energy space}. Stationary and time dependent Gross-Pitaevskii
equations, 129--148, Contemp. Math., \textbf{473}, Amer. Math. Soc.,
Providence, RI, 2008.

\bibitem {gerard-zhang}G\'{e}rard, Patrick; Zhang, Zhifei \textit{Orbital
stability of traveling waves for the one-dimensional Gross-Pitaevskii
equation}. J. Math. Pures Appl. (9) \textbf{91} (2009), no. 2, 178--210.

\bibitem {Gravejat}Gravejat, Philippe \textit{Decay for travelling waves in
the Gross-Pitaevskii equation}. Ann. Inst. H. Poincar\'{e} Anal. Non
Lin\'{e}aire \textbf{21} (2004), no. 5, 591--637.

\bibitem {gss87}Grillakis, Manoussos; Shatah, Jalal; Strauss, Walter,
\textit{Stability theory of solitary waves in the presence of symmetry. I.,
}J. Funct. Anal. \textbf{74} (1987), no. 1, 160--197.

\bibitem {gss90}Grillakis, Manoussos; Shatah, Jalal; Strauss, Walter,
\textit{Stability theory of solitary waves in the presence of symmetry. II,
}J. Funct. Anal. \textbf{94} (1990), no. 2, 308--348.

\bibitem {grieg-ohta12}Georgiev, Vladimir; Ohta, Masahito \textit{Nonlinear
instability of linearly unstable standing waves for nonlinear Schr\"{o}dinger
equations}. J. Math. Soc. Japan 64 (2012), no. 2, 533--548.

\bibitem {GNT07}Gustafson, Stephen; Nakanishi, Kenji; Tsai, Tai-Peng,
\textit{Global dispersive solutions for the Gross-Pitaevskii equation in two
and three dimensions}. Ann. Henri Poincar\'e 8 (2007), no. 7, 1303--1331.

\bibitem {GNT09}Gustafson, Stephen; Nakanishi, Kenji; Tsai, Tai-Peng,
\textit{Scattering theory for the Gross-Pitaevskii equation in three
dimensions}. Commun. Contemp. Math. 11 (2009), no. 4, 657--707.

\bibitem {henry-et-83}Henry, Daniel B.; Perez, J. Fernando; Wreszinski, Walter
F. \textit{Stability theory for solitary-wave solutions of scalar field
equations}. Comm. Math. Phys. 85 (1982), no. 3, 351--361.

\bibitem {Jang}Jang, Jaeduck, \textit{Uniqueness of positive solutions to
semilinear elliptic partial differential equations}, Proceedings of the
National Institute for Mahthematical Sciences , vol. 1, No.3 (2006), pp. 23-27.

\bibitem {jones-et-stability}Jones, C. A.; Putterman, S. J. and Roberts, P.
H., \textit{Motions in a Bose condensate V. Stability of solitary wave
solutions of nonlinear Schr\"{o}dinger equations in two and three dimensions}.
J. Phys. A, Math. Gen., \textbf{19} (1986), 2991--3011.

\bibitem {jones-et-symmetry}Jones, C. A. and Roberts, P. H., \textit{Motions
in a Bose condensate IV. Axisymmetric solitary waves}. J. Phys. A, Math. Gen.,
\textbf{15} (1982), 2599--2619.

\bibitem {killip-et}Killip, Rowan; Oh, Tadahiro; Pocovnicu, Oana; Vi\c{s}an,
Monica \textit{Global well-posedness of the Gross-Pitaevskii and cubic-quintic
nonlinear Schr\"{o}dinger equations with non-vanishing boundary conditions}.
Math. Res. Lett. \textbf{19} (2012), no. 5, 969--986.

\bibitem {Kwong}Kwong, Man Kam, \textit{Uniqueness of positive solutions of
}$\Delta u-u+u^{p}=0\ in\ R^{n}$, \textit{Arch. Rational Mech. Anal.}
\textbf{105} (1989), 243-266.

\bibitem {kuznetsov-rasmussen95}Kuznetsov, E. A. and Rasmussen, J. J.,
\textit{Instability of two-dimensional solitons and vortices in defocusing
media}, Phys. Rev. E \textbf{51} (5), (1995), 4479-4484.

\bibitem {lin-bubble-1d}Lin, Zhiwu, \textit{Stability and instability of
traveling solitonic bubbles}, Adv. Differential Equations, \textbf{7} (2002), 897-918.

\bibitem {lin-note-99}Lin, Zhiwu, \textit{Slow traveling bubbles in two and
three dimension}, unpulished manuscript, 1999.

\bibitem {LZ15}Lin, Zhiwu and Zeng, Chongchun, \textit{Instability, index
theorem, and exponential trichotomy for Linear Hamiltonian PDEs}, preprint.

\bibitem {lopes}Lopes, Orlando, \textit{A linearized instability result for
solitary waves}. Discrete Contin. Dyn. Syst. \textbf{8} (2002), no. 1, 115--119.

\bibitem {maris-4d-slow}Maris, Mihai \textit{Existence of nonstationary
bubbles in higher dimensions}. J. Math. Pures Appl. (9) 81 (2002), no. 12, 1207--1239.

\bibitem {Maris-annal}Maris, Mihai, \textit{Traveling waves for nonlinear
Schr\"{o}dinger equations with nonzero conditions at infinity}, Ann. Math.,
\textbf{178} (2013), 107-182.

\bibitem {Maris-non-existence}Maris, Mihai, \textit{Nonexistence of supersonic
traveling waves for nonlinear Schr\"{o}dinger equations with nonzero
conditions at infinity}, SIAM J. Math. Anal\textit{.}, \textbf{40} (2008), 1076-1103.

\bibitem {chiron-marisII-12}Maris, M. and Chiron, D. \textit{Traveling waves
for nonlinear Schr\"{o}dinger equations with nonzero conditions at infinity},
II, arXiv:1203.1912.

\bibitem {mitsumachi-instability}Mizumachi, Tetsu, \textit{A remark on
linearly unstable standing wave solutions to NLS}. Nonlinear Anal. \textbf{64}
(2006), no. 4, 657--676.

\bibitem {rousett-et-transversal}Rousset, Frederic; Tzvetkov, Nikolay
\textit{A simple criterion of transverse linear instability for solitary
waves}. Math. Res. Lett. \textbf{17} (2010), no. 1, 157--169.

\bibitem {shatah-strauss00}Shatah, J. and Strauss, W., \textit{Spectral
condition for instability}, Contemp. Math. \textbf{255} (2000), 189--198.

\bibitem {shizuta-83}Shizuta, Yasushi, \textit{On the classical solutions of
the Boltzmann equation}. Comm. Pure Appl. Math. \textbf{36} (1983), no. 6, 705--754.

\bibitem {vidav}Vidav, Ivan \textit{Spectra of perturbed semigroups with
applications to transport theory}. J. Math. Anal. Appl. \textbf{30}, (1970), 264--279.

\bibitem {weinstein85}Weinstein, Michael I. \textit{Modulational stability of
ground states of nonlinear Schr\"{o}dinger equations}. SIAM J. Math. Anal.
\textbf{16} (1985), no. 3, 472--491.
\end{thebibliography}
\end{document}